\documentclass[11pt]{amsart}
\usepackage[utf8]{inputenc}

\title{Invariant Gibbs measures for the three-dimensional wave equation with a  Hartree nonlinearity I: Measures}
\author{Bjoern Bringmann}
\date{\today}

\usepackage{graphics}
\usepackage{amsmath, amssymb, amsfonts, amsthm, verbatim, color}
\usepackage{fullpage}
\usepackage{mathabx}
\usepackage{hyperref}
\usepackage{xcolor}
\usepackage[numbers]{natbib}
\usepackage{mathtools}
\usepackage{enumerate}

\usepackage[scr=boondoxo]{mathalfa}
\newcommand{\cg}{\mathscr{g}}
\newcommand{\cgp}{\mathscr{g}^{\scalebox{0.7}{$\otimes$}}}

\newtheorem{theorem}{Theorem}%[section]
\newtheorem{proposition}[theorem]{Proposition} 
\newtheorem{corollary}[theorem]{Corollary}
\newtheorem{lemma}[theorem]{Lemma}
\newtheorem{definition}[theorem]{Definition}
\newtheorem{remark}[theorem]{Remark}
\numberwithin{equation}{section} %Equation numbering
\numberwithin{theorem}{section} %theorem numbering
\newtheorem*{maintheorem}{Main theorem}%[section]
\newtheorem{innerassumptions}{Assumptions}
\newenvironment{assumptions}[1]
  {\renewcommand\theinnerassumptions{#1}\innerassumptions}
  {\endinnerassumptions}

\newcommand{\cZ}{\mathcal{Z}}
\newcommand{\cZTl}{\mathcal{Z}^{\scriptscriptstyle{T},\lambda}}
\newcommand{\Z}{\mathbb{Z}}
\newcommand{\R}{\mathbb{R}}
\newcommand{\bR}{\mathbb{R}}
\newcommand{\Q}{\mathbb{Q}}

\newcommand{\E}{\mathbb{E}}
\newcommand{\T}{\mathbb{T}}
\newcommand{\bP}{\mathbb{P}}
\newcommand{\cF}{\mathcal{F}}
\newcommand{\calH}{\mathcal{H}} 					%%Hilbert space, multiple stochastic integral appendix
\newcommand{\calI}{\mathcal{I}}
\newcommand{\cG}{\mathcal{G}} 			                %%Gaussian chaos, stochastic object section 
\newcommand{\ctG}{\widetilde{\mathcal{G}}} 			%%Gaussian chaos, stochastic object section 
\newcommand{\cR}{\mathcal{R}}

\newcommand{\cH}{{L^2_{t,x}}}
\newcommand{\bH}{\mathbb{H}_a}
\newcommand{\cC}{\mathcal{C}}

\DeclareMathOperator{\Aux}{Aux}
\DeclareMathOperator{\Law}{Law}
\newcommand{\Jt}{J_t^{\scriptscriptstyle{T}}}
\newcommand{\Js}{J_s^{\scriptscriptstyle{T}}}
\newcommand{\JS}{J^{\scriptscriptstyle{S}}}
\newcommand{\JSm}{J^{\scriptscriptstyle{S_m}}}
\newcommand{\It}{I_t^{\scriptscriptstyle{T}}}
\newcommand{\Is}{I_s^{\scriptscriptstyle{T}}}
\newcommand{\IS}{I^{\scriptscriptstyle{S}}}
\newcommand{\IT}{I^{\scriptscriptstyle{T}}}
\newcommand{\IP}{I^{\scriptscriptstyle{T}}}
\newcommand{\Iinf}{I_\infty^{\scriptscriptstyle{T}}}
\newcommand{\XiT}{\Xi^{\scriptscriptstyle{T}}}
\newcommand{\OpT}{\operatorname{Op}^T_t(\gamma,r)}
\newcommand{\OpTtwo}{\operatorname{Op}^T_t(\gamma,2)}
\newcommand{\OpTN}{\operatorname{Op}^T_t(r;K_1,K_2,N_1,N_2)}

\newcommand{\aT}{a^{\scriptscriptstyle{T}}_t}
\newcommand{\bT}{b^{\scriptscriptstyle{T}}_t}
\newcommand{\cT}{c^{\scriptscriptstyle{T},\lambda}}
\newcommand{\MT}{\mathcal{M}^{\scriptscriptstyle{T}}_t}
\newcommand{\MN}{\mathcal{M}_N}
\newcommand{\Ms}{\mathcal{M}_s^{\scriptscriptstyle{T}}}
\newcommand{\MS}{\mathcal{M}^{\scriptscriptstyle{S}}}
\newcommand{\Minf}{\mathcal{M}^{\scriptscriptstyle{T}}_\infty}
\newcommand{\Cort}{\mathfrak{C}^T_t[N_1,N_2]}
							
\newcommand{\lcol}{:\hspace{-0.5ex}}
\newcommand{\rcol}{\hspace{-0.5ex}:}

\newcommand{\defe}{\overset{\text{def}}{=}} %Definition symbol which does not clash with Wick-ordering

\newcommand{\cV}{\mathcal{V}}                       
\newcommand{\cVT}{\mathcal{V}^{\scriptscriptstyle{T},\lambda}}

\newcommand{\tmu}{\widetilde{\mu}}
\newcommand{\tmuT}{\widetilde{\mu}_{T}}

\newcommand{\bmu}{\widebar{\mu}}
\newcommand{\muT}{\mu_T}

\newcommand{\Qinf}{\mathbb{Q}^u_{\infty}}
\newcommand{\QT}{\mathbb{Q}^{u}_{T}}

\newcommand{\QTN}{\mathbb{Q}^{u}_{T,N}}
\newcommand{\DT}{D_{T}}
\newcommand{\mup}{\mu^{\scalebox{0.7}{$\otimes$}}}

\newcommand{\REFT}{\nu_{T}}

\newcommand{\rhoT}[1]{\rho^{\scriptscriptstyle{T}}_{#1}}
\newcommand{\rhoS}[1]{\rho^{\scriptscriptstyle{S}}_{#1}}
\newcommand{\sigmaT}[1]{\sigma^{\scriptscriptstyle{T}}_{#1}}
\newcommand{\sigmaS}[1]{\sigma^{\scriptscriptstyle{S}}_{#1}}

\newcommand{\PsiT}{\Psi^{\scriptscriptstyle{T},\varphi}_\lambda}
\newcommand{\PsiTz}{\Psi^{\scriptscriptstyle{T},0}_\lambda}

\newcommand{\BtuTN}{B^{u^{\scalebox{0.8}{$\scriptscriptstyle{T},\scriptscriptstyle{N}$}}}_t}

\newcommand{\bW}{\mathbb{W}}
\newcommand{\tW}{W^{\scriptscriptstyle{T},u^{\scalebox{0.8}{$\scriptscriptstyle{T}$}}}}	
\newcommand{\W}[2]{W_{#1}^{{\scriptscriptstyle{T}},#2}}

\newcommand{\WtuTN}{W^{u^{\scalebox{0.8}{$\scriptscriptstyle{T},\scriptscriptstyle{N}$}}}_t}
\newcommand{\WtTuTN }{W^{\scriptscriptstyle{T},u^{\scalebox{0.8}{$\scriptscriptstyle{T},\scriptscriptstyle{N}$}}}_t}
\newcommand{\WPTuT }{W^{\scriptscriptstyle{T},u^{\scalebox{0.8}{$\scriptscriptstyle{T}$}}}}
\newcommand{\WtTuT}{W^{\scriptscriptstyle{T},u^{\scalebox{0.8}{$\scriptscriptstyle{T}$}}}_t}
\newcommand{\WinfTuT}{W^{\scriptscriptstyle{T},u^{\scalebox{0.8}{$\scriptscriptstyle{T}$}}}_\infty}
\newcommand{\WsTuTN }{W^{\scriptscriptstyle{T},u^{\scalebox{0.8}{$\scriptscriptstyle{T},\scriptscriptstyle{N}$}}}_s}
\newcommand{\WuTN}{W^{u^{\scalebox{0.8}{$\scriptscriptstyle{T},\scriptscriptstyle{N}$}}}}
\newcommand{\WTuTN}{W^{\scriptscriptstyle{T},u^{\scalebox{0.8}{$\scriptscriptstyle{T},\scriptscriptstyle{N}$}}}}

\newcommand{\WsTuT}{W^{\scriptscriptstyle{T},u^{\scalebox{0.8}{$\scriptscriptstyle{T}$}}}_s}

\newcommand{\bWt}[1]{\mathbb{W}_t^{\scriptscriptstyle{T},#1}}
\newcommand{\bWinf}[1]{\mathbb{W}_\infty^{\scriptscriptstyle{T},#1}}
\newcommand{\WP}{W^{\scriptscriptstyle{T}}} %process

\makeatletter  
\newcommand{\Wt}[1][\@nil]{%
  \def\tmp{#1}%
   \ifx\tmp\@nnil
       W_t^{\scriptscriptstyle{T}}
    \else
         W_t^{{\scriptscriptstyle{T}},#1}
    \fi}
\makeatother

\makeatletter  
\newcommand{\Ws}[1][\@nil]{%
  \def\tmp{#1}%
   \ifx\tmp\@nnil
       W_s^{\scriptscriptstyle{T}}
    \else
         W_s^{{\scriptscriptstyle{T}},#1}
    \fi}
\makeatother

\makeatletter   
\newcommand{\Winf}[1][\@nil]{%
  \def\tmp{#1}%
   \ifx\tmp\@nnil
       W_\infty^{\scriptscriptstyle{T}}
    \else
         W_\infty^{{\scriptscriptstyle{T}},#1}
    \fi}
\makeatother

\makeatletter   
\newcommand{\WS}[2]{%
 \def\firstarg{#1}%
\def \tmp{}
 \ifx\firstarg\tmp
    W_{#2}^{\scriptscriptstyle{S}}
  \else
      W_{#2}^{\scriptscriptstyle{S},#1}
    \fi}
\makeatother

\newcommand{\WSu}[2]{%
 \def\firstarg{#1}%
\def \tmp{}
 \ifx\firstarg\tmp
    W_{#2}^{{\scriptscriptstyle{S}},u}
  \else
      W_{#2}^{{\scriptscriptstyle{S}},u,#1}
    \fi}
\makeatother

\newcommand{\WSmu}[2]{%
 \def\firstarg{#1}%
\def \tmp{}
 \ifx\firstarg\tmp
    W_{#2}^{{\scriptscriptstyle{S_m}},u}
  \else
      W_{#2}^{{\scriptscriptstyle{S_m}},u,#1}
    \fi}
\makeatother

\newcommand{\Wu}[2]{%
 \def\firstarg{#1}%
\def \tmp{}
 \ifx\firstarg\tmp
    W_{#2}^{u}
  \else
      W_{#2}^{u,#1}
    \fi}
\makeatother

\makeatletter   
\newcommand{\bWS}[2]{%
 \def\firstarg{#1}%
\def \tmp{}
 \ifx\firstarg\tmp
    \mathbb{W}_{#2}^{\scriptscriptstyle{S}}
  \else
      \mathbb{W}_{#2}^{\scriptscriptstyle{S},#1}
    \fi}
\makeatother

\makeatletter
\newcommand{\bWSu}[2]{%
 \def\firstarg{#1}%
\def \tmp{}
 \ifx\firstarg\tmp
    \mathbb{W}_{#2}^{\scriptscriptstyle{S},u}
  \else
      \mathbb{W}_{#2}^{\scriptscriptstyle{S},u,#1}
    \fi}
\makeatother

\makeatletter
\newcommand{\bWu}[2]{%
 \def\firstarg{#1}%
\def \tmp{}
 \ifx\firstarg\tmp
    \mathbb{W}_{#2}^{u}
  \else
      \mathbb{W}_{#2}^{u,#1}
    \fi}
\makeatother

\makeatletter
\newcommand{\AS}[2]{%
 \def\firstarg{#1}%
\def \tmp{}
 \ifx\firstarg\tmp
    A_{#2}^{\scriptscriptstyle{S}}
  \else
      A_{#2}^{\scriptscriptstyle{S},#1}
    \fi}
\makeatother

\makeatletter
\newcommand{\bWp}[2]{%					%%%p for plus: I already had a \bW command that I couldn't delete right away
 \def\firstarg{#1}%
\def \tmp{}
 \ifx\firstarg\tmp
    \mathbb{W}_{#2}^{}
  \else
      \mathbb{W}_{#2}^{#1}
    \fi}
\makeatother

\newcommand{\uP}{u^{\scriptscriptstyle{T}}}
\newcommand{\ut}{u^{\scriptscriptstyle{T}}_t}
\newcommand{\us}{u^{\scriptscriptstyle{T}}_s}

\newcommand{\uNP}{u^{\scriptscriptstyle{T},\scriptscriptstyle{N}}}
\newcommand{\uNt}{u^{\scriptscriptstyle{T},\scriptscriptstyle{N}}_t}
\newcommand{\uNs}{u^{\scriptscriptstyle{T},\scriptscriptstyle{N}}_s}
\newcommand{\uMP}{u^{\scriptscriptstyle{T},\scriptscriptstyle{M}}}

\newcommand{\uMs}{u^{\scriptscriptstyle{T},\scriptscriptstyle{M}}_s}
\newcommand{\uTw}{u^{\scriptscriptstyle{T},w}}

\newcommand{\hTw}{h^{\scriptscriptstyle{T},w}}

\newcommand{\rTw}{r^{\scriptscriptstyle{T},w}}

\newcommand{\lP}{l^{\scriptscriptstyle{T}}}
\newcommand{\lt}{l^{\scriptscriptstyle{T}}_t}

\newcommand{\tauTN}{\tau_{\scriptscriptstyle{T},\scriptscriptstyle{N}}}
\newcommand{\tauTM}{\tau_{\scriptscriptstyle{T},\scriptscriptstyle{M}}}

\renewcommand{\d}{\mathrm{d}}
\newcommand{\dx}{\,\mathrm{d}x}
\newcommand{\dy}{\,\mathrm{d}y}
\newcommand{\ds}{\, \mathrm{d}s}
\newcommand{\dt}{\, \mathrm{d}t}
\newcommand{\dmu}{\, \mathrm{d}\mu}

\begin{document}
\maketitle

\begin{abstract}
In this two-paper series, we prove the invariance of the Gibbs measure for a three-dimensional wave equation with a Hartree nonlinearity. The main novelty is the singularity of the Gibbs measure with respect to the Gaussian free field. The singularity has several consequences in both measure-theoretic and dynamical aspects of our argument. 

\noindent In this paper, we construct and study the Gibbs measure. Our approach is based on earlier work of Barashkov and Gubinelli for the $\Phi^4_3$-model. Most importantly, our truncated Gibbs measures are tailored towards the dynamical aspects in the second part of the series. In addition, we develop new tools dealing with the non-locality of the Hartree interaction. We also determine the exact threshold between singularity and absolute continuity of the Gibbs measure depending on the regularity of the interaction potential. 
\end{abstract}
\setcounter{tocdepth}{2}
\tableofcontents

\section*{Introduction to the series}
In this two-paper series, we study the renormalized wave equation with a Hartree nonlinearity and random initial data given by
\begin{equation}\label{intro_series:eq_nlw} \tag{a}
\begin{cases}
-\partial_{tt}^2 u - u +\Delta   u = \, \lcol  (V*u^2) u \rcol  \qquad (t,x)\in \bR \times \T^3,  \\
u|_{t=0} = \phi_0, \quad \partial_t u |_{t=0} = \phi_1. 
\end{cases}
\end{equation}
Here, $\T \defe \bR / 2\pi \Z$ is the torus and the interaction potential $V \colon \T^3 \rightarrow \bR$ is of the form $V(x)= c_\beta |x|^{-(3-\beta)}$ for all small $x\in \T^3$, where $0<\beta<3$, satisfies $V(x)\gtrsim 1$ for all $x\in \T^3$, is even, and is smooth away from the origin. The nonlinearity $\lcol (V*u^2) u \rcol $ is a renormalization of $(V\ast u^2) u$ (see Definition \ref{measure_re:definition_renormalization} below).  \\

The nonlinear wave equation \eqref{intro_series:eq_nlw} is a prototypical example of a Hamiltonian partial differential equation. The formal Hamiltonian is given by 
\begin{equation*}
H[u,\partial_t u](t) = \frac{1}{2} \Big( \| u(t) \|_{L_x^2}^2  + \| \nabla u(t) \|_{L_x^2}^2   + \| \partial_t u(t) \|_{L_x^2}^2   \Big) + \frac{1}{4} \int_{\T^3} \lcol (V\ast u^2)(t,x) u(t,x)^2 \rcol \dx,
\end{equation*}
where $L^2_x=L_x^2(\T^3)$. 
Based on the Hamiltonian structure, we expect  the formal Gibbs measure $\mup$ given by
\begin{align}\label{intro_series:eq_Gibbs}  \tag{b}
\mathrm{d} \mup (\phi_0,\phi_1)
 = \mathcal{Z}^{-1} \exp(- H(\phi_0,\phi_1) ) \,  \mathrm{d}\phi_0 \mathrm{d}\phi_1 
\end{align}
to be invariant under the flow of \eqref{intro_series:eq_nlw}, where $\cZ$ is a normalization constant. 

The \emph{first part} of this series focuses on the rigorous construction and properties of $\mup$. With a primary focus on the related $\Phi^4_d$-models, similar constructions have been studied in constructive quantum field theory. Recently, this area of research has been revitalized through advances in singular stochastic partial differential equations. The main difficulties come from the quartic interaction $\lcol (V \ast u^2) u^2\rcol$ in the Hamiltonian. In fact, without the interactions, one obtains the Gaussian free field 
\begin{equation*}
\mathrm{d} \cgp (\phi_0,\phi_1) = \cZ_0^{-1}
\exp\Big( - \frac{1}{2} \| \phi_0 \|_{L_x^2}^2  - \frac{1}{2} \| \nabla  \phi_0 \|_{L_x^2}^2 \Big) \, \mathrm{d} \phi_0 ~  \otimes ~ \cZ_1^{-1}
 \exp\Big( - \frac{1}{2} \| \phi_1 \|_{L_x^2}^2 \Big) \, \mathrm{d} \phi_1,
\end{equation*}
which can be constructed through elementary arguments. Using our representation of the Gibbs measure $\mup$, we also prove that $\mup$ and $\cgp$ are mutually singular for $0<\beta<1/2$.

In the \emph{second part} of this series, we study the dynamics of \eqref{intro_series:eq_nlw} with random initial data drawn from the Gibbs measure $\mup$. Due to the low spatial regularity, the local theory requires a mix of techniques from dispersive equations, harmonic analysis, and probability theory. More specifically, we rely on ideas from the para-controlled calculus of Gubinelli, Imkeller, and Perkowski \cite{GIP15}. The heart of this series, however, lies in the global theory. Our main contribution is a new form of Bourgain's globalization argument \cite{Bourgain94}, which addresses the singularity of the Gibbs measure and its consequences.

We now state an qualitative version our main theorem, which combines our measure-theoretic and dynamical results. For the quantitative version, we refer the reader to Theorem \ref{theorem:gibbs_portable} below and Theorem 1.3 in the second part of this series. We recall that the parameter $0<\beta<3$ determines the regularity of the interaction potential $V$. 

\begin{maintheorem}[Global well-posedness and invariance, qualitative version] 
The formal Gibbs measure $\mup$ exists and, for $0<\beta<1/2$, is singular with respect to the Gaussian free field $\cgp$. The renormalized wave equation with Hartree nonlinearity \eqref{intro_series:eq_nlw} is globally well-posed on the support of $\mup$ and the dynamics leave $\mup$ invariant. 
\end{maintheorem}

This is the first example of an invariant Gibbs measure for a dispersive equation which is singular with respect to the Gaussian free field $\cgp$.

\section{Introduction}
In the first paper of this series, we rigorously construct and study the formal Gibbs measure $\mup$ from \eqref{intro_series:eq_Gibbs} above. Since the Hamiltonian $H[\phi_0,\phi_1]$ splits into a sum of functions in $\phi_0$ and $\phi_1$, we can rewrite \eqref{intro_series:eq_Gibbs} as 
\begin{align*}
&\mathrm{d} \mup (\phi_0,\phi_1)  \\
=& \cZ^{-1}_0 \exp\Big(   \hspace{-0.5ex} - \frac{1}{4} \int_{\T^3} \lcol ( V\ast \phi_0^2) \phi_0^2 \rcol \dx  - \frac{1}{2} \| \phi_0 \|_{L^2}^2  - \frac{1}{2} \| \nabla \phi_0 \|_{L^2}^2 \Big) \, \mathrm{d} \phi_0 \,  \otimes \, \cZ^{-1}_1
 \exp\Big(  \hspace{-0.5ex}  - \frac{1}{2} \| \phi_1\|_{L^2}^2 \Big) \, \mathrm{d} \phi_1. 
\end{align*}
The construction and properties of the second factor are elementary (as will be explained below), and we now focus on the first factor. As a result, we are interested in the rigorous construction of a measure $\mu$ which is formally given by 
\begin{equation}\label{intro:eq_mu}
\mathrm{d}\mu (\phi) = \cZ^{-1}  \exp\Big(   \hspace{-0.5ex} - \frac{1}{4} \int_{\T^3} \lcol ( V\ast \phi^2) \phi^2 \rcol \dx  - \frac{1}{2} \| \phi \|_{L^2(\T^3)}^2  - \frac{1}{2} \| \nabla \phi \|_{L^2(\T^3)}^2 \Big) \, \mathrm{d} \phi. 
\end{equation}
Our Gibbs measure $\mu$ is closely related to the  $\Phi^4_d$-models, which replace the three-dimensional torus $\T^3$ by the more general $d$-dimensional torus  $\T^d$ and replace the integrand $\lcol (V\ast \phi^2) \phi^2 \rcol$ by the renormalized quartic power $\lcol \phi^4 \rcol$. Thus, the $\Phi^4_d$-model is formally given by 
\begin{equation}\label{intro:eq_phi4d}
\mathrm{d}\Phi^4_d(\phi) = \cZ^{-1}  \exp\Big(   \hspace{-0.5ex} - \frac{1}{4} \int_{\T^d} \lcol \phi^4 \rcol \dx  - \frac{1}{2} \| \phi \|_{L^2(\T^d)}^2  - \frac{1}{2} \| \nabla \phi \|_{L^2(\T^d)}^2 \Big) \, \mathrm{d} \phi. 
\end{equation}
Aside from their connection to Hamiltonian PDEs, such as nonlinear wave and Schr\"{o}dinger equations, the $\Phi^4_d$-models are of independent interest in quantum field theory (cf. \cite{Folland08}). In most rigorous constructions of measures such $\mu$ or the $\Phi^4_d$-models, the first step consists of a regularization. For instance, one may insert a frequency-truncation in the nonlinearity or replace the continuous spatial domain by a discrete lattice. In a second step, one then proves the convergence of the regularized measures as the regularization is removed, either by direct estimates or compactness arguments. \\

With a particular focus on $\Phi^4_d$-models, the question of convergence of the regularized measures has been extensively studied over several decades. The first proof of convergence was a major success of the constructive field theory program, which thrived during the 1970s and 1980s. We refer the reader to the excellent introduction of \cite{GH19} for a detailed overview and the original works \cite{BCG78,BDF83,FO76,GJ87,MS76,P77,Simon74,W89}. \\
In the 1990s, Bourgain \cite{Bourgain94,Bourgain97,Bourgain96} revisited  the $\Phi^4_d$-model in dimension $d=1,2$ using tools from harmonic analysis and introduced these problems into the dispersive PDE community. Bourgain's works \cite{Bourgain94,Bourgain97,Bourgain96} also contain important dynamical insights, which will be utilized in the second part of this series. \\
Based on the method of stochastic quantization, which was introduced by Nelson \cite{Nelson66,Nelson67} and Parisi-Wu \cite{PW81}, the construction and properties of the $\Phi^4_d$-models have also been studied over the last twenty years in the stochastic PDE community. The main idea behind stochastic quantization is that the $\Phi^4_d$-measure is formally invariant under the stochastic nonlinear heat equation
\begin{equation}\label{intro:eq_heat}
\partial_t u + u - \Delta u = - \lcol u^3 \rcol + \sqrt{2} \xi \qquad (t,x)\in \bR \times \T^d, 
\end{equation}
where $\xi$ is space-time white noise. After prescribing simple initial data, such as $u(0)=0$, one hopes to obtain the $\Phi^4_d$-measure as the limit of the law of $u(t)$ as $t\rightarrow \infty$. In spatial dimensions $d=1,2$, this approach was carried out by Iwata \cite{I85} and Da Prato-Debussche \cite{DPD03}, respectively. In spatial dimension $d=3$, however, \eqref{intro:eq_heat} is highly singular and the local well-posedness theory of \eqref{intro:eq_heat} is beyond classical methods in stochastic partial differential equations. In groundbreaking work \cite{Hairer14}, Hairer introduced regularity structures, which provide a detailed description  of the local dynamics of \eqref{intro:eq_heat}. Alternatively, the local well-posedness of \eqref{intro:eq_heat} was also obtained by Catellier and Chouk in \cite{CC18}, which is based on the para-controlled calculus of Gubinelli, Imkeller, and Perkowski \cite{GIP15}. In order to construct the $\Phi^4_3$-model using \eqref{intro:eq_heat}, however, local control over the solution is not sufficient, and one needs a global well-posedness theory. The global theory has been addressed very recently in \cite{AK17,GH19,HM18,MW17}, which combine regularity structures or para-controlled calculus with further PDE arguments, such as the energy method. Using similar tools, Barashkov and Gubinelli \cite{BG18,BG20} recently developed a variational approach to the $\Phi^4_3$-model, which does not directly rely on the stochastic heat equation \eqref{intro:eq_heat}. Their work forms the basis of this paper and will be discussed in more detail below.  \\

After this broad overview of the relevant literature, we now begin a more detailed discussion of the previous methods. Throughout this discussion we encourage the reader to think of the nonlinear wave equation as a Hamiltonian system of ordinary differential equations in Fourier space. We begin with the elementary construction of the Gaussian free field. Then, we discuss the construction of the $\Phi^4_1$ and $\Phi^4_2$-models using harmonic analysis, similar as in Bourgain's works \cite{Bourgain94,Bourgain96}, and the construction of the $\Phi^4_3$-model using the variational approach of Barashkov and Gubinelli \cite{BG18}. \\

Given a function $\phi \colon \T^d \rightarrow \bR$, its Fourier expansion is given by
\begin{equation}
\phi(x) = \sum_{n\in \Z^d} \widehat{\phi}(n) e^{i\langle n , x \rangle}. 
\end{equation}
Due to the real-valuedness of $\phi$, the sequence $(\widehat{\phi}(n))_{n \in \Z^d}$ satisfies the symmetry condition $\overline{\widehat{\phi}}(n)= \widehat{\phi}(-n)$. In order to respect this symmetry, we let $\Lambda\subseteq \Z^d$ be such that 
$\Z^d = \{ 0 \} \, \biguplus \, \Lambda \, \biguplus \, (-\Lambda )$, where $\biguplus$ denotes the disjoint union. For $n\in \Lambda$, we denote by $\mathrm{d}\widehat{\phi}(n)$ the Lebesgue measure on $\mathbb{C}$, and for $n=0$, we denote by  $\mathrm{d}\widehat{\phi}(0)$ the Lebesgue measure on $\R$. We can then formally identify the $d$-dimensional Gaussian free field
\begin{equation}\label{intro:eq_GFF_d}
\mathrm{d}\cg_d(\phi) =\cZ^{-1}  \exp\Big(   - \frac{1}{2} \| \phi \|_{L^2(\T^d)}^2  - \frac{1}{2} \| \nabla \phi \|_{L^2(\T^d)}^2 \Big) \, \mathrm{d} \phi
\end{equation}
as the push-forward under the Fourier transform of 
\begin{equation}\label{intro:eq_GFF_d_product}
\begin{aligned}
&\cZ^{-1}  \exp\Big( - \frac{1}{2} \sum_{n\in \Z^d} (1+|n|^2) |\widehat{\phi}(n)|^2 \Big) \bigotimes_{n\in \{ 0 \} \cup  \Lambda} \mathrm{d} \widehat{\phi}(n) \\
&= \frac{1}{2\pi}  \exp\Big( - \frac{1}{2} |\widehat{\phi}(0)|^2  \Big) \mathrm{d} \widehat{\phi}(0)   \otimes \Big( \bigotimes_{n\in  \Lambda}  \frac{1}{\pi \langle n \rangle^2} \exp\Big( -    \langle n \rangle^2  |\widehat{\phi}(n)|^2 \Big) \mathrm{d} \widehat{\phi}(n)  \Big),
\end{aligned}
\end{equation}
where $\langle n \rangle^2 = 1 + |n|^2 $. While \eqref{intro:eq_GFF_d} is entirely formal, the right-hand side of \eqref{intro:eq_GFF_d_product} is a well-defined product measure. Under the measure in \eqref{intro:eq_GFF_d_product}, $\widehat{\phi}(0)$ is a standard real-valued Gaussian and  $ ( \widehat{\phi}(n))_{n\in  \Lambda} $ is a sequence of independent complex Gaussians satisfying $\E |\widehat{\phi}( n )|^2 = \langle n \rangle^{-2}$. Turning this formal discussion around, we let $(\Omega,\mathcal{F},\bP)$ be an ambient probability space containing a sequence of independent complex-valued standard Gaussians $(g_n)_{n \in \Lambda}$ and a standard real-valued Gaussian $g_0$. Then, we can rigorously define the Gaussian free field $\cg_d$ by 
\begin{equation}\label{intro:eq_GFF_d_rigorous}
\mathrm{d}\cg_d(\phi) = \Big( \sum_{n\in \Z^d} \frac{g_n}{\langle n\rangle} e^{i\langle n , x \rangle}\Big)_{\#} \bP, 
\end{equation}
where the subscript $\#$ denotes the pushforward. Using the representation \eqref{intro:eq_GFF_d_rigorous}, we see that a typical sample of $\cg_d$ almost surely lies in $H_x^s(\T^d)$ for all $s < 1 - d/2$ but not in $H_x^{1-d/2}(\T^d)$. \\

We now turn to the construction of the $\Phi^4_1$ and $\Phi^4_2$-models. Based on our formal expression of the $\Phi^4_1$-model in \eqref{intro:eq_phi4d}, we would like to define 
\begin{equation}\label{intro:eq_Phi41}
\mathrm{d} \Phi^4_1(\phi) \defe \cZ^{-1} \exp\Big( - \frac{1}{4} \int_{\T} \phi^4(x) \dx \Big) \mathrm{d}\cg_1(\phi). 
\end{equation}
Using either Sobolev embedding or Khintchine's inequality, we obtain $\cg_1$-almost surely that $0 < \| \phi \|_{L^4(\T)}< \infty$. This implies that the density $\mathrm{d} \Phi^4_1 / \mathrm{d}\cg_1$ is well-defined, almost surely positive, and lies in $L^q(\cg_1)$ for all $1\leq q \leq \infty$. In particular, the $\Phi^4_1$-model is absolutely continuous with respect to the Gaussian free field $\cg_1$. We emphasize that the potential energy in \eqref{intro:eq_Phi41} does not require a renormalization. Furthermore, we can define truncated $\Phi^4_1$-models by 
\begin{equation*}
\mathrm{d} \Phi^4_{1;N}(\phi) \defe \cZ^{-1}_N \exp\Big( - \frac{1}{4} \int_{\T} (P_{\leq N} \phi)^4(x) \dx \Big) \mathrm{d}\cg_1(\phi),
\end{equation*}
where $N$ is a dyadic integer and $P_{\leq N}$ a Littlewood-Paley projection. As was shown in \cite{Bourgain94}, direct estimates yield the convergence of  $\mathrm{d} \Phi^4_{1;N}/ \mathrm{d}\cg_1$  in $L^q(\cg_1)$ for all $1\leq q < \infty$ and hence $\Phi^4_{1;N}$ converges to $\Phi^4_1$ in total variation as $N$ tends to infinity. 

In two spatial dimensions, however, we encounter a new difficulty. Since $\cg_1$-almost surely $\| \phi\|_{L^2}=\infty$,  the potential energy $\| \phi \|_{L^4}^4$ is almost surely infinite. As a result, the potential energy requires a renormalization. A direct calculation using the definition of $P_{\leq N}$ in \eqref{notation:eq_PleqN} below yields 
\begin{equation*}
\sigma_N^2 = \int_{0}^\infty \mathrm{d} \cg_2(\phi) \| P_{\leq N} \phi \|_{L^2(\T^2)}^2 \sim \log(N). 
\end{equation*}
We then replace the monomial $(P_{\leq N} \phi)^4$ by the Hermite polynomial 
\begin{equation*}
\lcol (P_{\leq N} \phi)^4 \rcol = (P_{\leq N} \phi)^4 - 6 \sigma_N^2 (P_{\leq N} \phi)^2 + 3 \sigma_N^4. 
\end{equation*}
This leads to the truncated $\Phi^4_2$-model given by
\begin{equation*}
\mathrm{d} \Phi^4_{2;N}(\phi) \defe \cZ^{-1}_N \exp\Big( - \frac{1}{4} \int_{\T^2} \lcol (P_{\leq N} \phi)^4 \rcol (x) \dx \Big) \mathrm{d}\cg_2(\phi). 
\end{equation*}
After this renormalization, one can show (cf. \cite{OT18}) that the densities $\mathrm{d} \Phi^4_{2;N}/\mathrm{d}\cg_2$  converge in $L^q(\cg_2)$ for all $1\leq q < \infty$ and we can define $\Phi^4_2$ as the limit (in total-variation) of $\Phi^4_{2;N}$ as $N\rightarrow \infty$. As in one spatial dimension, the $\Phi^4_2$-model is absolutely continuous with respect to the Gaussian free field $\cg_2$. Using similar tools as for the $\Phi^4_2$-model, Bourgain \cite{Bourgain97} constructed the Gibbs measure $\mu$ for the Hamiltonian with a Hartree interaction for $\beta>2$, which corresponds to a relatively smooth interaction potential $V$.  The key point of this paragraph is that  the $\Phi^4_1$-model, the $\Phi^4_2$-model, and the Gibbs measure $\mu$ for a smooth interaction potential can be constructed through ``hard'' analysis. As a result, one obtains strong modes of convergence and absolute continuity with respect to Gaussian free field. \\

The construction of the $\Phi^4_3$-model, however, is much more complicated. As will be described below, several of the ``hard'' conclusions, such as convergence in total-variation or absolute continuity with respect to the Gaussian free field, are either unavailable or fail. As a result, we have to (partially) replace hard estimates by softer compactness arguments. We now give a short overview of the variational approach in  \cite{BG18,BG20}, which forms the basis of this paper. \\ 
In order to use techniques from stochastic control theory, we introduce a family of Gaussian processes $(W_t(x))_{t\geq 0}$ on an ambient probability space $(\Omega,\mathcal{F},\bP)$ satisfying $\Law_{\bP}(W_\infty) = \cg_3$, which will be defined in Section \ref{section:stochastic_control}. We view $t$ as a stochastic time-variable which serves as a regularization parameter. Using this terminology, we obtain a truncated $\Phi^4_3$-model by setting
\begin{align*}
\mathrm{d} \Phi^4_{3;T}(\phi) = (W_\infty)_{\#} \big( \mathrm{d} \overline{\Phi}^4_{3;T}(\phi) \big) 
\end{align*}
and
\begin{align*}
\mathrm{d} \overline{\Phi}^4_{3;T}(\phi) = \cZ_T^{-1} \exp\big( - \frac{1}{4} \int_{\T^3} W_T^4(x) - a_T W_T^2(x) - b_T \dx \big) \mathrm{d}\bP. 
\end{align*}
We emphasize already that the $\Phi^4_{3;T}$-measure does not correspond to a truncated Hamiltonian, which will be discussed in full detail in Section \ref{section:stochastic_control}.  In order to construct the $\Phi^4_3$-model, the main step is to prove the tightness of the $\Phi^4_{3;T}$-measures.  Using Prokhorov's theorem, this implies the weak convergence of a subsequence of $\Phi^4_{3;T}$ and we can define the $\Phi^4_3$-measure as the weak limit.  To prove tightness, Barashkov and Gubinelli obtain uniform bounds in $T$ on the Laplace transform 
\begin{equation*}
f \in C( \cC_x^{-\frac{1}{2}-}(\T^3);\R) \rightarrow \int \mathrm{d} \Phi^4_{3;T}(\phi) ~  e^{-f(\phi)}. 
\end{equation*}
The main ingredients for the uniform bounds are the Boué-Dupuis formula (Theorem \ref{thm:bd_formula}) and the para-controlled calculus of Gubinelli, Imkeller, and Perkowski \cite{GIP15}, which has also been used in the stochastic quantization approach to the $\Phi^4_3$-model (cf. \cite{GH19}).  \\
While the variational approach yields the existence of the $\Phi^4_3$-measure, it only yields limited information regarding its properties. In spatial dimensions $d=1,2$, the $\Phi^4_d$-model is absolutely continuous with respect to the Gaussian free field $\cg_d$, and hence the samples of $\Phi^4_3$ for many purposes behave like a random Fourier series with independent coefficients. This is an essential ingredient in almost all invariance arguments for random dispersive equations (see e.g. \cite{Bourgain97,Bourgain96,DNY19,NORS12}). Unfortunately, the $\Phi^4_3$-measure is singular with respect to the Gaussian free field $\cg_3$. This fact seems to be part of the folklore in mathematical physics, but it is surprisingly difficult to find a detailed reference. In an unpublished note available to the author \cite{HairerNote}, Martin Hairer proved the singularity using the stochastic quantization approach and regularity structures. Using the Girsanov-transformation, Barashkov and Gubinelli \cite{BG20} constructed a reference measure $\nu_3^4$ for the $\Phi^4_3$-model,  which serves a similar purpose as the Gaussian free field for $\Phi^4_1$ and $\Phi^4_2$. The samples of $\nu^4_3$ are given by an explicit  Gaussian chaos of finite order and $\Phi^4_3$ is absolutely continuous with respect to $\nu^4_3$. Furthermore, Barashkov and Gubinelli proved that the reference measure $\nu^4_3$ and the Gaussian free field $\cg_3$ are mutually singular, which yields a self-contained proof of the singularity of $\Phi^4_3$ with respect to the Gaussian free field $\cg_3$.

\subsection{Main results and methods} 

In the following, we simply write $\cg=\cg_3$ for the three-dimensional Gaussian free field. Let $N\geq 1$ be a dyadic integer and define the  renormalized potential energy by 
\begin{equation}\label{intro:eq_potential}
\lcol \cV_N^\lambda(\phi)\rcol \defe  \frac{\lambda}{4} \int_{\T^3} \Big( ( V \ast \phi^2) \phi^2  - 2 a_N \phi^2 - 4 (\MN \phi) \phi + \widehat{V}(0) a_N^2 + 2 b_N \Big) \dx  + c_N^\lambda.  
\end{equation}
The coupling constant $\lambda>0$ is introduced for illustrative purposes, but the reader may simply set $\lambda=1$ as in all previous discussions. The renormalization constants $a_N,b_N$, and $c_N^\lambda$ are as in Definition \ref{measure_re:definition_renormalization_dynamics} and Proposition \ref{measure_var:eq_renormalized_V} and the renormalization multiplier $\MN$ is as in Definition \ref{measure_re:definition_renormalization_dynamics}. We emphasize that the renormalization in \eqref{intro:eq_potential} goes beyond the usual Wick-ordering, which is only based on the mass $\| P_{\leq N} \phi \|_{L^2}^2$. The additional renormalization is contained in the renormalization constant $c_N^\lambda$, which is related to the mutual singularity of $\mup$ and $\cg$ (for $0<\beta<1/2$). 
The truncated and renormalized Hamiltonian $H_N$ is given by 
\begin{equation}\label{intro:eq_HN}
H_N[\phi_0,\phi_1] \defe \frac{1}{2} \Big( \| \phi_0 \|_{L^2}^2 + \| \nabla \phi_0 \|_{L^2}^2 + \| \phi_1 \|_{L^2}^2 \Big) + 
\lcol \cV_N^\lambda(P_{\leq N} \phi_0)\rcol~,
\end{equation}
where we omit the dependence on $\lambda>0$ from our notation. We emphasize that only the quartic term contains a frequency-truncation and renormalization, whereas the quadratic terms remain unchanged. As described in the beginning of the introduction, we focus on the first factor of the truncated Gibbs measure $\mup_N$, which is given by
\begin{equation}
\dmu_N(\phi) = \frac{1}{\cZ_N^\lambda} \exp\Big( - \lcol \cV_N^\lambda(P_{\leq N} \phi)\rcol \Big) \mathrm{d}\cg(\phi). 
\end{equation}
Before we state our main result, we recall the assumptions on the interaction potential $V\colon \T^3 \rightarrow \bR$ from the introduction to the series. In these assumptions, $0<\beta<3$ is a parameter. 

\begin{assumptions}{A}\label{assumptions:V}
We assume that the interaction potential $V$ satisfies 
\begin{enumerate}
\item $V(x)= c_\beta |x|^{-(3-\beta)}$ for some $c_\beta>0$ and all $x\in \T^3$ satisfying $\| x\|\leq 1/10$,
\item $V(x) \gtrsim_\beta 1$ for all $x\in \T^3$, 
\item $V(x)=V(-x)$ for all $x\in \T^3$,
\item $V$ is smooth away from the origin. 
\end{enumerate}
\end{assumptions}

We now state the conclusions of this paper which will be needed in the second part of this series \cite{Bringmann20}. A more comprehensive version of our results will then be stated in Theorem \ref{theorem:tightness}, Theorem \ref{theorem:reference}, and Theorem \ref{theorem:singularity} below. The additional results may be useful in further applications, such as invariant measures for a Schr\"{o}dinger equation with a Hartree nonlinearity. 

\begin{theorem}[The Gibbs measure]\label{theorem:gibbs_portable}
Let $\kappa>0$ be a fixed positive parameter, let $0<\beta<3$ be a parameter, and let the  interaction potential $V$ be as in the Assumptions \ref{assumptions:V}. Then, the sequence of truncated Gibbs measures $(\mu_N)_{N\geq 1}$ converges weakly to a probability measure $\mu_\infty$ on $\cC_x^{-1/2-\kappa}(\T^3)$, which is called the Gibbs measure. If in addition $0<\beta<1/2$, the Gibbs measure $\mu_\infty$ and the Gaussian free field $\cg$ are mutually singular.  Furthermore, there exists a sequence of reference measures $(\nu_N)_{N\geq 1}$ on $\cC_x^{-1/2-\kappa}(\T^3)$ and an ambient probability space $(\Omega,\cF,\bP)$ satisfying the following properties:
\begin{enumerate}
\item (Absolute continuity and $L^q$-bounds) The truncated Gibbs measures $\mu_N$ are absolutely continuous with respect to the reference measures $\nu_N$. More quantitatively, there exists a parameter $q>1$ and a constant $C\geq 1$, depending only on $\beta$, such that 
\begin{equation*}
\mu_N(A) \leq C \nu_N(A)^{1-\frac{1}{q}}
\end{equation*}
for all Borel sets $A \subseteq  \cC_x^{-1/2-\kappa}(\T^3)$. 
\item (Representation of $\nu_N$) Let $\gamma=\min(1/2+\beta,1)$. There exists a large integer $k=k(\beta)$ and two random functions $\cG, \cR_N\colon (\Omega,\cF) \rightarrow   \cC_x^{-1/2-\kappa}(\T^3)$  satisfying for all $p\geq 2$ that 
\begin{equation*}
\nu_N = \Law_{\bP}\big( \cG + \cR_N \big), \quad \cg = \Law_{\bP} \big( \cG \big) , \quad \text{and} \quad \| \cR_N \|_{L^p_\omega \cC_x^{\gamma-\kappa}(\Omega\times \T^3)} \leq p^{\frac{k}{2}}. 
\end{equation*}
\end{enumerate}
\end{theorem}

\begin{remark}\label{remark:OOT}
After the completion of (the first version of) this series, the author learned of independent work by Oh, Okamoto, and Tolomeo \cite{OOT20}, which discusses the focusing and defocusing three-dimensional (stochastic) nonlinear wave equation with a Hartree nonlinearity. In the focusing case, the authors provide a complete picture of the construction and properties of the focusing Gibbs measures, which distinguishes the three regimes $\beta>2$, $\beta=2$, and $\beta<2$ (cf. \cite{OOT20}). In the defocusing case, the authors construct the Gibbs measures for $\beta>0$ and prove the singularity for $0<\beta\leq 1/2$, which includes the endpoint $\beta=1/2$. The reference measures are briefly discussed in \cite[Appendix C]{OOT20}, but only play a minor role in their analysis. The $L^q$-bound in Theorem \ref{theorem:gibbs_portable}, which will be essential in the second part of this series \cite{Bringmann20}, is not proven in \cite{OOT20}. 

In the first version of this manuscript, we proved the tightness of the truncated Gibbs measures $(\mu_N)_{N\geq 1}$, which only implies that a subsequence of $(\mu_N)_{N\geq 1}$. In \cite{OOT20}, the authors proved the uniqueness of weak subsequential limits, which lead to the convergence of the full sequence. A version of the uniqueness argument from \cite{OOT20}, which has been modified to match our notation, has now been included in Appendix \ref{section:uniquess}.

While the measure-theoretic part of  \cite{OOT20} treats all $\beta>0$, the dynamical results are restricted to $\beta>1$. In particular, the singular regime $0<\beta<1/2$ is not covered, which is the main object of this series.
\end{remark}
In addition to the singular regime $0<\beta<1/2$,  the most interesting cases in Theorem \ref{theorem:gibbs_portable} are the Newtonian potential $|x|^{-2}$ (corresponding to $\beta=1$) and the Coulomb potential $|x|^{-1}$ (corresponding to $\beta=2$).  As mentioned earlier in the introduction, Bourgain \cite{Bourgain97} proved a version of Theorem \ref{theorem:gibbs_portable} in the limited range $\beta>2$, which corresponds to a relatively smooth interaction potential. \\

We now split the main theorem  (Theorem \ref{theorem:gibbs_portable}) into three parts:
\begin{itemize}
\item[$\bullet$] the tightness and weak convergence of the truncated Gibbs measures $\mu_N$, 
\item[$\bullet$] the construction and properties of the reference measures $\nu_N$,
\item[$\bullet$] the mutual singularity of the Gibbs measure and the Gaussian free field. 
\end{itemize}

\begin{theorem}[Tightness and convergence]\label{theorem:tightness}
The truncated Gibbs measures $(\mu_N)_{N\geq 1}$ are tight on $\cC_x^{-1/2-\kappa}(\T^3)$. Furthermore, the sequence $(\mu_N)_{N\geq 1}$ weakly converges to a limiting measure $\mu_\infty$. 
\end{theorem}

The overall strategy of the proof of Theorem \ref{theorem:tightness} is the same as in the variational approach of Barashkov and Gubinelli \cite{BG18}. In comparison with \cite{BG18}, the terms in this paper often have a more complicated algebraic structure but obey better analytical estimates. As any reader familiar with regularity structures or para-controlled calculus may certify, the algebraic structure of most stochastic objects is already quite complicated, so this trade-off is not always favorable. In addition, the non-locality of the nonlinearity requires different analytical estimates and we mention the two most important examples: 
\begin{enumerate}[(i)] 
\item The coercive term $\| f\|_{L^4}^4$ in the variational problem for the $\Phi^4_3$-model is replaced by the potential energy
\begin{equation*}
\int_{\T^3} (V \ast f^2) f^2 \dx. 
\end{equation*}
We emphasize that the coercive term in the variational problem does not contain a renormalization, which is a result of the binomial formula in Lemma \ref{measure_re:lemma_binomial}. In order  to use the potential energy in our estimates, we rely  on a fractional derivative estimate of Visan \cite[(5.17)]{Visan07}. 
\item In the variational problem, we encounter mixed terms of the form 
\begin{equation*}
\int_{\T^3} \Big[ \big( V \ast ( P_{\leq N} W_\infty \cdot P_{\leq N} f_1 )  \cdot P_{\leq N} W_\infty \cdot P_{\leq N} f_2 - \big( \MN P_{\leq N} f_1 \big) P_{\leq N} f_2 \Big] \dx,
\end{equation*}
where $(W_t)_{t\geq 0}$ is the Gaussian process from the introduction. Based on the literature on random dispersive equations \cite{Bourgain97,Bourgain96,DNY19,DNY20,GKO18a}, it is tempting to bound this mixed term through Fourier-analytic and random matrix techniques. We instead develop a simpler and elegant physical-space approach. 
\end{enumerate}

The next theorem gives a more detailed description of the reference measures in Theorem \ref{theorem:gibbs_portable}. To simplify the notation, we allow the truncation parameter $N$ to take the value $\infty$. 
\begin{theorem}[Reference measures]\label{theorem:reference}
There exists a family of reference measures $(\nu_N)_{1 \leq N \leq \infty}$  and an ambient probability space $(\Omega,\cF,\bP)$ satisfying the following properties:
\begin{enumerate}
\item Absolute continuity and $L^q$-bounds: The truncated Gibbs measures $\mu_N$ are absolutely continuous with respect to the reference measures $\nu_N$. More quantitatively, there exists a parameter $q>1$ and a constant $C\geq 1$, depending only on $\beta$, such that 
\begin{equation*}
\mu_N(A) \leq C \nu_N(A)^{1-\frac{1}{q}}
\end{equation*}
for all Borel sets $A \subseteq  \cC_x^{-1/2-\kappa}(\T^3)$. 
\item Representation of $\nu_N$: We have that 
\begin{equation*}
\nu_N = \Law_{\bP} \big( \cG^{(1)} + \cG^{(3)}_N + \cG^{(n)}_N \big). 
\end{equation*}
Here, $n=n(\beta)$ is a large integer and the linear, cubic, and $n$-th order Gaussian chaoses are explicitly given by
\begin{align*}
\cG^{(1)}&= W_\infty, \\
\cG^{(3)}_N  &= -\lambda   P_{\leq N} \int_0^\infty J_s^2 \Big( \lcol  (  V \ast ( P_{\leq N} W_s )^2 ) P_{\leq N} W_s \rcol \Big) \ds , \\ 
\cG^{(n)}_N &=  P_{\leq N} \int_0^\infty \langle \nabla \rangle^{-\frac{1}{2}} J_s^2 \Big( \lcol ( \langle \nabla \rangle^{-\frac{1}{2}} P_{\leq N} W_s )^n \rcol \Big) \ds,
\end{align*}
where we refer the reader to Section \ref{section:stochastic_control} and Definition \ref{measure_re:definition_renormalization} for the definitions of $J_s$ and the renormalizations. 
\end{enumerate}
\end{theorem}

We emphasize that the representation of $\nu_N$ in Theorem \ref{theorem:reference} is much more detailed than stated in Theorem \ref{theorem:gibbs_portable}. This additional information is not required in our proof of global well-posedness and invariance in the second part of the series. However, we believe that the more detailed representation way be relevant for the Schr\"{o}dinger equation with a Hartree nonlinearity. The reason lies in low$\times$low$\times$high-interactions, which are more difficult in Schr\"{o}dinger equations than in wave equations. In the last two years, we have seen new and intricate methods dealing with these interactions \cite{B18,DNY19,DNY20}, but all of these papers heavily rely on the independence of the Fourier coefficients. In fact, overcoming this obstruction is mentioned as an open problem in \cite[Section 9.1]{DNY20}. \\
The proof of Theorem \ref{theorem:reference} is based on the Girsanov-approach of Barashkov and Gubinelli \cite{BG20}. As mentioned earlier, however, we cannot use the same approximate Gibbs measures as in \cite{BG20}, since they do not correspond to a frequency-truncated Hamiltonian. In the second part of the series, the frequency-truncated Hamiltonians are an essential ingredient in the proof of global well-posedness and invariance. This difference will be discussed in detail in Section \ref{section:stochastic_control}. For now, we simply mention that there is a trade-off between desirable properties from a PDE or probabilistic perspective. \\

Our last theorem describes the relationship between the Gibbs measure $\mu_\infty$ and the Gaussian free field $\cg$.

\begin{theorem}[Singularity]\label{theorem:singularity}
If \( 0< \beta < 1/2 \), then the Gibbs measure \( \mu_\infty \) and the Gaussian free field \( \cg \) are mutually singular. If \( \beta > 1/2 \), then the Gibbs measure is absolutely continuous with respect to the Gaussian free field \( \cg \).
\end{theorem}

Theorem \ref{theorem:singularity} determines the exact threshold between absolute continuity and singularity of $\mu_\infty$ with respect to $\cg$. As mentioned in Remark \ref{remark:OOT}, the singularity at the endpoint $\beta=1/2$ has been obtained in independent work by Oh, Okamoto, and Tolomeo \cite{OOT20}. The absolute continuity for $\beta>1/2$ already follows from the variational estimates in our construction of $\mu_\infty$. The main step is the mutual singularity of $\mu_\infty$ and $\cg$ for $0<\beta<1/2$. We provide an explicit event witnessing this singularity, which is based on the behaviour of the frequency-truncated potential energy
\begin{equation*}
\int_{\T^3} \lcol (V \ast (P_{\leq N} \phi)^2 )  (P_{\leq N} \phi)^2 \rcol \dx 
\end{equation*}
under the different measures. \\

\textbf{Acknowledgements:} The author thanks his advisor Terence Tao for his patience and invaluable guidance. The author also thanks Nikolay Barashkov, Martin Hairer, Redmond McNamara, Dana Mendelson, Tadahiro Oh, and Felix Otto for helpful discussions.

\subsection{Overview} 
To orient the reader, let us review the rest of this paper. In Section \ref{section:stochastic_control}, we introduce the stochastic control perspective and recall the Boué-Dupuis formula. In Section \ref{section:stochastic_objects}, we estimate several stochastic objects, such as the renormalized nonlinearity $\lcol (V\ast W_t^2) W_t\rcol$. Our main tools will be Itô's formula and Gaussian hypercontractivity. In Section \ref{section:construction}, we prove the tightness of the truncated Gibbs measures $\mu_N$ and construct the limiting measure $\mu_\infty$. Using the Laplace transform and the Boué-Dupuis formula, the proof of tightness reduces to estimates for a variational problem, which occupy most of this section. In Section \ref{section:drift_measure}, we first construct the reference measures $\nu_N$ and then examine their properties. The main ingredients are Girsanov's transformation and our earlier variational estimates. Finally, in  Section \ref{section:singularity}, we prove the singularity of the Gibbs measure \( \mu_\infty \) with respect to the Gaussian free field \( \cg \) for all \( 0 <\beta < 1/2 \).  \\

\subsection{Notation} In the rest of the paper, we use \(  \defe \) instead of \( := \) for definitions. The reason is that the colon in \( := \) may be confused with our notation for renormalized powers in Definition \ref{measure_re:definition_renormalization} below. With a slight abuse of notation, we write \( \dx \) for the normalized Lebesgue measure on \( \T^3 \). That is, we implicitly normalize
\begin{equation*}
\int_{\T^3} 1 \dx = 1.
\end{equation*}
We define the Fourier transform of a function \( f \colon \T^3 \rightarrow \mathbb{C} \) by
\begin{equation*}
\widehat{f}(n) \defe \int_{\T^3} f(x) e^{-inx} \dx. 
\end{equation*}

For any \( k \in \mathbb{N} \) and \( n_1,\hdots,n_k \in \Z^3 \), we define 
\begin{equation}\label{eq_n12}
n_{12\hdots k} \defe \sum_{j=1}^k n_j. 
\end{equation}
For instance, \( n_{12}= n_1 +n_2 \) and \( n_{123} = n_1 + n_2 + n_3 \). 

We now introduce our frequency-truncation operators. 
We let \( \rho \colon \R_{>0} \rightarrow [0,1] \) be a smooth, non-increasing  function satisfying \( \rho(y) = 1 \) for all \( 0\leq y \leq 1/4 \) and \( \rho (y) = 0 \) for all \( y \geq 4 \). We also assume that \( \min( \rho(y), - \rho^\prime(y)) \gtrsim 1 \)  for all \( 1/2 \leq y \leq 2 \). 
For any \( t \geq 0 \) and \( n \in \Z^3 \), we also define
\begin{equation*}
\rho_t(n) \defe \rho \Big( \frac{\|n\|_2}{\langle t \rangle}\Big). 
\end{equation*}
In particular, it holds that \( t \mapsto \rho_t(\xi) \) is non-decreasing. In order to break up the frequency truncation, we also set 
\begin{equation}\label{intro:eq_sigma}
\sigma_t(n) \defe \Big( \frac{\mathrm{d}}{\mathrm{d}t}\rho_t(n) \Big)^{\frac{1}{2}}. 
\end{equation}
This continuous approach instead of the usual discrete decomposition will be essential in the stochastic control approach (Section \ref{section:stochastic_control}). Nevertheless, we will sometimes use the usual dyadic Littlewood-Paley operators. For any dyadic \( N \geq 1 \), we define \( P_{\leq N} \) by 
\begin{equation}\label{notation:eq_PleqN}
\widehat{P_{\leq N}f}(n) = \rho_N(n) \widehat{f}(n). 
\end{equation}
We further set 
\begin{equation*}
P_{1} f= P_{\leq 1} f \qquad \text{and} \qquad P_N f = P_{\leq N} f - P_{\leq N/2} f \quad \text{for all } N \geq 2.
\end{equation*}
The corresponding Fourier multipliers are denoted by 
\begin{equation}\label{notation:eq_chi}
\chi_1(n)= \rho_1(n) \qquad \text{and} \qquad \chi_N(n) = \rho_{N}(n) - \rho_{N/2}(n) \quad \text{for all } N \geq 2. 
\end{equation}

For any \( s \in \R \), the \( \cC_x^s(\T^3) \)-norm is defined as 
\begin{equation}
\| f \|_{\cC_x^s(\T^3)} \defe \sup_{N\geq 1} N^s \| P_N f \|_{L^\infty_x(\T^3)}.
\end{equation}
We then define the corresponding space \( \cC_x^s(\T^3) \) by 
\begin{equation}\label{notation:eq_Cx}
\cC_x^s(\T^3) \defe \big\{ f \colon \T^3 \rightarrow \R |~  \| f \|_{\cC_x^s} < \infty, \lim_{N\rightarrow \infty} N^s \| P_N f \|_{L^\infty_x(\T^3)} = 0 \big\}. 
\end{equation}
Due to the additional constraint as \( N \rightarrow \infty \), the space \( \cC_x^s(\T^3) \) is separable. This allows us to later use Prokhorov's theorem for families of measures on \( \cC_x^s(\T^3) \). We also define
\begin{equation}\label{notation:eq_CtCx}
\begin{aligned}
&\cC_t^0 \cC_x^s([0,\infty]\times\T^3)  \\
&\defe \big \{ f \colon [0,\infty) \times \T^3 \rightarrow \R |~ \sup_{t\geq 0} \| f(t,\cdot) \|_{\cC^s_x(\T^3)} < \infty, \lim_{t\rightarrow \infty} f(t,\cdot) \text{ exists in } \cC_x^s(\T^3) \big\}. 
\end{aligned}
\end{equation}
Similar as above, the additional restriction as \( t\rightarrow \infty \) makes \( \cC_t^0 \cC_x^s([0,\infty]\times\T^3) \) separable. \\
As a measure of tightness in  \(\cC_t^0 \cC_x^s([0,\infty]\times\T^3) \), we define for any \( 0 < \alpha < 1 \) and \( \eta >0 \) the norm
\begin{equation}\label{notation:eq_Cweighted}
\| f \|_{\cC_t^{\alpha,\eta}\cC_x^s([0,\infty]\times \T^3)} \defe \| f(0)\|_{\cC_x^s(\T^3)} + \sup_{0\leq t,t^\prime \leq \infty} \bigg( \min( \langle t\rangle, \langle t^\prime\rangle)^\eta \frac{\| f(t) - f(t^\prime)\|_{\cC_x^s(\T^3)}}{1 \wedge |t-t^\prime|^\alpha} \bigg). 
\end{equation}
For \( 1 \leq r \leq \infty\), we also define the Sobolev space  \( \mathbb{W}^{s,r}_x(\T^3) \) as the completion of \( C^\infty_x(\T^3) \) with respect to 
\begin{equation*}
\| f \|_{\mathbb{W}_x^{s,r}} = \| N^s P_N f \|_{\ell_N^r L_x^r}.
\end{equation*}
We hope that the subscript \( x \) prevents any confusion with the stochastic objects in Section \ref{section:stochastic_objects}.

\section{Stochastic objects}\label{section:SO}

In this section, we introduce the stochastic control framework and describe several stochastic objects. While the reader with a background in singular SPDE and advanced stochastic calculus can think of this section as standard, much of this section may be new to a reader with a primary background in dispersive PDE. As a result, we include full details for most standard arguments but encourage the expert to skip the proofs. 

\subsection{Stochastic control perspective}\label{section:stochastic_control}

We let \( (B_t^n)_{n\in \Z^3\backslash\{0\}} \)  be a sequence of standard complex Brownian motions such that \( B_t^{-n} = \overline{B_t^n} \) and \( B_t^n, B_t^m \) are independent for \( n \neq \pm m \). We let \( B_t^0 \) be a standard real-valued Brownian motion independent of  \( (B_t^n)_{n\in \Z^3\backslash\{0\}} \). Furthermore, we let \( B_t(\cdot) \) be the Gaussian process with Fourier coefficients \( (B_t^n)_{n \in \Z^3} \), i.e., 
\begin{equation}
B_t(x) \defe \sum_{n\in \Z^3} e^{i\langle n,x \rangle} B_t^n. 
\end{equation}
For every \( t\geq 0 \), the Gaussian process formally satisfies \( \E[ B_t(x)B_t(y) ] = t \cdot \delta(x-y)  \) and hence \( B_t(\cdot)\) is a scalar multiple of spatial white noise.   We also let \( (\cF_t)_{t\geq 0} \) be the filtration corresponding to the family of Gaussian processes \( (B_t^n)_{t\geq 0} \). For future use, we denote the ambient probability space by \( (\Omega,\cF,\bP ) \).

The Gaussian free field \( \cg \), however, has covariance \( (1-\Delta)^{-1} \). To this end, we now introduce the Gaussian process \( W_t(x) \). For \( \sigma_t(x) \) as in \eqref{intro:eq_sigma} and any \( n \in \Z^3 \), we define
\begin{equation}\label{measure:eq_wtn}
W_t^{n} \defe \int_0^t \frac{\sigma_s(n)}{\langle n\rangle} \,  \d B_s^n ~. 
\end{equation}
We note that \( W_t^n \) is a complex Gaussian random variable with variance \( \rho_t^2(n)/\langle n \rangle^2\). We finally set 
\begin{equation}
W_t(x) \defe \sum_{n\in \Z^3} e^{i \langle n , x \rangle} W_t^n. 
\end{equation} 
It is easy to see for any \( \kappa >0 \) that  \( W \in \cC_t^0 \cC_x^{-1/2-\kappa}([0,\infty]\times \T^3) \) almost surely. With a slight abuse of notation, we write \( \d \bP(W) \) for the integration with respect to the law of \( W \) under \( \bP \), i.e., we omit the pushforward by \( W \),  and we write \( W \) for the canonical process on \(  \cC_t^0 \cC_x^{-1/2-\kappa}([0,\infty]\times \T^3) \). Comparing \( W_t \) and \( B_t \), we have changed the covariance from \( t\operatorname{Id}\) to \( \rho_t(\nabla)^2 (I-\Delta)^{-1} \). For any fixed \( T \geq 0 \), we have that 
\begin{equation}
\Law_{\bP}(W_T) = \Law_{\bP}(\rho_T(\nabla)W_\infty). 
\end{equation}
We already emphasize, however, that the processes \( t \mapsto W_t \) and \( t \mapsto \rho_t(\nabla) W_\infty\) have different laws, since only the first process has independent increments. This difference will be important in the definition of \( \tmu_{T} \) below. 
To simplify the notation, we also introduce the Fourier multiplier \( J_t \), which is defined by
\begin{equation}
\widehat{J_t f}(n) \defe \frac{\sigma_t(n)}{\langle n\rangle} \widehat{f}(n), 
\end{equation}
Using this notation, we can represent the Gaussian process \( W_t \) through the stochastic integral 
\begin{equation*}
W_t = \int_0^t J_s \, \d B_s. 
\end{equation*}
In a similar spirit, we define for any \( u\colon [0,\infty)\times \T^3 \rightarrow \R \) the integral \( I_t[u] \) by 
\begin{equation}
I_t[u] \defe \int_0^t J_s u_s \ds. 
\end{equation}
We now recall the Boué-Dupuis formula \cite{BD98}, where our formulation closely follows \cite{BG18,BG20}. We let \( \bH \) be the space of \( \cF_t\)-progressively measurable processes \( u \colon \Omega \times [0,\infty) \times \T^3 \rightarrow \R \) which \( \bP\)-almost surely belong to \( L_{t,x}^2([0,\infty)\times \T^3) \).  

\begin{theorem}[Boué-Dupuis formula]\label{thm:bd_formula}
Let \( 0< T <\infty\), let \( F: {C_t([0,T],C_x^\infty(\T^3)) \rightarrow \R}\) be a Borel measurable function, and let \( 1<p,q<\infty \). Assume that
\begin{equation}\label{measure:eq_bd_condition}
\frac{1}{p}+\frac{1}{q}=1, \quad \E_{\bP} \big[ |F(W)|^p \big] < \infty, \quad \text{and} \quad \E_{\bP} \big[ e^{-q F(W)}\big] < \infty,
\end{equation}
where we regard \( W \) as an element of \( C_t([0,T],C_x^\infty(\T^3)) \). Then, 
\begin{equation}\label{sc:eq_bd}
- \log \E_{\bP} \Big[ e^{-F(W)}\Big] = \inf_{u\in \bH} \E_{\bP} \Big[ F(W+I(u)) + \frac{1}{2} \int_0^T \| u_s\|_{L^2(\T^3)}^2 \ds \Big]. 
\end{equation}
\end{theorem}

\begin{remark}
The optimization problem in \eqref{sc:eq_bd} and, more generally, the change of perspective from \( W_\infty \) to the whole process \( t\mapsto W_t \), is reminiscent of stochastic control theory. \\
Due to the frequency projection in the definition of \( J_t \), we have that \( W_t,  I_t[u] \in C_t([0,T],C^\infty_x(\T^3)) \). In our arguments below, the smoothness can be used to verify \eqref{measure:eq_bd_condition} through soft methods. Of course, a soft method cannot yield uniform bounds in \( T \), which are one of the main goals of this section. 
\end{remark}

In the introduction, we discussed the Gibbs measure \( \mu_N \) corresponding to the truncated dynamics induced by $H_N$, which has been defined in \eqref{intro:eq_HN}. In the spirit of the stochastic control approach, we now change our notation and use the parameter \( T \) to denote the truncation. Since the law of \( W_\infty \) under \( \bP \) is the same as the Gaussian free field \( \cg \) and $P_{\leq T} = \rho_T(\nabla)$, we obtain that 
\begin{equation}\label{sc:eq_def_muT}
\dmu_T(\phi) =  \frac{1}{\cZTl} \exp \Big( - \lcol \cVT(\rho_T(\nabla)\phi)\rcol \Big) \, \mathrm{d}\big( (W_\infty)_\# \bP \big)(\phi). 
\end{equation}
The renormalized potential energy \( \cVT \) is as in \eqref{measure_var:eq_renormalized_V}. We view \( \muT  \) as a measure on the space \( \cC_x^{-1/2-\kappa}(\T^3) \) for any fixed \( \kappa >0 \).  In order to utilize the Boué-Dupuis formula, we lift \( \muT  \) to a measure on \( \cC_t^0 \cC_x^{-1/2-\kappa}([0,\infty] \times \T^3) \).
\begin{definition}\label{sc:definition_tmu}
We define the measure \( \tmuT \) on \( \cC_t^0 \cC_x^{-1/2-\kappa}([0,\infty]\times \T^3) \) by 
\begin{equation}\label{sc:eq_tmu}
\d \tmu_T(W) \defe   \frac{1}{\cZTl}  \exp\big( - \lcol \cVT( \rho_T(\nabla) W_\infty) \rcol \big) \, \d \bP(W). 
\end{equation}
\end{definition}
The content of the next lemma explains the relationship between \( \tmu_T \) and \( \muT  \). 
\begin{lemma}\label{sc:lemma_mu_tmu}
The Gibbs measure \( \muT  \) is the pushforward of \( \tmu_T \) under \( W_\infty \), i.e.,
\begin{equation}\label{sc:eq_mu_tmu}
\muT  = (W_\infty)_\# \tmu_T . 
\end{equation}
\end{lemma}
Due to its central importance to the rest of the paper, we prove this basic identity. 
\begin{proof}
For any measurable function \( f \colon \cC_x^{-\frac{1}{2}-\kappa}(\T^3) \rightarrow \R \), we have that 
\begin{align*}
\int f(\phi) \d \muT (\phi) &=  \frac{1}{\cZTl} \int f(\phi) \exp( - \lcol \cVT(\rho_T(\nabla) \phi) \rcol ) \d \big( (W_\infty)_\# \bP\big) (\phi)  \\
&=  \frac{1}{\cZTl} \int f(W_\infty) \exp( - \lcol \cVT(\rho_T(\nabla) W_\infty) \rcol ) \d \bP(W)  \\
&=  \int f(W_\infty) \mathrm{d}\tmu_T(W) \\
&= \int f(\phi) \mathrm{d} \big( (W_\infty)_\# \tmu_T\big)(\phi). 
\end{align*}
This proves the desired identity \eqref{sc:lemma_mu_tmu}. 
\end{proof}
In \cite{BG18,BG20}, Barashkov and Gubinelli work with the lifted measure
\begin{equation}\label{sc:eq_bmu}
\d \bmu_T(W) =   \frac{1}{\cZTl} \exp\big( - \lcol \cVT( W_T) \rcol \big) \, \d \bP(W). 
\end{equation}
While \( W_T \) and \( \rho_T(\nabla) W_\infty \) have the same distribution, the measures \( \tmu_T \) and \( \bmu_T \) do \emph{not} coincide. Since this is an important difference between this paper and the earlier works \cite{BG18,BG20}, let us explain our motivation for working with \( \tmu_T \) instead of \( \bmu_T \). From a probabilistic stand-point, the measure \( \bmu_T \) has better properties than \( \tmu_T \). This is related to the independent increments of the process \( t \mapsto W_t \) and we provide further comments in Remark \ref{measure_drift:remark_Lp} below. From a PDE perspective, however, \( \bmu_T \) behaves much worse than \( \tmu_T \). For the proof of global well-posedness and invariance in the second part of this series, it is essential that \( \muT  = (W_\infty)_\# \tmu_T \) is invariant under the Hamiltonian flow of \eqref{intro:eq_HN}. In contrast, the author is not aware of an explicit expression for the pushforward of \( \bmu_T \) under \( W_\infty \).  In particular, \( (W_\infty)_\# \bmu_T \) is not directly related to \( \muT  \) and not necessarily invariant under the Hamiltonian flow of $H_N$. Alternatively, we could work with the pushforward of \( \bmu_T \) under \( W_T \). A similar calculation as in the proof of Lemma \ref{sc:lemma_mu_tmu} shows that \( (W_T)_\# \bmu_T = (\rho_T(\nabla))_\# \muT  \). Unfortunately, \( (\rho_T(\nabla))_\# \muT  \) also does not seem to be invariant under a truncation of the nonlinear wave equation.  To summarize, while the measure \( \bmu_T \) has useful probabilistic properties, it lacks a direct relationship to the truncated dynamics and is ill-suited for our globalization and invariance arguments. 

Since we rely on \( \rho_T(\nabla) W_\infty \) in the definition of \( \tmu_T \), the Gaussian process 
\(  t \mapsto \rho_T(\nabla) W_t\) will play an important role in the rest of this paper. As a result, we now deal with both values \( T \) and \( t \) simultaneously. In most arguments, \(T \) will remain fixed while we use Itô's formula and martingale properties in \( t \). To simplify the notation, we now write
\begin{equation}\label{sc:eq_WTt}
\Wt \defe \rho_T(\nabla)W_t \qquad \text{and} \qquad \Wt[n] \defe \rho_T(n) W_t^n. 
\end{equation}
Since this will be convenient below, we also define
\begin{equation}
\rho^T_t(n) \defe \rho_T(n) \cdot \rho_t(n), \qquad \sigma_t^T(n) \defe \rho_T(n) \sigma_t(n),\qquad \text{and} \qquad \Jt \defe \rho_T(\nabla) J_t.
\end{equation}
Furthermore, we define the integral operator \( \It \) by
\begin{equation}
\It[u] = \rho_T(\nabla) I_t[u] = \int_0^t \Js u_s \ds . 
\end{equation}

\subsection{Stochastic objects and renormalization}\label{section:stochastic_objects}

We now proceed with the construction and renormalization of several stochastic objects. Similar constructions are standard in the probability theory literature and a comprehensive and well-written introduction  can be found in \cite{GP18,MWX17,OT18}. In order to make this section accessible to readers with a primary background in dispersive PDEs, however, we include full details. In a similar spirit, we follow a hands-on approach and mainly rely on Itô calculus. In Lemma \ref{measure_re:lemma_stochastic_objects_III}, however, this approach becomes computationally infeasible and we also use multiple stochastic integrals (see \cite{Nualart06} or Section \ref{section:multiple_stochastic_integrals}).

\begin{lemma}\label{measure_re:lemma_products}
Let \( S_N \) be the symmetric group on \( \{ 1,\hdots,N\} \) and let \( \Wt[n] \) be as in \eqref{sc:eq_WTt}. Then, we have for all \( n_1,n_2,n_3,n_4 \in \Z^3 \) that 
\begin{align}
\W{t}{n_1}&= \int_0^t \, \d \W{t_1}{n_1} \label{measure_re:eq_products_1} \allowdisplaybreaks[1]\\
\W{t}{n_1} \W{t}{n_2}&= \sum_{\pi \in S_2} \int_0^t \int_0^{t_1} \, \d \W{t_2}{n_{\pi(2)}} \d \W{t_1}{n_{\pi(1)}} + \delta_{n_1+n_2=0} \frac{\rhoT{t}(n_1)^2}{\langle n_1\rangle^2}, \label{measure_re:eq_products_2}  \allowdisplaybreaks[1]\\
\W{t}{n_1} \W{t}{n_2} \W{t}{n_3} & = \sum_{\pi \in S_3} \int_0^t \int_0^{t_1} \int_0^{t_2}  \, \d \W{t_3}{n_{\pi(3)}} \d \W{t_2}{n_{\pi(2)}}  \d \W{t_1}{n_{\pi(1)}} \label{measure_re:eq_products_3}\\
&+ \frac{1}{2} \sum_{\pi\in S_3} \delta_{n_{\pi(1)}+n_{\pi(2)}=0} \frac{\rhoT{t}(n_{\pi(1)})^2}{\langle n_{\pi(1)}\rangle^2} \W{t}{n_{\pi(3)}}, \notag  \allowdisplaybreaks[1]\\
\W{t}{n_1} \W{t}{n_2} \W{t}{n_3} \W{t}{n_4} &= \sum_{\pi \in S_4} \int_0^t \int_0^{t_1} \int_0^{t_2}  \int_0^{t_3} \, \d \W{t_4}{n_{\pi(4)}} \d \W{t_3}{n_{\pi(3)}} \d \W{t_2}{n_{\pi(2)}}  \d \W{t_1}{n_{\pi(1)}}\label{measure_re:eq_products_4} \\
&+ \frac{1}{4} \sum_{\pi\in S_4} \delta_{n_{\pi(1)}+n_{\pi(2)}=0} \frac{\rhoT{t}(n_{\pi(1)})^2}{\langle n_{\pi(1)} \rangle^2} \W{t}{n_{\pi(3)}} \W{t}{n_{\pi(4)}} \notag \\
&-\frac{1}{8} \sum_{\pi \in S_4} \delta_{n_{\pi(1)}+n_{\pi(2)}=n_{\pi(3)}+n_{\pi(4)}=0} \frac{\rhoT{t}(n_{\pi(1)})^2}{\langle n_{\pi(1)} \rangle^2}\frac{\rhoT{t}(n_{\pi(3)})^2}{\langle n_{\pi(3)} \rangle^2}. \notag
\end{align}
\end{lemma}
The integrals in \eqref{measure_re:eq_products_1}-\eqref{measure_re:eq_products_4} are iterated Itô integrals. This lemma is related to the product formula for multiple stochastic integrals, see e.g. \cite[Proposition 1.1.3]{Nualart06}. 

\begin{proof}
The first equation \eqref{measure_re:eq_products_1} follows from the definition of the Itô derivative \( \d W_t^{n} \). \\
The second equation \eqref{measure_re:eq_products_2} follows from Itô's product formula. Indeed, we have that  
\begin{align*}
\W{t}{n_1} \W{t}{n_2} &=\int_0^t \W{s}{n_2} \d \W{s}{n_1}+ \int_0^t \W{s}{n_1} \d \W{s}{n_2}  + \int_0^t \d \langle \W{}{n_1}, \W{}{n_2} \rangle_s \\
&=  \int_0^t \Big( \int_0^s  \d \W{\tau}{n_2} \Big) \d \W{s}{n_1} + \int_0^t \Big( \int_0^s  \d \W{\tau}{n_1} \Big) \d \W{s}{n_2}  + \delta_{n_1+n_2 =0} \int_0^t  \frac{\sigmaT{s}(n_1)^2}{\langle n_1 \rangle^2} \ds \\
&=  \sum_{\pi \in S_2} \int_0^t \int_0^{t_1} \, \d \W{t_2}{n_{\pi(2)}}  \d \W{t_1}{n_{\pi(1)}} + \delta_{n_1+n_2=0} \frac{\rhoT{t}(n_1)^2}{\langle n_1\rangle^2}. 
\end{align*}
The third equation \eqref{measure_re:eq_products_3} follows from Itô's formula and the second equation \eqref{measure_re:eq_products_2}. Using Itô's formula, we have that 
\begin{align*}
&\W{t}{n_1} \W{t}{n_2} \W{t}{n_3} \\
&= \frac{1}{2} \sum_{\pi \in S_3} \int_0^t \W{s}{n_{\pi(3)}} \W{s}{n_{\pi(2)}} \d \W{s}{n_{\pi(1)}} + \frac{1}{2} \sum_{\pi\in S_3}  \int_0^t \W{s}{n_{\pi(3)}} \d \langle \W{}{n_{\pi(2)}},  \W{}{n_{\pi(1)}} \rangle_s . 
\end{align*}
The easiest way to keep track of the pre-factors throughout the proof is to compare the number of terms of each type and the cardinality of the symmetric group. In the formula above, we have three terms of each type and the cardinality \( \# S_3 = 6 \), so we need the pre-factor \( 1/2 \). By inserting the second equation \eqref{measure_re:eq_products_2} and our expression for the cross-variation, we obtain
\begin{align*}
&\W{t}{n_1} \W{t}{n_2} \W{t}{n_3} \\
&= \sum_{\pi \in S_3} \int_0^t \int_0^{t_1} \int_0^{t_2}  \, \d \W{t_3}{n_{\pi(3)}} \d \W{t_2}{n_{\pi(2)}}  \d \W{t_1}{n_{\pi(1)}} 
+	\frac{1}{2} \sum_{\pi\in S_3}	\delta_{n_{\pi(3)}+n_{\pi(2)}=0} \int_0^t \frac{\rhoT{s}(n_{\pi(2)})^2}{\langle n_{\pi(2)} \rangle^2} \d \W{s}{n_{\pi(1)}} \\
&+ \frac{1}{2} \sum_{\pi\in S_3} \delta_{n_{\pi(1)}+n_{\pi(2)}=0}   \int_0^t  \frac{ \sigmaT{s}(n_{\pi(1)})^2}{\langle n_{\pi(1)}\rangle^2} \W{s}{n_{\pi(3)}} \ds  \allowdisplaybreaks[1]\\
&= \sum_{\pi \in S_3} \int_0^t \int_0^{t_1} \int_0^{t_2}  \, \d \W{t_3}{n_{\pi(3)}} \d \W{t_2}{n_{\pi(2)}}  \d \W{t_1}{n_{\pi(1)}}  \\
&+ \frac{1}{2} \sum_{\pi \in S_3} \delta_{n_{\pi(1)}+n_{\pi(2)}=0} \int_0^t \bigg(  \frac{ \sigmaT{s}(n_{\pi(1)})^2}{\langle n_{\pi(1)}\rangle^2} \W{s}{n_{\pi(3)}} \ds + \frac{ \rhoT{s}(n_{\pi(1)})^2}{\langle n_{\pi(1)}\rangle^2} \d \W{s}{n_{\pi(3)}}\bigg)\allowdisplaybreaks[1]\\
&= \sum_{\pi \in S_3} \int_0^t \int_0^{t_1} \int_0^{t_2}  \, \d \W{t_3}{n_{\pi(3)}} \d \W{t_2}{n_{\pi(2)}}  \d \W{t_1}{n_{\pi(1)}}  
+ \frac{1}{2} \sum_{\pi\in S_3} \delta_{n_{\pi(1)}+n_{\pi(2)}=0} \frac{\rhoT{t}(n_{\pi(1)})^2}{\langle n_{\pi(1)}\rangle^2} \W{t}{n_{\pi(3)}}. 
\end{align*}
For the second equality, we also used the permutation invariance of any sum over \( \pi \in S_3 \). This completes the proof of the third equation \eqref{measure_re:eq_products_3}. \\
We now prove the fourth and final equation \eqref{measure_re:eq_products_4}. The argument differs from the proof of the third equation only in its notational complexity. Using Itô's formula and the third equation \eqref{measure_re:eq_products_3}, we obtain that 
\begin{align*}
&\W{t}{n_1} \W{t}{n_2} \W{t}{n_3} \W{t}{n_4} \\
&= \frac{1}{6} \sum_{\pi \in S_4} \int_0^t \W{s}{n_{\pi(4)}} \W{s}{n_{\pi(3)}} \W{s}{n_{\pi(2)}} \d \W{s}{n_{\pi(1)}}+ \frac{1}{4} \sum_{\pi \in S_4} \int_0^t \W{s}{n_{\pi(4)}} \W{s}{n_{\pi(3)}} \d \langle \W{}{n_{\pi(2)}} , \W{}{n_{\pi(1)}} \rangle_s  \\
&= \sum_{\pi \in S_4} \int_0^t \int_0^{t_1} \int_0^{t_2}  \int_0^{t_3} \, \d \W{t_4}{n_{\pi(4)}}  \d \W{t_3}{n_{\pi(3)}} \d \W{t_2}{n_{\pi(2)}}  \d \W{t_1}{n_{\pi(1)}} \\
&+ \frac{1}{2} \sum_{\pi \in S_4} \frac{\delta_{n_{\pi(1)}+n_{\pi(2)}=0}}{\langle n_{\pi(1)} \rangle^2} \int_0^t \rhoT{s}(n_{\pi(1)})^2 \W{s}{n_{\pi(4)}} \d \W{s}{n_{\pi(3)}} + \frac{1}{4}  \sum_{\pi \in S_4} \int_0^t \sigmaT{s}(n_{\pi(1)})^2 \W{s}{n_{\pi(4)}} \W{s}{n_{\pi(3)}} \ds \\
&= \sum_{\pi \in S_4} \int_0^t \int_0^{t_1} \int_0^{t_2}  \int_0^{t_3} \, \d \W{t_4}{n_{\pi(4)}}  \d \W{t_3}{n_{\pi(3)}} \d \W{t_2}{n_{\pi(2)}}  \d \W{t_1}{n_{\pi(1)}} + \frac{1}{4} \sum_{\pi \in S_4} \bigg[ \frac{\delta_{n_{\pi(1)}+n_{\pi(2)}=0}}{\langle n_{\pi(1)} \rangle^2}  ~ \times \\
&\int_0^t \bigg( \sigmaT{s}(n_{\pi(1)})^2 \W{s}{n_{\pi(4)}} \W{s}{n_{\pi(3)}} \ds + \rhoT{s}(n_{\pi(1)})^2 \W{s}{n_{\pi(4)}} \d \W{s}{n_{\pi(3)}}+ \rhoT{s}(n_{\pi(1)})^2\W{s}{n_{\pi(3)}} \d \W{s}{n_{\pi(4)}} \bigg) \bigg].
\end{align*}
Using Itô's formula, we obtain that 
\begin{equation*}
\begin{aligned}
&\int_0^t \bigg( \sigmaT{s}(n_{\pi(1)})^2 \W{s}{n_{\pi(4)}} \W{s}{n_{\pi(3)}} \ds + \rhoT{s}(n_{\pi(1)})^2 \W{s}{n_{\pi(4)}} \d \W{s}{n_{\pi(3)}} + \rhoT{s}(n_{\pi(1)})^2 \W{s}{n_{\pi(3)}} \d \W{s}{n_{\pi(4)}} \bigg) \allowdisplaybreaks[4] \\
&= \rhoT{t}(n_{\pi(1)})^2 \W{t}{n_{\pi(3)}} \W{t}{n_{\pi(4)}} -  \delta_{n_{\pi(3)}+n_{\pi(4)}=0} \int_0^t \rhoT{s}(n_{\pi(1)})^2 \frac{\sigmaT{s}(n_{\pi(3)})^2}{\langle n_{\pi(3)}\rangle^2} \ds . 
\end{aligned}
\end{equation*}
The total contribution of the second summand is 
\begin{align*}
&- \frac{1}{4}  \sum_{\pi \in S_4} \frac{\delta_{n_{\pi(1)}+n_{\pi(2)}=n_{\pi(3)}+n_{\pi(4)}=0}}{\langle n_{\pi(1)} \rangle^2 \langle n_{\pi(3)} \rangle^2} \int_0^t \rhoT{s}(n_{\pi(1)})^2 \sigmaT{s}(n_{\pi(3)})^2 \ds \allowdisplaybreaks[4]\\
&=-\frac{1}{8}  \sum_{\pi \in S_4} \frac{\delta_{n_{\pi(1)}+n_{\pi(2)}=n_{\pi(3)}+n_{\pi(4)}=0}}{\langle n_{\pi(1)} \rangle^2 \langle n_{\pi(3)} \rangle^2} \int_0^t \Big( \rhoT{s}(n_{\pi(1)})^2 \sigmaT{s}(n_{\pi(3)})^2 + \sigmaT{s}(n_{\pi(1)})^2 \rhoT{s}(n_{\pi(3)})^2 \Big) \ds\allowdisplaybreaks[4] \\
&=-\frac{1}{8} \sum_{\pi \in S_4} \delta_{n_{\pi(1)}+n_{\pi(2)}=n_{\pi(3)}+n_{\pi(4)}=0} \frac{\rhoT{t}(n_{\pi(1)})^2}{\langle n_{\pi(1)} \rangle^2}\frac{\rhoT{t}(n_{\pi(3)})^2}{\langle n_{\pi(3)} \rangle^2}. 
\end{align*}
This completes the proof of the fourth equation \eqref{measure_re:eq_products_4}. 
\end{proof}

\begin{definition}[Renormalization]\label{measure_re:definition_renormalization}
We define the renormalization constants \( \aT, \bT \in \R \) and the multiplier \( \MT \colon L^2(\T^3) \rightarrow L^2(\T^3) \) by
\begin{equation*}
\aT \defe \sum_{n\in \Z^3} \frac{\rhoT{t}(n)^2}{\langle n\rangle^2}, \qquad \bT \defe \sum_{n_1,n_2\in \Z^3} \frac{\widehat{V}(n_1+n_2) \rhoT{t}(n_1)^2 \rhoT{t}(n_2)^2}{\langle n_1 \rangle^2 \langle n_2 \rangle^2}
\end{equation*}
and 
\begin{equation*}
 \widehat{\MT f}(n) \defe \Big( \sum_{m \in \Z^3} \widehat{V}(n+m)\frac{\rhoT{t}(m)^2}{\langle m \rangle^2} \Big) \widehat{f}(n). 
\end{equation*}
Using this notation, we set 
\begin{align}
\lcol f^2 \rcol  &\overset{def}{=} f^2 -\aT,  \label{measure_re:eq_renormalized_quadratic}\\
\lcol (V * f^2) f \rcol &\overset{def}{=} (V* f^2) f - \aT  \widehat{V}(0)  f - 2 \MT f,  \label{measure_re:eq_renormalized_cubic}\\
\lcol (V*f^2) f^2 \rcol &\overset{def}{=} (V*f^2) f^2 - \aT V * f^2 - \aT \widehat{V}(0)  f^2  - 4 (\MT f) f + (\aT)^2 \widehat{V}(0)  + 2 \bT.\label{measure_re:eq_renormalized_quartic}
\end{align}
\end{definition}

\begin{remark}
As is clear from the definition, the renormalized powers in \eqref{measure_re:eq_renormalized_quadratic}, \eqref{measure_re:eq_renormalized_cubic}, and \eqref{measure_re:eq_renormalized_quartic} depend on the regularization parameter \( t \). This dependence will always be clear from context and we thus do not reflect it in our notation. 
\end{remark}

\begin{definition}[Renormalization of the dynamics]\label{measure_re:definition_renormalization_dynamics}
For any \( N \geq 1 \), we define 
\begin{equation}
a_N \defe a^N_\infty = a^\infty_N , \quad b_N \defe b^N_\infty =  b^\infty_N, \quad \text{and} \quad \mathcal{M}_N \defe \mathcal{M}^N_\infty = \mathcal{M}^\infty_N. 
\end{equation}
\end{definition}
Throughout most of the paper, we will only work with  the renormalization constants from Definition \ref{measure_re:definition_renormalization}, which contain two finite parameters. The renormalization constants in Definition \ref{measure_re:definition_renormalization_dynamics} will be more important in the second part of this series. 

\begin{proposition}[Stochastic integral representation of renormalized powers]\label{measure_re:proposition_renormalization}
With \( n_{12}, n_{123},\) and \( n_{1234} \) defined as in \eqref{eq_n12}, we have that 
\begin{align}
\lcol (\Wt)^2 \rcol &=  2  \sum_{\substack{n_1,n_2 \in \Z^3 } }  e^{i \langle n_{12}, x \rangle} \int_0^t \int_0^{t_1}  \, \d \W{t_2}{n_{2}}  \d \W{t_1}{n_{1}} \label{measure_re:eq_renormalization_2}  \allowdisplaybreaks[4]\\
\lcol (V* (\Wt)^2) \Wt \rcol  &=  \sum_{\substack{n_1,n_2,n_3 \in \Z^3\\\pi \in S_3}} \hspace{-2ex} \widehat{V}(n_{\pi(1)}+n_{\pi(2)}) e^{i \langle n_{123} ,x \rangle} \int_0^t \int_0^{t_1} \int_0^{t_2}  \, \d \W{t_3}{n_3} \d \W{t_2}{n_2}  \d \W{t_1}{n_1}   \label{measure_re:eq_renormalization_3} \allowdisplaybreaks[4]\\
\lcol (V* (\Wt)^2) (\Wt)^2 \rcol &= \sum_{\substack{n_1,n_2, n_3, n_4 \in \Z^3 \\ \pi \in S_4}} \bigg[ \widehat{V}(n_{\pi(1)}+n_{\pi(2)}) e^{i \langle n_{1234},x \rangle} \label{measure_re:eq_renormalization_4}\\
&  \times \int_0^t \int_0^{t_1} \int_0^{t_2}  \int_0^{t_3} \, \d \W{t_4}{n_4}  \d \W{t_3}{n_3} \d \W{t_2}{n_2}  \d \W{t_1}{n_1} \bigg]. \notag \allowdisplaybreaks[1]
\end{align}
Furthermore, it holds that 
\begin{equation}\label{measure_re:eq_potential_derivative}
\int_{\T^3} \lcol (V* (\Wt)^2) (\Wt)^2 \rcol \dx  = 4 \int_0^t \int_{\T^3}  \lcol (V* (\Ws)^2) \Ws \rcol   \d \Ws  . 
\end{equation}
\end{proposition}
\begin{remark}
The "lower-order" terms in Definition \ref{measure_re:definition_renormalization} were chosen precisely to obtain the result in Proposition \ref{measure_re:proposition_renormalization}. The renormalized powers of \( \Wt  \) can be represented solely using iterated stochastic integrals, which have many desirable properties. \\
\end{remark}

Proposition \ref{measure_re:proposition_renormalization} essentially follows from Lemma \ref{measure_re:lemma_products}, Definition \ref{measure_re:definition_renormalization}, and a tedious calculation. For the sake of completeness, however, we provide full details. 
\begin{proof}
We first prove \eqref{measure_re:eq_renormalization_2}.  Using \eqref{measure_re:eq_products_2}, we have that 
\begin{align*}
(\Wt)^2 &=  \sum_{n_1,n_2 \in \Z^3} e^{i\langle n_1+ n_2 , x\rangle} \W{t}{n_1} \W{t}{n_2}  \allowdisplaybreaks[1]\\
&= \sum_{\pi \in S_2} \sum_{n_1,n_2 \in \Z^3}  e^{i\langle n_1+ n_2 , x\rangle}  \int_0^t \int_0^{t_1} \, \d \W{t_2}{n_{\pi(2)}}  \d \W{t_1}{n_{\pi(1)}} + \sum_{n_1,n_2 \in \Z^3} \delta_{n_1+n_2=0} \frac{\rhoT{t}(n_1)^2}{\langle n_1\rangle^2} e^{i\langle n_1+n_2, x \rangle}  \allowdisplaybreaks[1]\\
&= \sum_{\pi \in S_2} \sum_{n_1,n_2 \in \Z^3}  e^{i\langle n_1+ n_2 , x\rangle}  \int_0^t \int_0^{t_1} \, \d \W{t_2}{n_{\pi(2)}}  \d \W{t_1}{n_{\pi(1)}} + \aT. 
\end{align*}
By subtracting \( \aT \) from both sides and symmetrizing, this leads to the desired identity. \\
We now turn to the proof of \eqref{measure_re:eq_renormalization_3}. From \eqref{measure_re:eq_products_3}, we obtain that 
\begin{align*}
(V*(\Wt)^2) \Wt  &= \sum_{n_1,n_2,n_3\in \Z^3} \widehat{V}(n_1+n_2) e^{i\langle n_{123} , x \rangle} \W{t}{n_1} \W{t}{n_2} \W{t}{n_3}  \allowdisplaybreaks[1]\\
&= \sum_{\pi \in S_3} \sum_{n_1,n_2,n_3 \in \Z^3} \widehat{V}(n_1+n_2) e^{i \langle n_{123} , x \rangle}  \int_0^t \int_0^{t_1} \int_0^{t_2}  \, \d \W{t_3}{n_{\pi(3)}} \d \W{t_2}{n_{\pi(2)}}  \d \W{t_1}{n_{\pi(1)}} \\
&+ \frac{1}{2} \sum_{\pi\in S_3} \sum_{n_1,n_2,n_3 \in \Z^3} \widehat{V}(n_1+n_2)  e^{i\langle n_{123} ,x \rangle} \delta_{n_{\pi(1)}+n_{\pi(2)}=0} \frac{\rhoT{t}(n_{\pi(1)})^2}{\langle n_{\pi(1)}\rangle^2} \W{t}{n_{\pi(3)}}, \notag  \allowdisplaybreaks[1] \\
&= \sum_{\pi \in S_3} \sum_{n_1,n_2,n_3 \in \Z^3} \widehat{V}(n_1+n_2) e^{i \langle n_{123}, x \rangle}  \int_0^t \int_0^{t_1} \int_0^{t_2}  \, \d \W{t_3}{n_{\pi(3)}} \d \W{t_2}{n_{\pi(2)}}  \d \W{t_1}{n_{\pi(1)}} \\
&+ \sum_{n_1, n_3\in \Z^3} \widehat{V}(0) e^{i\langle n_3 , x \rangle} \frac{\rhoT{t}(n_1)^2}{\langle n_1 \rangle^2} \W{t}{n_3} + 2 \sum_{n_1,n_3\in \Z^3} \widehat{V}(n_1+n_3) e^{i\langle n_3, x \rangle} \frac{\rhoT{t}(n_1)^2}{\langle n_1\rangle^2} \W{t}{n_3} \allowdisplaybreaks[1] \\
&= \sum_{\pi \in S_3} \sum_{n_1,n_2,n_3 \in \Z^3} \widehat{V}(n_1+n_2) e^{i \langle n_{123} , x \rangle}  \int_0^t \int_0^{t_1} \int_0^{t_2}  \, \d \W{t_3}{n_{\pi(3)}} \d \W{t_2}{n_{\pi(2)}}  \d \W{t_1}{n_{\pi(1)}} \\ 
&+ \aT \widehat{V}(0) \Wt  + 2 \MT \Wt  .
\end{align*}
After symmetrizing and comparing with Definition \ref{measure_re:definition_renormalization}, this leads to the desired identity. Next, we prove the identity \eqref{measure_re:eq_renormalization_4}. Using \eqref{measure_re:eq_products_4}, we have that 
\begin{align}
&(V*(\Wt)^2) (\Wt)^2 \notag \\
&= \sum_{n_1,n_2,n_3,n_4\in \Z^3} \widehat{V}(n_1+n_2) e^{i \langle n_{1234},x \rangle} \W{t}{n_1} \W{t}{n_2}\W{t}{n_3} \W{t}{n_4}  \notag \allowdisplaybreaks[1]\\
&=  \sum_{ \substack{n_1,n_2,n_3,n_4\in \Z^3 \\ \pi \in S_4}}  \widehat{V}(n_1+n_2) e^{i \langle n_{1234},x \rangle} \int_0^t \int_0^{t_1} \int_0^{t_2}  \int_0^{t_3} \, \d \W{t_4}{n_{\pi(4)}}  \d \W{t_3}{n_{\pi(3)}} \d \W{t_2}{n_{\pi(2)}}  \d \W{t_1}{n_{\pi(1)}}   \notag \\
&+ \frac{1}{4} \sum_{ \substack{n_1,n_2,n_3,n_4\in \Z^3 \\ \pi \in S_4}}  \widehat{V}(n_1+n_2) e^{i \langle n_{1234} ,x \rangle} \delta_{n_{\pi(1)}+n_{\pi(2)}=0} \frac{\rhoT{t}(n_{\pi(1)})^2}{\langle n_{\pi(1)} \rangle^2} \W{t}{n_{\pi(3)}} \W{t}{n_{\pi(4)}}  \label{measure_re:eq_renormalization_p1}\\
&-\frac{1}{8}\sum_{ \substack{n_1,n_2,n_3,n_4\in \Z^3 \\ \pi \in S_4}}  \widehat{V}(n_1+n_2) e^{i \langle n_{1234} ,x \rangle} \delta_{n_{\pi(1)}+n_{\pi(2)}=n_{\pi(3)}+n_{\pi(4)}=0} \frac{\rhoT{t}(n_{\pi(1)})^2}{\langle n_{\pi(1)} \rangle^2}\frac{\rhoT{t}(n_{\pi(3)})^2}{\langle n_{\pi(3)} \rangle^2}. \label{measure_re:eq_renormalization_p2}
\end{align}
It remains to simplify \eqref{measure_re:eq_renormalization_p1} and \eqref{measure_re:eq_renormalization_p2}. Regarding \eqref{measure_re:eq_renormalization_p1}, we have that 
\begin{align*}
&\frac{1}{4} \sum_{ \substack{n_1,n_2,n_3,n_4\in \Z^3 \\ \pi \in S_4}}  \widehat{V}(n_1+n_2) e^{i \langle n_{1234} ,x \rangle} \delta_{n_{\pi(1)}+n_{\pi(2)}=0} \frac{\rhoT{t}(n_{\pi(1)})^2}{\langle n_{\pi(1)} \rangle^2} \W{t}{n_{\pi(3)}} \W{t}{n_{\pi(4)}} \allowdisplaybreaks[1] \\
&= \sum_{n_1,n_2,n_3 \in \Z^3} \widehat{V}(n_1+n_2) \frac{\rhoT{t}(n_3)^2}{\langle n_3\rangle^2} e^{i\langle n_1+n_2,x\rangle} \W{t}{n_1} \W{t}{n_2} \\
&+ 4 \sum_{n_1,n_2,n_3 \in \Z^3} \widehat{V}(n_1+n_2) \frac{\rhoT{t}(n_2)^2}{\langle n_2\rangle^2} e^{i\langle n_1 + n_3 ,x \rangle} \W{t}{n_1} \W{t}{n_3} \\
&+ \sum_{n_1,n_3,n_4 \in \Z^3} \widehat{V}(0) \frac{\rhoT{t}(n_1)^2}{\langle n_1\rangle^2} e^{i\langle n_3+n_4,x\rangle} \W{t}{n_3} \W{t}{n_4} \allowdisplaybreaks[1] \\
&= \aT V* (\Wt)^2 + 4 ( \MT \Wt ) \Wt  + \aT \widehat{V}(0) (\Wt)^2.
\end{align*}
Regarding \eqref{measure_re:eq_renormalization_p2}, we note that 
\begin{align*}
&-\frac{1}{8} \sum_{\pi \in S_4}\sum_{n_1,n_2,n_3,n_4\in \Z^3}   \widehat{V}(n_1+n_2) e^{i \langle n_{1234} ,x \rangle} \delta_{n_{\pi(1)}+n_{\pi(2)}=n_{\pi(3)}+n_{\pi(4)}=0} \frac{\rhoT{t}(n_{\pi(1)})^2}{\langle n_{\pi(1)} \rangle^2}\frac{\rhoT{t}(n_{\pi(3)})^2}{\langle n_{\pi(3)} \rangle^2} \\
&= - \sum_{n_1,n_3 \in \Z^3} \widehat{V}(0) \frac{\rhoT{t}(n_1)^2 \rhoT{t}(n_3)^2}{\langle n_1 \rangle^2 \langle n_3 \rangle^2} - 2 \sum_{n_1,n_2 \in \Z^3} \frac{ \widehat{V}(n_1+n_2) \rhoT{t}(n_1)^2) \rhoT{t}(n_2)^2}{\langle n_1 \rangle^2 \langle n_2 \rangle^2} \\
&= - \widehat{V}(0) (\aT)^2 - 2 \bT. 
\end{align*}
After symmetrizing, this completes the proof of \eqref{measure_re:eq_renormalization_4}. \\
Finally, it remains to prove \eqref{measure_re:eq_potential_derivative}. Since \( V \) is real-valued and even, we have that \( \widehat{V}(n) = \overline{\widehat{V}(n)} = \widehat{V}(-n) \). As long as \( n_{1234}=0 \), this implies
\begin{equation}\label{measure_re:eq_renormalization_p3}
\sum_{\pi \in S_4} \widehat{V}(n_{\pi(1)}+n_{\pi(2)}) = 4 \sum_{\pi \in S_3} \widehat{V}(n_{\pi(1)}+n_{\pi(2)}). 
\end{equation}
Using \eqref{measure_re:eq_renormalization_p3}, \eqref{measure_re:eq_potential_derivative} follows after inserting \eqref{measure_re:eq_renormalization_3} and \eqref{measure_re:eq_renormalization_4} into the two sides of the identity. 
\end{proof}

Like the monomials and Hermite polynomials (further discussed below), the generalized and renormalized powers in  Definition \ref{measure_re:definition_renormalization} satisfy a binomial formula. 

\begin{lemma}[Binomial formula]\label{measure_re:lemma_binomial}
For any \( f \in H^1(\T^3) \), we have the binomial formulas 
\begin{equation}\label{measure_re:eq_cubic_binomial}
\begin{aligned}
&\lcol( V* (\Wt +f)^2) (\Wt + f) \rcol \\
&=  \lcol (V*(\Wt)^2) \Wt  \rcol + (V* \lcol (\Wt)^2\rcol) f + 2 \big[ (V* (\Wt  f)) \Wt  - \MT f \big] \\
&+ 2 (V*(\Wt  f))f + (V*f^2) \Wt  + (V*f^2) f
\end{aligned}
\end{equation}
and 
\begin{equation}\label{measure_re:eq_quartic_binomial}
\begin{aligned}
&\int_{\T^3} \lcol (V* (\Wt  +f )^2)(\Wt +f )^2 \rcol \dx \\
&= \int_{\T^3} \lcol (V*(\Wt)^2) (\Wt)^2 \rcol \dx + 4 \int_{\T^3} \lcol (V*(\Wt)^2) \Wt \rcol f \dx + 2 \int_{\T^3} ( V* \lcol (\Wt)^2\rcol) f^2 \dx  \\
&+ 4 \int_{\T^3} \Big[ (V* (\Wt  f)) \Wt  f - (\MT f ) f \Big] \dx + 4 \int_{\T^3} (V* f^2) f \, \Wt  \dx + \int_{\T^3} (V* f^2) f^2 \dx. 
\end{aligned}
\end{equation}
\end{lemma}

\begin{remark}  Overall, the terms in \eqref{measure_re:eq_quartic_binomial} obey better analytical estimates than their counterparts for the \( \Phi^4_3\)-model in \cite{BG20}. However, their algebraic structure is more complicated. 
The most challenging term is 
\begin{equation*}
 \int_{\T^3} \Big[ (V* (\Wt  f)) \Wt  f - (\MT f ) f \Big] \dx,
\end{equation*}
which requires a delicate random matrix estimate (Section \ref{section:rmt}). 
\end{remark}

\begin{proof}[Proof of Lemma \ref{measure_re:lemma_binomial}:]
This follows from Definition \ref{measure_re:definition_renormalization} and the classical binomial formula. For the quartic binomial formula \eqref{measure_re:eq_quartic_binomial}, we also used the self-adjointness of the convolution with \( V \) and the multiplier \( \MT \). 
\end{proof}

While this is not reflected in our notation, it is clear from  Definition \ref{measure_re:definition_renormalization} that the multiplier \( \MT \) depends linearly on the interaction potential \( V \). In the proof of the random matrix estimate (Proposition \ref{measure_rmt:proposition_estimate}), we will need to further decompose \( \MT \), both with respect to the interaction potential \( V \) and dyadic frequency blocks. We introduce the notation corresponding to this decomposition in the next definition.

\begin{definition}\label{measure_re:definition_MT_decomposition}
We let \( \MT[V;N_1,N_2] \) be the Fourier multiplier corresponding to the symbol
\begin{equation}\label{measure_re:eq_MT_decomposition_symbol}
n \mapsto  \sum_{k\in \Z^3} \frac{\widehat{V}(n+k)}{\langle k \rangle^2} \chi_{N_1}(k) \chi_{N_2}(k)  \rhoT{t}(k)^2.
\end{equation}
\end{definition}

In the next definition, we define our last renormalization of a stochastic object. 
\begin{definition}\label{measure_re:definition_renormalization_translation}
We define the correlation function on \( \T^3 \) by 
\begin{equation}
\Cort(y) \defe \sum_{k\in \Z^3} \frac{\chi_{N_1}(k) \chi_{N_2}(k)}{\langle k\rangle^2} \rhoT{t}(k)^2  e^{i\langle k,y \rangle}. 
\end{equation}
We further define
\begin{equation}
\lcol (\tau_y P_{N_1} \Wt) P_{N_2} \Wt \rcol(x) \defe  (\tau_y P_{N_1} \Wt)(x) P_{N_2} \Wt(x) - \Cort(y). 
\end{equation}
Here, \( \tau_y \) denotes the translation operator \( \tau_y f(x) = f(x-y) \). 
\end{definition}

The next lemma relates the multiplier and correlation function from Definition \ref{measure_re:definition_MT_decomposition} and Definition \ref{measure_re:definition_renormalization_translation}, respectively. 

\begin{lemma}[Physical space representation of \( \MT \)]\label{measure_re:lemma_physical_MT}
For any \( f \in C^\infty_x(\T^3) \), we have that 
\begin{equation}
\MT[V;N_1,N_2] f = \big( \Cort V\big) * f. 
\end{equation}
\end{lemma}

\begin{proof}
By definition of the multiplier \( \MT[V;N_1,N_2] \) and since 
\begin{equation}\label{measure_re:eq_proof_physical}
k \mapsto  \frac{1}{\langle k \rangle^2} \chi_{N_1}(k) \chi_{N_2}(k)  \rhoT{t}(k)^2
\end{equation}
is even, the symbol in \eqref{measure_re:eq_MT_decomposition_symbol} is the convolution of \( \widehat{V} \) with \eqref{measure_re:eq_proof_physical}. As a result, the sequence  ${n\mapsto \MT[V;N_1,N_2](n)}$ has the inverse Fourier transform is given by 
\begin{align*}
 \Big ( \sum_{k\in \Z^3} \frac{\chi_{N_1}(k) \chi_{N_2}(k) }{\langle k \rangle^2}  \rhoT{t}(k)^2 e^{i\langle k , x \rangle} \Big) V(x) = \Cort(x) V(x). 
\end{align*}
\end{proof}
In Lemma \ref{measure_re:lemma_products}, Proposition \ref{measure_re:proposition_renormalization}, Lemma \ref{measure_re:lemma_binomial}, and Lemma \ref{measure_re:lemma_physical_MT}, we have dealt with the algebraic structure of stochastic objects. We now move from algebraic aspects towards analytic estimates. In the following lemmas, we show that several stochastic objects are well-defined and study their regularities.

\begin{lemma}[Stochastic objects I]\label{measure_re:lemma_stochastic_objects_I}
For every \( p \geq 1 \), \( \epsilon >0 \), and every \( 0 < \gamma < \min(\beta,1)\), we have that 
\begin{align}
\sup_{t\geq 0} \Big( \E \Big[ \| \lcol (\Wt)^2 \rcol \|_{\cC^{-1-\epsilon}_x(\T^3)}^p\Big] \Big)^{\frac{1}{p}} & \lesssim p, \label{measure_re:eq_so_1}\\
\sup_{t\geq 0} \Big( \E \Big[ \| V*  \lcol (\Wt)^2 \rcol \|_{\cC^{-1+\beta-\epsilon}_x(\T^3)}^p\Big] \Big)^{\frac{1}{p}} & \lesssim p, \label{measure_re:eq_so_2}\\
\sup_{t\geq 0} \Big( \E \Big[ \| \lcol (V* (\Wt)^2) \Wt  \rcol \|_{\cC^{-\frac{3}{2}+\gamma}_x(\T^3)}^p\Big] \Big)^{\frac{1}{p}} & \lesssim p^{\frac{3}{2}}. \label{measure_re:eq_so_3}
\end{align}
Furthermore, as \( t\rightarrow \infty\) and/or \( T \rightarrow \infty\), the stochastic objects \( \lcol (\Wt)^2 \rcol  \), \( V* \lcol (\Wt)^2 \rcol  \), and \( {\lcol (V* (\Wt)^2) \Wt  \rcol} \) converge in their respective spaces indicated by \eqref{measure_re:eq_so_1}-\eqref{measure_re:eq_so_3}. 
\end{lemma}

\begin{remark}
The statement and proof of Lemma \ref{measure_re:lemma_stochastic_objects_I} are standard and the respective regularities can be deduced by simple ``power-counting''. Nevertheless, we present the proof to familiarize the reader with our set-up and as a warm-up for Lemma \ref{measure_re:lemma_stochastic_objects_III} below. 
\end{remark}

\begin{proof}
The first step in the proofs of \eqref{measure_re:eq_so_1}-\eqref{measure_re:eq_so_3} is a reduction to an estimate in \( L^2(\Omega\times \T^3) \) using Gaussian hypercontractivity. We provide the full details of this step for \eqref{measure_re:eq_so_1}, but will omit similar details in the remaining estimates \eqref{measure_re:eq_so_2}-\eqref{measure_re:eq_so_3}. \\
Let \( N \geq 1 \) and let \( q=q(\epsilon)\geq 1 \) be sufficiently large. By using Hölder's inequality in \( \omega \in \Omega \), it suffices to prove the estimates for \( p \geq q \). Using Bernstein's inequality and Minkowski's integral inequality, we obtain
\begin{align*}
\| P_{N} \lcol (\Wt)^2 \rcol \|_{L_\omega^p \cC_x^{-1-\epsilon}(\Omega \times \T^3)} 
\lesssim N^{-1-\frac{\epsilon}{2}} \| P_{N} \lcol (\Wt)^2 \rcol \|_{L_\omega^p L_x^q(\Omega \times \T^3)} 
\leq N^{-1-\frac{\epsilon}{2}} \| P_{N} \lcol (\Wt)^2 \rcol \|_{L_x^q L_\omega^p(\T^3 \times \Omega )}. 
\end{align*}
By Gaussian hypercontractivity (Lemma \ref{prelim:lemma_hypercontractivity}), we obtain that 
\begin{equation*}
N^{-1-\frac{\epsilon}{2}} \| P_{N} \lcol (\Wt)^2 \rcol \|_{L_x^q L_\omega^p(\T^3 \times \Omega )} \lesssim N^{-1-\frac{\epsilon}{2}} p \| P_{N} \lcol (\Wt)^2 \rcol \|_{L_x^q L_\omega^2(\T^3 \times \Omega )}. 
\end{equation*}
Since the distribution of \( \lcol (\Wt)^2 \rcol \) is translation invariant, the function \( x \mapsto \|\lcol (\Wt)^2\rcol \|_{L_\omega^2(\Omega)} \) is constant. We can then replace \( L_x^q(\T^3) \) by \( L_x^2(\T^3) \) and obtain
\begin{align*}
N^{-1-\frac{\epsilon}{2}} p  \| P_{N} \lcol (\Wt)^2 \rcol \|_{L_x^q L_\omega^2(\T^3 \times \Omega )} &\lesssim N^{-1-\frac{\epsilon}{2}} p  \| P_{N} \lcol (\Wt)^2 \rcol \|_{L_x^2 L_\omega^2(\T^3 \times \Omega )}\\
&\lesssim N^{-\frac{\epsilon}{4}} p \| \lcol (\Wt)^2 \rcol \|_{L_\omega^2 H_x^{-1-\frac{\epsilon}{4}}(\Omega \times \T^3)}. 
\end{align*}
In order to prove  \eqref{measure_re:eq_so_1}, it therefore remains to show uniformly in \( T,t \geq 0 \) that 
\begin{equation}
\| \lcol (\Wt)^2 \rcol \|_{L_\omega^2 H_x^{-1-\epsilon}(\Omega\times \T^3)}^2 \lesssim 1. 
\end{equation}
Using Proposition \ref{measure_re:proposition_renormalization}, the orthogonality of the iterated stochastic integrals, and Itô's isometry, we have that 
\begin{align*}
\| \lcol (\Wt)^2 \rcol \|_{L_\omega^2 H_x^{-1-\epsilon}}^2 
&= 4 \sum_{n\in \Z^3} \frac{1}{\langle n \rangle^{2+2\epsilon}} ~ \E \Big[ \Big(  \sum_{\substack{n_1,n_2\in \Z^3\colon \\ n_1+n_2 = n}} \int_0^t \int_0^{t_1}  \, \d \W{t_2}{n_{2}}  \d \W{t_1}{n_{1}} \Big)^2 \Big]  \\
&\lesssim \sum_{ \substack{n,n_1,n_2\in \Z^3 \\ n_1+n_2 = n}} \frac{1}{\langle n \rangle^{2+2\epsilon} \langle n_1 \rangle^2 \langle n_2 \rangle^2} \rhoT{t}(n_1)^2 \rhoT{t}(n_2)^2 \\
&\lesssim \sum_{n_1,n_2 \in \Z^3} \frac{1}{\langle n_1 +n_2 \rangle^{2+2\epsilon} \langle n_1 \rangle^2 \langle n_2 \rangle^2} \lesssim 1. 
\end{align*}
This completes the proof of \eqref{measure_re:eq_so_1}. The estimate \eqref{measure_re:eq_so_2} can be deduced from the smoothing properties of \( V \) or by repeating the exact same argument. It remains to prove \eqref{measure_re:eq_so_3}, which can be reduced using hypercontractivity (and the room in \( \gamma\)) to the estimate 
\begin{equation*}
 \| \lcol ( V* (\Wt)^2) \Wt  \rcol \|_{L_\omega^2 H_x^{-\frac{3}{2}+\gamma}}^2 \lesssim 1. 
\end{equation*}
Using Proposition \ref{measure_re:proposition_renormalization}, the orthogonality of the iterated stochastic integrals, and Itô's isometry, we have that 
\begin{align*}
 &\| \lcol ( V* (\Wt)^2) \Wt  \rcol \|_{L_\omega^2 H_x^{-\frac{3}{2}+\gamma}}^2\\
&= \sum_{n\in \Z^3} \frac{1}{\langle n \rangle^{3-2\gamma}} ~ \E \Big[ \Big(  \sum_{\pi\in S_3} \sum_{\substack{n_1,n_2,n_3 \in \Z^3\colon\\ n_1+n_2+n_3 = n}} \widehat{V}(n_{\pi(1)}+n_{\pi(2)}) \int_0^t \int_0^{t_1} \int_0^{t_2}  \, \d \W{t_3}{n_3} \d \W{t_2}{n_{2}}  \d \W{t_1}{n_{1}}  \Big)^2 \Big] \\
&\lesssim \sum_{ n_1,n_2,n_3 \in \Z^3} \frac{1}{\langle n_1 + n_2 + n_3 \rangle^{3-2\gamma}} \frac{1}{\langle n_1+n_2 \rangle^{2\beta}}\frac{1}{\langle n_1\rangle^2 \langle n_2 \rangle^2 \langle n_3 \rangle^2}.
\end{align*}
By first summing in \( n_3 \), using that \( 3 - 2 \gamma > 1 \), and then in \( n_1 \) and \(n_2 \), using \( \gamma < \beta \), we obtain 
\begin{align*}
&\sum_{ n_1,n_2,n_3 \in \Z^3} \frac{1}{\langle n_1 + n_2 + n_3 \rangle^{3-2\gamma}} \frac{1}{\langle n_1+n_2 \rangle^{2\beta}}\frac{1}{\langle n_1\rangle^2 \langle n_2 \rangle^2 \langle n_3 \rangle^2} \\
&\lesssim \sum_{n_1,n_2 \in \Z^3} \frac{1}{\langle n_1 + n_2\rangle^{2+ 2 (\beta-\gamma)}} \frac{1}{\langle n_1 \rangle^2 \langle n_2 \rangle^2} \lesssim 1. 
\end{align*}
\end{proof}

We also record the following refinement of \eqref{measure_re:eq_so_3} in Lemma \ref{measure_re:lemma_stochastic_objects_I}, which will be needed in the proof of Lemma \ref{measure_re:lemma_stochastic_objects_III} below.  

\begin{corollary}\label{measure_re:corollary_stochastic_objects_I}
 For every \( 0< \gamma < \min(1,\beta) \) and any \( n \in \Z^3 \), we can control the Fourier coefficients of \( \lcol (V* (\Wt)^2) \Wt  \rcol \) by
\begin{equation} 
\sup_{T,t\geq 0}  \E_{\bP} \Big| \mathcal{F} \Big( \lcol (V* (\Wt)^2) \Wt  \rcol \Big)(n) \Big|^2 \lesssim \langle n \rangle^{-2\gamma}. 
\end{equation}
\end{corollary}
\begin{proof}
Arguing as in the proof of Lemma \ref{measure_re:lemma_stochastic_objects_I}, it suffices to prove that 
\begin{equation}\label{measure_re:eq_so_I_p1}
\sum_{ \substack{n_1,n_2,n_3\in \Z^3\colon \\ n_{123}=n }} \frac{1}{\langle n_{12}\rangle^{2\beta} \langle n_1 \rangle^2 \langle n_2 \rangle^2 \langle n_3 \rangle^2} \lesssim \frac{1}{\langle n \rangle^{2\gamma}}. 
\end{equation}
Indeed, after parametrizing the sum by \( n_1 \) and \( n_3 \), \eqref{measure_re:eq_so_I_p1} follows from 
\begin{align*}
\sum_{ \substack{n_1,n_2,n_3\in \Z^3\colon \\ n_{123}=n }} \frac{1}{\langle n_{12}\rangle^{2\beta} \langle n_1 \rangle^2 \langle n_2 \rangle^2 \langle n_3 \rangle^2} 
&= \sum_{ n_1, n_3 \in \Z^3 } \frac{1}{\langle n-n_3 \rangle^{2\beta} \langle n_1 \rangle^2 \langle n-n_1-n_3 \rangle^2 \langle n_3 \rangle^2} \\
&\lesssim \sum_{ n_3 \in \Z^3} \frac{1}{\langle n-n_3 \rangle^{1+2\beta} \langle n_3 \rangle^2} \\
&\lesssim  \langle n \rangle^{-2\gamma}. 
\end{align*}
\end{proof}

\begin{lemma}[Stochastic objects II]\label{measure_re:lemma_stochastic_objects_II}
For any sufficiently small \( \delta>0\) and any \( N_1,N_2 \geq 1 \), it holds that 
\begin{equation}
\sup_{T,t\geq 0}  \Big ( \E \Big[ \sup_{y\in \T^3} \| \lcol (\tau_y P_{N_1} \Wt)  P_{N_2} \Wt \rcol  \|_{\cC^{-1-\delta}_x(\T^3)}^p \Big] \Big)^{\frac{1}{p}} \lesssim \max(N_1,N_2)^{-\frac{\delta}{10}} p. 
\end{equation}
\end{lemma}
\begin{proof}
Arguing as in the proof of \eqref{measure_re:eq_so_1} in Lemma \ref{measure_re:lemma_stochastic_objects_I}, we have that 
\begin{equation}\label{measure_re:eq_so_II_p1}
 \sup_{y\in \T^3} \Big ( \E \Big[  \| \lcol (\tau_y P_{N_1} \Wt)  P_{N_2} \Wt \rcol  \|_{\cC^{-1-\delta}_x(\T^3)}^p \Big] \Big)^{\frac{1}{p}} \lesssim  \max(N_1,N_2)^{-\frac{\delta}{2}} p. 
\end{equation}
It only remains to move the supremum in \( y\in \T^3 \) into the expectation. From a crude estimate, we have for all \( y,y^\prime \in \T^3 \) that 
\begin{equation*}
\Big ( \E \Big[  \| \lcol (\tau_y P_{N_1} \Wt)  P_{N_2} \Wt \rcol - \lcol (\tau_{y^\prime} P_{N_1} \Wt)  P_{N_2} \Wt \rcol   \|_{\cC^{-1-\delta}_x(\T^3)}^p \Big] \Big)^{\frac{1}{p}} \lesssim \max(N_1,N_2)^3 \| y- y^\prime\| \, p.  
\end{equation*}
By Kolmogorov's continuity theorem (cf. \cite[Theorem 4.3.2]{Stroock11}), we obtain for any \( 0 < \alpha < 1 \) that  
\begin{equation*}
\bigg ( \E \bigg[  \sup_{y,y^\prime\in \T^3} \bigg(\frac{\| \lcol (\tau_y P_{N_1} \Wt)  P_{N_2} \Wt \rcol - \lcol (\tau_{y^\prime} P_{N_1} \Wt)  P_{N_2} \Wt \rcol   \|_{\cC^{-1-\delta}_x(\T^3)}}{\| y-y^\prime\|^\alpha} \bigg)^p \bigg] \bigg)^{\frac{1}{p}} \lesssim_{\alpha} \max(N_1,N_2)^3 p.
\end{equation*}
Combining this with \eqref{measure_re:eq_so_II_p1} leads to the desired estimate. 
\end{proof}

The next lemma is similar to Lemma \ref{measure_re:lemma_stochastic_objects_I}, but is concerned with more complicated stochastic objects. In order to shorten the argument, we will no longer use Itô's formula to express products of stochastic integrals. Instead, we will utilize the product formula for multiple stochastic integrals from \cite[Proposition 1.1.3]{Nualart06}. Before we state the lemma, we follow \cite{BG18,BG20} and define
\begin{equation}\label{measure_re:eq_wt[3]}
\bWt{[3]} \defe \int_0^t (\Js)^2 \lcol (V* (\Ws)^2) \Ws \rcol \ds. 
\end{equation}
We emphasize that \( \bWt{[3]}\) contains the interaction potential \( V \) even though this is not reflected in our notation.

\begin{lemma}[Stochastic objects III]\label{measure_re:lemma_stochastic_objects_III}
For every \( p \geq 1 \), \( \epsilon >0 \), and every \( 0 < \gamma < \min(\beta,\frac{1}{2})\), we have that 
\begin{align}
\sup_{T,t\geq 0} \Big ( \E \Big[ \| \bWt{[3]} \|_{\cC^{\frac{1}{2}+\gamma}_x(\T^3)}^p \Big] \Big)^{\frac{1}{p}} &\lesssim p^{\frac{3}{2}}, \label{measure_re:eq_so_III_1}\\
\sup_{T,t\geq 0} \Big( \E \Big[ \| (V*\lcol (\Wt)^2\rcol) \bWt{[3]} \|_{\cC_x^{-1+ \gamma}(\T^3)}^p \Big] \Big)^{\frac{1}{p}} &\lesssim p^{\frac{5}{2}}, \allowdisplaybreaks[1] \label{measure_re:eq_so_III_2}\\
\sup_{T,t\geq 0} \Big( \E \Big[ \big\| \big( V* (\Wt \bWt{[3]})\big) \Wt - \MT \bWt{[3]} \big\|_{\cC_x^{-1+\gamma}(\T^3)}^p  \Big] \Big)^{\frac{1}{p}} &\lesssim p^{\frac{5}{2}}.\label{measure_re:eq_so_III_3}
\end{align}
\end{lemma}
\begin{remark}
The analog of \( (V*\lcol (\Wt)^2\rcol ) \bWt{[3]} \) for the \( \Phi^4_3\)-model in \cite{BG18} requires a further logarithmic renormalization. In our case, however, the additional smoothing from the interaction potential \( V \) eliminates the responsible logarithmic divergence. 
\end{remark}
\begin{proof}

We first prove \eqref{measure_re:eq_so_III_1}, which is (by far) the easiest estimate. As in the proof of  Lemma \ref{measure_re:lemma_stochastic_objects_I}, we can use Gaussian hypercontractivity (Lemma \ref{prelim:lemma_hypercontractivity}) to reduce \eqref{measure_re:eq_so_III_2} to the estimate  
\begin{equation}
\E \Big[ \| \bWt{[3]} \|_{H^{\frac{1}{2}+\gamma}_x(\T^3)}^2 \Big] \lesssim 1. 
\end{equation}
The rest of the argument follows from Corollary \ref{measure_re:corollary_stochastic_objects_I} and a deterministic estimate. More precisely, it follows from \( \| \sigmaT{s} \|_{L^2_s}= 1 \) that 
\begin{align*}
\| \bWt{[3]} \|_{H^{\frac{1}{2}+\gamma}_x(\T^3)}^2 &=\Big\| \int_0^t \sigmaT{s}(\nabla)^2 \langle \nabla \rangle^{-\frac{3}{2}+\gamma} \lcol (V* (\Ws)^2) \Ws \rcol  \ds \Big\|_{L^2_x}^2  \\
&=\sum_{n\in \Z^3}  \Big| \int_0^t \sigmaT{s}(n)^2 \mathcal{F}\Big( \langle \nabla \rangle^{-\frac{3}{2}+\gamma} \lcol (V* (\Ws)^2) \Ws \rcol  \Big)(n) \ds \Big|^2 \\
&\leq \sum_{n\in \Z^3} \int_0^t \sigmaT{s}(n)^2 \Big| \mathcal{F}\Big( \langle \nabla \rangle^{-\frac{3}{2}+\gamma} \lcol (V* (\Ws)^2) \Ws \rcol  \Big)(n)  \Big|^2 \ds. 
\end{align*}
For a small \( \delta >0 \), we obtain from Corollary \ref{measure_re:corollary_stochastic_objects_I} (with \( \gamma \) replaced by \( \gamma+\delta\)) that 
\begin{align*}
\E \Big[ \| \bWt{[3]} \|_{H^{\frac{1}{2}+\gamma}_x(\T^3)}^2 \Big] &\leq \sum_{n\in \Z^3} \int_0^t \sigmaT{s}(n)^2 \E \bigg[ \Big| \mathcal{F}\Big( \langle \nabla \rangle^{-\frac{3}{2}+\gamma} \lcol (V* (\Ws)^2) \Ws \rcol  \Big)(n)  \Big|^2 \bigg] \ds \\
&\lesssim \sum_{n\in \Z^3} \int_0^t \sigmaT{s}(n)^2 \frac{1}{\langle n \rangle^{3+\delta}} \ds \lesssim 1. 
\end{align*}

We now turn to the proof of \eqref{measure_re:eq_so_III_2}. Using the same reductions based on Gaussian hypercontractivity as before, it suffices to prove that 
\begin{equation}
\E \Big[ \| (V*\lcol (\Wt)^2\rcol) \bWt{[3]} \|_{H_x^{-1+ \gamma}(\T^3)}^2 \Big] \lesssim 1. 
\end{equation}
We first rewrite \(  (V*\lcol (\Wt)^2\rcol)(x) \bWt{[3]}(x) \) as a product of multiple stochastic integrals instead of iterated stochastic integrals. This allows us to use the product formula from Lemma \ref{appendix:lemma_product_formula}, which leads to a (relatively) simple expression. To simplify the notation below, we define the symmetrization of \( \widehat{V}(n_1+n_2) \) by 
\begin{equation*}
\widehat{V}_S(n_1,n_2,n_3) = \frac{1}{6} \sum_{\pi \in S_3} \widehat{V}(n_{\pi(1)}+n_{\pi(2)}). 
\end{equation*}
From Proposition \ref{measure_re:proposition_renormalization}, \eqref{measure_re:eq_wt[3]}, and the stochastic Fubini theorem (see \cite[Theorem 4.33]{PZ92}), we have that 
\begin{align*}
&\bWt{[3]}(x) \\
&= \sum_{\substack{n_1,n_2,n_3 \in \Z^3 \\ \pi \in S_3}} \frac{\widehat{V}(n_{\pi(1)}+n_{\pi(2)})}{\langle n_{123} \rangle^2} e^{ i \langle n_{123}, x \rangle} 
\int_0^t \sigmaT{s}(n_{123})^2 \Big(  \int_0^s \int_0^{t_1} \int_0^{t_2} \d \W{t_3}{n_3}\d \W{t_2}{n_2} \d \W{t_1}{n_1} \Big) \ds \\
&= \sum_{n_1,n_2,n_3\in \Z^3} \frac{\widehat{V}_S(n_1,n_2,n_3)}{\langle n_{123} \rangle^2} e^{i \langle n_{123}, x \rangle} \int_0^t \int_0^{t_1} \int_0^{t_2}  \Big( \int_{\max(t_1,t_2,t_3)}^t \sigmaT{s}(n_{123})^2 \ds \Big)  \d \W{t_3}{n_3}\d \W{t_2}{n_2} \d \W{t_1}{n_1}
\end{align*}
We define the symmetric function \( f \) by 
\begin{equation*}
f(t_1,n_1,t_2,n_2,t_3,n_3;t,x) \defe \frac{\widehat{V}_s(n_1,n_2,n_3)}{6 \langle n_{123} \rangle^2}  \Big( \int_{\max(t_1,t_2,t_3)}^t  \sigmaT{s}(n_{123})^2 \ds \Big) e^{i \langle n_{123}, x \rangle}  1\{ 0 \leq t_1,t_2,t_3 \leq t\}. 
\end{equation*}
where we view both \( t \in \R_{>0} \) and \( x \in \T^3 \) as fixed parameters. Using the language from Section \ref{section:multiple_stochastic_integrals} and Lemma \ref{appendix:lemma_iterated_vs_multiple}, we obtain that 
\begin{equation}\label{measure_re:eq_integral_f}
\bWt{[3]}(x) = \calI_3[f(\cdot;t,x)],
\end{equation}
where \( \calI_3 \) is a multiple stochastic integral. After defining 
\begin{equation*}
g(t_4,n_4,t_5,n_5;t,x) \defe \widehat{V}(n_4+n_5) e^{i \langle n_{45}, x \rangle} 1\{ 0 \leq t_4,t_5 \leq t\},
\end{equation*}
a similar but easier calculation leads to 
\begin{equation}\label{measure_re:eq_integral_g}
(V*\lcol (\Wt)^2\rcol)(x) = \calI_2[g(\cdot;t,x)]. 
\end{equation}
By combining  \eqref{measure_re:eq_integral_f} and \eqref{measure_re:eq_integral_g}, we obtain that 
\begin{equation*}
(V*\lcol (\Wt)^2\rcol)(x) \bWt{[3]}(x) = \calI_3[f(\cdot;t,x)] \calI_2[g(\cdot;t,x)]. 
\end{equation*}
By using the product formula for multiple stochastic integrals (Lemma \ref{appendix:lemma_product_formula}), we obtain that 
\begin{equation*}
(V*\lcol (\Wt)^2\rcol)(x) \bWt{[3]}(x) = \calI_5[f(\cdot;t,x) g(\cdot;t,x)] + 6 \cdot \calI_3[f(\cdot;t,x)  \otimes_1 g(\cdot;t,x)] + 3 \cdot \calI_1[f(\cdot;t,x)  \otimes_2 g(\cdot;t,x)] . 
\end{equation*}

Inserting the definitions of \( f \) and \( g \), this leads to 
\begin{equation}
(V*\lcol (\Wt)^2\rcol)(x) \bWt{[3]}(x)  = \cG_5(t,x) + \cG_3(t,x) + \cG_1(t,x), 
\end{equation}
where the Gaussian chaoses \( \cG_5, \cG_3 \), and \( \cG_1 \) are given by
\begin{align*}
\cG_5(t,x) 
&= \sum_{n_1,\hdots ,n_5 \in \Z^3} \frac{\widehat{V}(n_{12}) \widehat{V}(n_{45})}{\langle n_{123} \rangle^2} e^{i \langle n_{12345}, x \rangle} \int_{[0,t]^5} \Big( \int_{\max(t_1,t_2,t_3)}^t \sigmaT{s}(n_{123})\ds \Big)  \d \W{t_5}{n_5} \hdots \d \W{t_1}{n_1}, \allowdisplaybreaks[4]\\
\cG_3(t,x) 
&= \sum_{n_1,\hdots, n_5 \in \Z^3} \bigg[ \delta_{n_{35}=0} \frac{\widehat{V}_s(n_1,n_2,n_3) \widehat{V}(n_{45})}{\langle n_{123}\rangle^2 \langle n_3 \rangle^2} e^{i \langle n_{124}, x \rangle}  \\
&\times \int_{[0,t]^3} \Big( \int_0^t \int_{\max(t_1,t_2,t_3)}^t \sigmaT{t_3}(n_3)^2  \sigmaT{s}(n_{123})^2 \ds \d t_3 \Big) \d \W{t_4}{n_4} \d \W{t_2}{n_2} \d \W{t_1}{n_1} \bigg], \allowdisplaybreaks[4] \\
\cG_1(t,x) 
&= \frac{1}{2} \sum_{n_1,\hdots,n_5 \in \Z^3} \bigg[ \delta_{n_{24}=n_{35}=0} \frac{\widehat{V}_s(n_1,n_2,n_3) \widehat{V}(n_{45})}{\langle n_{123}\rangle^2 \langle n_2 \rangle^2 \langle n_3 \rangle^2} e^{i \langle n_{1}, x \rangle}  \\
&\times \int_{[0,t]}\Big( \int_0^t \int_0^t \int_{\max(t_1,t_2,t_3)}^t \sigmaT{t_2}(n_2)^2\sigmaT{t_3}(n_3)^2  \sigmaT{s}(n_{123})^2 \ds \d t_3  \d t_2 \Big) \d \W{t_1}{n_1} \bigg]
\end{align*}
Using the \( L^2\)-orthogonality of the multiple stochastic integrals together with \( \| \sigmaT{s}\|_{L^2_s(\R_{>0})} \leq 1 \), we obtain that 

\begin{align}
&\E \Big[ \| (V*\lcol (\Wt)^2\rcol) \bWt{[3]} \|_{H_x^{-1+\gamma}}^2 \Big]  \notag\\
& \lesssim \E \Big[ \| \cG_5\|_{H_x^{-1+\gamma}}^2 \Big]  + \E \Big[ \| \cG_3\|_{H_x^{-1+\gamma}}^2 \Big]  + \E \Big[ \| \cG_1\|_{H_x^{-1+\gamma}}^2 \Big]   \notag \\
&\lesssim \sum_{n_1,n_2,n_3,n_4,n_5 \in \Z^3}  \langle n_{12345} \rangle^{-2+2\gamma} \langle n_{123} \rangle^{-4} |\widehat{V}(n_{12})|^2 |\widehat{V}(n_{45})|^2 \prod_{j=1}^5 \langle n_j \rangle^{-2}, \label{measure_re:eq_so_p1} \\
&+\sum_{n_1,n_2,n_4 \in \Z^3} \langle n_{124} \rangle^{-2+2\gamma} \Big( \sum_{n_3\in \Z^3} \langle n_{123} \rangle^{-2} \langle n_3 \rangle^{-2} |\widehat{V}_s(n_1,n_2,n_3)| |\widehat{V}(n_{34})| \Big)^2   \prod_{j=1,2,4} \langle n_j \rangle^{-2} \label{measure_re:eq_so_p2} \\
&+\sum_{n_1 \in \Z^3} \langle n_1 \rangle^{-4+2\gamma} \Big( \sum_{n_2,n_3 \in \Z^3} \langle n_{123} \rangle^{-2} |\widehat{V}_s(n_1,n_2,n_3)| |\widehat{V}(n_{23})| \langle n_2 \rangle^{-2} \langle n_3 \rangle^{-2} \Big)^2. \label{measure_re:eq_so_p3}
\end{align}
The estimates of the sums \eqref{measure_re:eq_so_p1}-\eqref{measure_re:eq_so_p3} follow from standard arguments. We present the details for \eqref{measure_re:eq_so_p1} and \eqref{measure_re:eq_so_p3}, but omit the details for the intermediate term \eqref{measure_re:eq_so_p2}.

We start with the estimate of \eqref{measure_re:eq_so_p1}. The interaction with \( n_1,n_2,n_3 \) at low frequency scales and \( n_4,n_5 \) at high frequency scales is worse than all other contributions, so there is a lot of room in several steps below. Using Lemma \ref{appendix:lemma_sum_estimates} for the sum in \( n_5 \), which requires \( \gamma < \min(1,\beta) \), and summing in \( n_4\), we obtain for a small \( \delta >0 \) that
\begin{align*}
 &\sum_{n_1,n_2,n_3,n_4,n_5 \in \Z^3}  \langle n_{12345} \rangle^{-2+2\gamma} \langle n_{123} \rangle^{-4} |\widehat{V}(n_{12})|^2 |\widehat{V}(n_{45})|^2 \prod_{j=1}^5 \langle n_j \rangle^{-2} \\
&\lesssim \sum_{n_1,n_2,n_3,n_4\in \Z^3}  \langle n_{123} \rangle^{-4} \langle n_{12}\rangle^{-2\beta} \Big(\prod_{j=1}^4 \langle n_j \rangle^{-2} \Big) \Big( \sum_{n_5 \in \Z^3} \langle n_{1234} + n_5 \rangle^{-2+2\gamma} \langle n_4 +n_5 \rangle^{-2 \beta} \langle n_5 \rangle^{-2} \Big) \\
&\lesssim \sum_{n_1,n_2,n_3 \in \Z^3} \langle n_{123} \rangle^{-4} \langle n_{12} \rangle^{-2\beta}   \Big(\prod_{j=1}^3 \langle n_j \rangle^{-2}   \Big) \Big( \sum_{n_4\in \Z^3} \big( \langle n_{1234} \rangle^{-1-\delta} + \langle n_4 \rangle^{-1-\delta} \big) \langle n_4 \rangle^{-2} \Big)  \\
&\lesssim \sum_{n_1,n_2,n_3 \in \Z^3}  \langle n_{123} \rangle^{-4} \langle n_{12} \rangle^{-2\beta}   \prod_{j=1}^3 \langle n_j \rangle^{-2} . 
\end{align*}
Summing in \( n_3 \), \(n_2 \), and \( n_1 \), we obtain that 
\begin{align*}
\sum_{n_1,n_2,n_3 \in \Z^3}  \langle n_{123} \rangle^{-4} \langle n_{12} \rangle^{-2\beta}   \prod_{j=1}^3 \langle n_j \rangle^{-2}
\lesssim \sum_{n_1,n_2 \in \Z^3} \langle n_{12} \rangle^{-3-2\beta} \langle n_1 \rangle^{-2} \langle n_2 \rangle^{-2}  
\lesssim \sum_{n_1 \in \Z^3} \langle n_1 \rangle^{-4} \lesssim 1. 
\end{align*}
We now turn to \eqref{measure_re:eq_so_p3}, which corresponds to double probabilistic resonance. We emphasize that this term would be unbounded without smoothing effect of the potential \( V \), which is the reason for the additional renormalization in the \( \Phi^4_3\)-model, see e.g. \cite[Lemma 24]{BG18}. Using {Lemma \ref{appendix:lemma_sum_estimates}} for the sum in \( n_3 \), we obtain that 
\begin{align*}
&\sum_{n_1 \in \Z^3} \langle n_1 \rangle^{-4+2\gamma} \Big( \sum_{n_2,n_3 \in \Z^3} \langle n_{123} \rangle^{-2} |\widehat{V}_s(n_1,n_2,n_3)| |\widehat{V}(n_{23})| \langle n_2 \rangle^{-2} \langle n_3 \rangle^{-2} \Big)^2 \\
&\lesssim \sum_{n_1 \in \Z^3} \langle n_1 \rangle^{-4+2\gamma}  \Big( \sum_{n_2,n_3 \in \Z^3} \langle n_{123} \rangle^{-2} \langle n_{23} \rangle^{-\beta} \langle n_2 \rangle^{-2} \langle n_3 \rangle^{-2} \Big)^2 \\
&\lesssim \sum_{n_1 \in \Z^3} \langle n_1 \rangle^{-4+2\gamma}  \Big( \sum_{n_2 \in \Z^3}  \big( \langle n_{12} \rangle^{-1-\beta} + \langle n_2 \rangle^{-1-\beta} \big) \langle n_2 \rangle^{-2}   \Big)^2 \\
&\lesssim \sum_{n_1 \in \Z^3} \langle n_1 \rangle^{-4+2\gamma} \lesssim 1,
\end{align*}
provided that \( \gamma < 1/2 \). This completes the proof of \eqref{measure_re:eq_so_III_2}. \\

We now turn to the proof of \eqref{measure_re:eq_so_III_3}. This stochastic object has a more complicated algebraic structure than the stochastic object in \eqref{measure_re:eq_so_III_2}, but a similar analytic behavior. From the definition of \( \MT \), we obtain that
\begin{align*}
 &\big( V* (\Wt \bWt{[3]})\big)(x) \Wt(x) - \MT \bWt{[3]}(x)\\
&= \sum_{m_1,m_4,m_5 \in \mathbb{Z}^3} \widehat{V}(m_{14})  e^{i \langle m_{145}, x \rangle} \widehat{\bWt{[3]}}(m_1)  \Big( \Wt[m_4] \Wt[m_5] - \delta_{m_{45}=0} \frac{\rhoT{t}(m_4)^2}{\langle m_4 \rangle^2} \Big) \\
&= \frac{1}{2} \sum_{m_1,m_4,m_5 \in \mathbb{Z}^3} \big(\widehat{V}(m_{14}) +\widehat{V}(m_{15}) \big) e^{i \langle m_{145}, x \rangle} \widehat{\bWt{[3]}}(m_1)  \Big( \Wt[m_4] \Wt[m_5] - \delta_{m_{45}=0} \frac{\rhoT{t}(m_4)^2}{\langle m_4 \rangle^2} \Big). 
\end{align*} 
Using the variable names \( m_1,m_4,m_5 \in \Z^3 \) instead of \( m_1,m_2,m_3 \in \Z^3 \) is convenient once we insert an expression for \(  \bWt{[3]}   \). A minor modification of the derivation of \eqref{measure_re:eq_integral_f} shows that
\begin{equation}
\widehat{\bWt{[3]}}(m_1)  = \calI[f(\cdot;t,m_1)],
\end{equation}
where the symmetric function \( f(\cdot;t,m_1) \) is given by
\begin{align*}
&f(t_1,n_1,t_2,n_3,t_3,n_3;t,m_1) \\
&= 1\{ n_{123}=m_1\} \frac{1}{\langle n_{123} \rangle^2} \widehat{V}_S(n_1,n_2,n_3) \Big( \int_{\max(t_1,t_2,t_3)}^t \sigmaT{s}(n_{123})^2 \ds \Big) 1\{ 0 \leq t_1,t_2,t_3 \leq t\}. 
\end{align*}
Using Lemma \ref{measure_re:lemma_products} and Lemma \ref{appendix:lemma_iterated_vs_multiple}, we obtain that 
\begin{equation}
\Wt[m_4] \Wt[m_5] - \delta_{m_{45}=0} \frac{\rhoT{t}(m_4)^2}{\langle m_4 \rangle^2}  = \calI_2[g(\cdot;t,m_4,m_5)],
\end{equation}
where the symmetric function \( g(\cdot;t,m_4,m_5)\) is given by
\begin{align*}
g(t_4,n_4,t_5,n_5) 
\defe \frac{1}{2} \Big( 1\{ (n_4,n_5)= (m_4,m_5)\} + 1\{ (n_4,n_5)=(m_5,m_4)\} \Big) 1\{ 0 \leq t_4,t_5 \leq t\}. 
\end{align*}
The author believes that inserting indicators such as \( 1\{ (n_4,n_5)= (m_4,m_5)\}  \) is notationally unpleasant, but it allows us to use the multiple stochastic integrals from \cite{Nualart06}  without having to ``reinvent the wheel''.  With this notation, we obtain that 
\begin{align*}
 &\big( V* (\Wt \bWt{[3]})\big)(x) \Wt(x) - \MT \bWt{[3]}(x) \\
&=  \frac{1}{2} \sum_{m_1,m_4,m_5\in \Z^3} e^{i \langle m_{145}, x \rangle} \big( \widehat{V}(m_{14})+\widehat{V}(m_{15})\big) \cdot \calI_3[f(\cdot;t,m_1)] \cdot \calI_2[g(\cdot;t,m_4,m_5)]. 
\end{align*}
Using Lemma \ref{appendix:lemma_product_formula}, we obtain that
\begin{equation}\label{measure_re:eq_so_p4}
\big( V* (\Wt \bWt{[3]})\big)(x) \Wt(x) - \MT \bWt{[3]}(x) = \ctG_5(t,x)+ \ctG_3(t,x)+\ctG_1(t,x),
\end{equation}
where the Gaussian chaoses are defined as 
\begin{align*}
\ctG_5(t,x) &= \sum_{n_1,\hdots ,n_5 \in \Z^3} \frac{\widehat{V}(n_{12}) \widehat{V}(n_{1234})}{\langle n_{123} \rangle^2} e^{i \langle n_{12345}, x \rangle} \int_{[0,t]^5} \Big( \int_{\max(t_1,t_2,t_3)}^t \sigmaT{s}(n_{123})\ds \Big)  \d \W{t_5}{n_5} \hdots \d \W{t_1}{n_1}, \allowdisplaybreaks[4]\\
\ctG_3(t,x) &=  \frac{1}{2} \sum_{n_1,\hdots, n_5 \in \Z^3} \bigg[ \delta_{n_{35}=0} \frac{\widehat{V}_s(n_1,n_2,n_3) }{\langle n_{123}\rangle^2 \langle n_3 \rangle^2}  \Big( \widehat{V}(n_{12})+\widehat{V}(n_{1234}) \Big) e^{i \langle n_{124}, x \rangle}  \\
&\times \int_{[0,t]^3} \Big( \int_0^t \int_{\max(t_1,t_2,t_3)}^t \sigmaT{t_3}(n_3)^2  \sigmaT{s}(n_{123})^2 \ds \d t_3 \Big) \d \W{t_4}{n_4} \d \W{t_2}{n_2} \d \W{t_1}{n_1} \bigg], \allowdisplaybreaks[4] \\
\ctG_1(t,x) &=  \frac{1}{4} \sum_{n_1,\hdots,n_5 \in \Z^3} \bigg[ \delta_{n_{24}=n_{35}=0} \frac{\widehat{V}_s(n_1,n_2,n_3)}{\langle n_{123}\rangle^2 \langle n_2 \rangle^2 \langle n_3 \rangle^2} \Big( \widehat{V}(n_{12})+\widehat{V}(n_{13})\Big) e^{i \langle n_{1}, x \rangle}  \\
&\times \int_{[0,t]}\Big( \int_0^t \int_0^t \int_{\max(t_1,t_2,t_3)}^t \sigmaT{t_2}(n_2)^2\sigmaT{t_3}(n_3)^2  \sigmaT{s}(n_{123})^2 \ds \d t_3  \d t_2 \Big) \d \W{t_1}{n_1} \bigg]. 
\end{align*}
This concludes the algebraic aspects of the proof of \eqref{measure_re:eq_so_III_3}. Starting from  \eqref{measure_re:eq_so_p4}, the analytic estimates are essentially as in the proof of the earlier estimate \eqref{measure_re:eq_so_III_2} and we omit the details. This completes the proof of the lemma. 
\end{proof}

In the construction of the drift measure (Section \ref{section:drift_measure}), we need a renormalization of \( (\langle \nabla \rangle^{-1/2} \Wt )^n \). The term \( \langle \nabla \rangle^{-1/2} \Wt \) has regularity \( 0-\) and hence the \( n\)-th power is almost defined. While we could use iterated stochastic integrals to define the renormalized power, it is notationally convenient to use an equivalent definition through Hermite polynomials. This definition is also closer to the earlier literature in dispersive PDE. We recall that the Hermite polynomials \( \{ H_n(x,\sigma^2)\}_{n\geq 0} \) are defined through the generating function 
\begin{equation*}
e^{tx - \frac{1}{2} \sigma^2 t^2} = \sum_{n=0}^\infty \frac{t^n}{n!} H_n(x,\sigma^2). 
\end{equation*}

\begin{definition}
We define the renormalized \( n\)-th power by 
\begin{equation}
\lcol f^n \rcol \defe H_n\Big( f  , ~ \E \| \langle \nabla \rangle^{-\frac{1}{2}} \Wt  \|_{L^2_x}^2 \Big). 
\end{equation}
\end{definition}
We list two basic properties of the renormalized power in the next lemma.

\begin{lemma}[Stochastic objects IV]\label{measure_re:lemma_high_power}
We have for all \( n \geq 1 \), \( p \geq 1 \),  and \( \epsilon > 0 \) that 
\begin{equation}
\sup_{T\geq 0} \Big( \E \Big[ \| \lcol ( \langle \nabla \rangle^{-\frac{1}{2}} \Wt )^n \rcol \|_{C_x^{-\epsilon}(\T^3)}^p \Big] \Big)^{\frac{1}{p}}  \lesssim_{n,\epsilon} p^{\frac{n}{2}}. 
\end{equation}
Furthermore, we have for all \( f \in H^1_x(\T^3) \) the binomial formula
\begin{equation}
\lcol ( \langle \nabla \rangle^{-\frac{1}{2}} (\Wt +f))^n \rcol = \sum_{k=0}^n { n \choose k}  \lcol ( \langle \nabla \rangle^{-\frac{1}{2}} \Wt )^k \rcol ~ (\langle \nabla \rangle^{-\frac{1}{2}} f )^{n-k}. 
\end{equation}
\end{lemma}
Since the proof is standard, we omit the details. For similar arguments, we refer the reader to \cite{OT18}. 

\section{Construction of the Gibbs measure}\label{section:construction}

The goal of this section is to prove Theorem \ref{theorem:tightness}. 
The main ingredient is the Boué-Dupuis formula, which yields a variational formulation of the Laplace transform of \( \tmuT \). Our argument follows earlier work of Barashkov and Gubinelli \cite{BG18}, but the convolution inside the nonlinearity requires additional ingredients (see Section \ref{section:visan} and Section \ref{section:rmt}). 

\subsection{\protect{The variational problem, uniform bounds, and their consequences}} \label{section:variational}

Due to the singularity of the Gibbs measure for \( 0< \beta < 1/2 \), which is the main statement in Theorem \ref{theorem:singularity},  the construction will require one final renormalization. We recall that \( \lambda > 0 \) denotes the coupling constant in the nonlinearity and we let \( \cT \) be a real-valued constant which remains to be chosen. \\

For the rest of this section, we let \( \varphi \colon \cC_t^0 \cC_{x}^{-1/2-\kappa}([0,\infty]\times \R) \rightarrow \R \) be a functional with at most linear growth. 
We denote the (non-renormalized) potential energy by
\begin{equation}\label{measure_var:eq_cV}
\cV(f) \defe \int_{\T^3} (V* f^2)(x) f^2(x) \dx =  \int_{\T^3 \times \T^3} V(x-y) f(y)^2 f(x)^2 \dx \dy. 
\end{equation}
We denote the renormalized version of \( \cV(f) \) by 
\begin{equation}\label{measure_var:eq_renormalized_V}
\lcol \cVT(f) \rcol \overset{def}{=}  \frac{\lambda}{4} \cdot \int_{\T^3} \lcol ( V* f^2) f^2 \rcol \dx + \cT, 
\end{equation}
where \( \lcol ( V* f^2) f^2 \rcol \) is as in Definition \ref{measure_re:definition_renormalization}. To further simplify the notation, we denote for any \( u \colon [0,\infty) \times \T^3 \rightarrow \R \) the space-time \( L^2\)-norm by 
\begin{equation}
\| u \|_{\cH}^2 \defe \int_0^\infty \| u_t \|_{L^2_x(\T^3)}^2 \, \mathrm{d}t. 
\end{equation}
With this notation, we can now state the main estimate of this section. 
\begin{proposition}[Main estimate for the variational problem] \label{measure_var:proposition_main}
If the renormalization constants \( \cT \) are chosen appropriately, we have that
\begin{equation}\label{measure_var:eq_main}
\begin{aligned}
&\E_{\bP} \bigg[ \varphi( W + I[u])+  \lcol \cVT( \Winf  + \Iinf[u] )\rcol  + \frac{1}{2} \| u \|_{\cH}^2 \bigg] \\
&= \E_{\bP} \bigg[ \PsiT (W, I[u]) + \frac{\lambda}{4} \cV(\Iinf(u)) + \frac{1}{2} \| \lP[u] \|_{\cH}^2 \bigg],
\end{aligned}
\end{equation}
where 
\begin{equation}\label{measure_var:eq_lT}
\lt[u] \overset{def}{=} u_t +  \lambda \Jt \lcol (V*(\Wt)^2) \Wt  \rcol 
\end{equation}
and
\begin{equation}\label{measure_var:eq_PsiT}
| \PsiT (W, I[u])| \leq Q_T(W,\varphi,\lambda) + \frac{1}{2} \Big( \frac{\lambda}{4} \cV(\Iinf(u)) + \frac{1}{2} \| \lP[u] \|_{\cH}^2 \Big). 
\end{equation}
Here, \( Q_T(W,\varphi,\lambda) \) satisfies for all \( p \geq 1 \) the estimate 
\( \E[Q_T(W,\varphi,\lambda)^p] \lesssim_p 1 \), where the implicit constant is uniform in \( T \geq 1 \). 
\end{proposition}

The argument of \( \varphi \) in \eqref{measure_var:eq_main} is not regularized, that is, we are working with \( W \) instead of \( \WP \). This is important to obtain control over \( \muT  \), which is the pushforward of \( \tmu_T \) under \( W_\infty \). 

\begin{remark}
This is a close analog of \cite[Theorem 1]{BG18}. Due to the smoothing effect of the interaction potential \( V \), however, the shifted drift \( \lP[u] \) is simpler. In contrast to the \( \Phi^4_3 \)-model, the difference \( l^T(u) - u \) does not depend on \( u \). 
As is evident from the proof, we have that 
\begin{equation}\label{measure_var:eq_PsiTz}
 \PsiT (W, I[u]) = \varphi(W+I[u]) +  \PsiTz (W, I[u]). 
\end{equation}
This observation will only be needed in Proposition \ref{measure_var:proposition} below. 
\end{remark}

We first record the following proposition, which is a direct consequence of Proposition \ref{measure_var:proposition_main} and the Boué-Dupuis formula.

\begin{proposition}\label{measure_var:proposition} The measures \( \tmuT \) satisfy the following properties:
\begin{enumerate}[(i)]
\item The normalization constants  \( \cZTl \)  satisfy \( \cZTl \sim_\lambda 1 \), i.e., they are bounded away from zero and infinity uniformly in \( T \).  \label{measure_var:item_corollary_i} 
\item If the functional \( \varphi \colon \cC_t^0 \cC_x^{-1/2-\kappa}([0,\infty]\times \T^3) \rightarrow \R \) has at most linear growth,   \label{measure_var:item_corollary_ii}
then
\begin{equation*}
\sup_{T\geq 0 } \E_{\tmu_T}\Big[ \exp\big( - \varphi(W) \big) \Big] \lesssim_\varphi 1. 
\end{equation*}
\item The family of measures \( (\tmuT)_{T\geq0} \) is tight on \( \cC_t^0 \cC_x^{-\frac{1}{2}-\kappa}([0,\infty]\times \T^3)\). \label{measure_var:item_corollary_iii}
\end{enumerate}
\end{proposition}

\begin{proof}[Proof of Proposition \ref{measure_var:proposition}:] 
We first prove \eqref{measure_var:item_corollary_i}. From the definition of \( \mu_T \), we have that 
\begin{equation*}
\cZTl = \E_{\bP} \Big[ \exp( - \lcol \cVT( \Winf) \rcol ) \Big]. 
\end{equation*}
Using the Boué-Dupuis formula and Proposition \ref{measure_var:proposition_main}, we have that 
\begin{align*}
- \log( \cZTl) &= \inf_{u \in \bH} \E_{\bP} \Big[ \lcol \cVT ( \Winf + \Iinf[u] ) \rcol + \frac{1}{2} \| u\|_{\cH}^2 \Big] \\
&= \inf_{u \in \bH} \E_{\bP} \bigg[ \PsiTz (W, I[u]) + \frac{\lambda}{4} \cV(\Iinf(u)) + \frac{1}{2} \| \lP[u] \|_{\cH}^2 \bigg]. 
\end{align*}
From \eqref{measure_var:eq_PsiT}, we directly obtain that 
\begin{equation}\label{measure:var_eq_cT_p1}
- \log( \cZTl ) \geq - C_\lambda. 
\end{equation}
By choosing \( u_t \defe - \lambda \Jt \lcol (V* (\Wt)^2) \Wt \rcol \), which is equivalent to requiring \( \lt[u] = 0 \) and implies \( \It[u] = \bWt{[3]} \), we obtain from Lemma \ref{measure_re:lemma_stochastic_objects_III} that 
\begin{equation}\label{measure:var_eq_cT_p2}
-\log(\cZTl) \lesssim_\lambda 1 + \E_{\bP}  \Big[ \cV( \lambda \bWt{[3]})\Big] \lesssim_\lambda 1. 
\end{equation}
By combining \eqref{measure:var_eq_cT_p1} and \eqref{measure:var_eq_cT_p2}, we obtain that \( \cZTl \sim_\lambda 1\). \\
We now turn to \eqref{measure_var:item_corollary_ii}, which controls the Laplace transform of \( \tmuT \). Using the Boué-Dupuis formula and Proposition \ref{measure_var:proposition_main}, we obtain that 
\begin{align*}
- \log\Big( \E_{\tmu_T}\Big[ \exp\big( - \varphi(W) \big) \Big] \Big) = \log(\cZTl ) +  \inf_{u \in \bH} \E_{\bP} \bigg[ \PsiT (W, I[u]) + \frac{\lambda}{4} \cV(\Iinf(u)) + \frac{1}{2} \| \lP[u] \|_{\cH}^2 \bigg]. 
\end{align*}
The first summand \( \log(\cZTl) \) has already been controlled. The second summand can be controlled using exactly the same estimates. \\
We finally prove \eqref{measure_var:item_corollary_iii}. Let \( \alpha,\eta >0 \) be sufficiently small depending on \( \kappa\). Since the embedding \( \cC_t^{\alpha,\eta} \cC_x^{-\frac{1+\kappa}{2}} \hookrightarrow \cC_t^0 \cC_x^{-\frac{1}{2}-\kappa} \) is compact (see \eqref{notation:eq_Cweighted} for the definition), it suffices to estimate the Laplace transform evaluated at 
\begin{equation}
\varphi(W) = - \| W \|_{ \cC_t^{\alpha,\eta} \cC_x^{-\frac{1+\kappa}{2}}}. 
\end{equation}
While this is not a functional on \(  \cC_t^0 \cC_x^{-\frac{1}{2}-\kappa} \), we can proceed using a minor modification of the previous estimates. Using Proposition \ref{measure_var:proposition_main} and \eqref{measure_var:eq_PsiTz}, it suffices to prove
\begin{equation}
\E_{\bP} \big[ \| W \|_{ \cC_t^{\alpha,\eta} \cC_x^{-\frac{1+\kappa}{2}}}  \big] \lesssim 1  \quad \text{and} \quad  \| I_t[u] \|_{ \cC_t^{\alpha,\eta} \cC_x^{-\frac{1+\kappa}{2}}} \lesssim \| u \|_{L_{t,x}^2}. 
\end{equation}
The first estimate follows from Kolmogorov's continuity theorem (cf. \cite[Theorem 4.3.2]{Stroock11}). The second estimate is deterministic and follows from Sobolev embedding and Lemma \ref{appendix:lemma_It}. 
\end{proof}

Using Proposition \ref{measure_var:proposition}, we easily obtain Theorem \ref{theorem:tightness}.

\begin{proof}[Proof of Theorem \ref{theorem:tightness}:]
The tightness is included in Proposition \ref{measure_var:proposition}. The weak convergence of the sequence $(\mu_N)_{N\geq1}$ follows from tightness and the uniqueness of weak subsequential limits (Proposition \ref{appendix:prop_uniqueness}). 
\end{proof}

We also record the following consequence of the proof of Proposition \ref{measure_var:proposition_main}, which will play an important role in Section \ref{section:singularity}. The proof of this result will be postponed until Section \ref{section:proof_variational}. 

\begin{corollary}[Behavior of $\cT$]\label{measure_var:corollary_cT}
If \( \beta > 1/2 \), then we have for all \( \lambda >0 \) that 
\begin{equation}
\sup_{T\geq 1} |\cT| \lesssim_\lambda 1. 
\end{equation}
\end{corollary}

Proposition \ref{measure_var:proposition_main} is the most challenging part in the construction of the measure and the proof will be distributed over the remainder of this subsection.

\subsection{Visan's estimate and the cubic terms}\label{section:visan}
In the variational problem, the potential energy \( \cV(\Iinf[u]) \) appears with a favorable sign. This is crucial to control the terms in \( {\lcol \cVT(\Winf+\Iinf[u])\rcol}\) which are cubic in \( \Iinf[u]\) and hence cannot be controlled by the quadratic terms \( \| u\|_{L^2}^2 \) or \( \| l^T(u)\|_{L^2}^2 \). In the \( \Phi^4_3\)-model, the potential energy term \( \| \Iinf[u]\|_{L^4}^4 \) is both stronger and easier to handle. While we cannot change the strength of \( \cV(\Iinf[u]) \), Lemma \ref{measure_var:lemma_visan} solves the algebraic difficulties. \\

Due to the assumed lower-bound on \( V \), we first note that 
\begin{equation*}
\| f \|_{L^2_x(\T^3)}^4 = \| f^2 \|_{L^1_x(\T^3)}^2 \lesssim \int_{\T^3 \times \T^3} V(x-y) f(y)^2 f(x)^2 \dx \dy = \cV(f). 
\end{equation*}
Since at high-frequencies the kernel of \( \langle \nabla \rangle^{-\beta} \) essentially behaves like \( |x-y|^{-(3-\beta)} \), we also obtain that 
\begin{equation}\label{measure_var:eq_trivial_potential}
\| \langle \nabla \rangle^{-\frac{\beta}{2}}[f^2] \|_{L^2(\T^3)}^2 = \langle \big( \langle \nabla \rangle^{-\beta} f^2\big) , f^2  \rangle_{L^2_x(\T^3)} \lesssim  \int_{\T^3 \times \T^3} V(x-y) f(y)^2 f(x)^2 \dx \dy  = \cV(f). 
\end{equation}
Unfortunately, the square of \( f \) is inside the integral operator \( \langle \nabla \rangle^{-\frac{\beta}{2}} \), which makes it difficult to use this estimate.  The next lemma yields a much more useful lower bound on \( \cV(f) \). 
\begin{lemma}[Visan's estimate]\label{measure_var:lemma_visan}
Let \( 0 < \beta < 3 \) and \( f\in C^\infty(\T^3) \). Then, it holds that 
\begin{equation}
\| \langle \nabla \rangle^{-\frac{\beta}{4}} f \|_{L^4_x(\T^3)}^4 \lesssim \cV(f). 
\end{equation}
\end{lemma}

This estimate is a minor modification of \cite[(5.17)]{Visan07} and we omit the details. We now turn to the primary application of Visan's estimate in this work. 

\begin{lemma}[Cubic estimate]\label{measure_var:lemma_cubic}
For any small \( \delta >0\) and any \(  \frac{1+2\delta}{2} < \theta \leq 1  \), it holds that
\begin{equation}
\Big\| \langle \nabla \rangle^{\frac{1}{2}+\delta} \Big( (V* f^2) f \Big) \Big\|_{L^1_x(\T^3)} \lesssim \cV(f)^{\frac{1}{2}} \| f\|_{L^2_x(\T^3)}^{1-\theta} \| f \|_{H^1_x(\T^3)}^{\theta}. 
\end{equation}
\end{lemma}

\begin{proof} 
We use a Littlewood-Paley decomposition to write
\begin{equation*}
(V*f^2) f = \sum_{M,N_3} P_M \big( V*f^2\big) \cdot P_{N_3} f.  
\end{equation*}
We first estimate the contribution for \( N_3 \gtrsim M \).  We have that
\begin{align*}
&\sum_{\substack{M,N_3\colon  N_3 \gtrsim M}} \big\| \langle \nabla \rangle^{\frac{1}{2}+\delta} \Big( P_M \big( V*f^2\big) \cdot P_{N_3} f \Big) \big\|_{L^1_x} \\
&\lesssim \sum_{\substack{M,N_3 \colon  N_3 \gtrsim M}} N_3^{\frac{1}{2}+\delta}  \| P_M( V * f^2) \|_{L^2_x} \| P_{N_3} f \|_{L^2_x} \\
&\lesssim \Big( \sum_{\substack{M,N_3 \colon  N_3 \gtrsim M}} N_3^{\frac{1}{2}+\delta} M^{-\frac{\beta}{2}} N_3^{-\theta} \Big) \| \langle \nabla \rangle^{-\frac{\beta}{2}} f^2 \|_{L^2_x} \| f\|_{L^2_x}^{1-\theta} \| f \|_{H^1_x}^{\theta} \\
&\lesssim \| \langle \nabla \rangle^{-\frac{\beta}{2}} f^2 \|_{L^2_x} \| f\|_{L^2_x}^{1-\theta} \| f \|_{H^1_x}^{\theta} . 
\end{align*}
Due to \eqref{measure_var:eq_trivial_potential}, this contribution is acceptable. Next, we estimate the contribution of \( N_3 \lesssim M \). We further decompose
\begin{equation*}
f^2 = \sum_{N_1,N_2} P_{N_1} f \cdot P_{N_2} f. 
\end{equation*}
Then, the total contribution can be bounded using Hölder's inequality  and Fourier support considerations by
\begin{align*}
&\sum_{\substack{N_1,N_2,N_3,M\colon \\ N_3 \lesssim M \leq \max(N_1,N_2)}} \Big\| \langle \nabla\rangle^{\frac{1}{2}+\delta} \Big( P_M\big( V * (P_{N_1} f \cdot P_{N_2} f) \big) \cdot P_{N_3} f\Big) \Big\|_{L^1_x} \\
&\lesssim \sum_{\substack{N_1,N_2,N_3,M\colon \\ N_3 \lesssim M \leq \max(N_1,N_2)}}M^{\frac{1}{2}+\delta} \| P_M\big( V  * (P_{N_1} f\cdot P_{N_2} f ) \big) \|_{L_x^{\frac{4}{3}}} \| P_{N_3} f \|_{L^4_x} \allowdisplaybreaks[2]\\
&\lesssim  \sum_{\substack{N_1,N_2,M\colon \\ N_3 \lesssim M \leq\max(N_1,N_2)}}M^{\frac{1}{2}+\delta-\beta} N_3^{\frac{\beta}{4}}   \| P_{N_1} f\cdot P_{N_2} f  \|_{L_x^{\frac{4}{3}}} \|  P_{N_3} \langle \nabla \rangle^{-\frac{\beta}{4}} f \|_{L^4_x}\allowdisplaybreaks[2] \\
&\lesssim  \Big( \sum_{\substack{N_1,N_2,M\colon \\ N_1 \geq M,N_2}} M^{\frac{1}{2}+\delta-\frac{3\beta}{4}} N_1^{-\theta} N_2^{\frac{\beta}{4}} \Big) \| \langle \nabla \rangle^{-\frac{\beta}{4}} f \|_{L^4_x}^2 \| f \|_{L^2_x}^{1-\theta} \| f \|_{H^1_x}^{\theta} \allowdisplaybreaks[2] \\
&\lesssim \| \langle \nabla \rangle^{-\frac{\beta}{4}} f \|_{L^4_x}^2 \| f \|_{L^2_x}^{1-\theta} \| f \|_{H^1_x}^{\theta}.
\end{align*}
In the last line, it is simplest to first perform the sum in \( N_2 \), then in \( N_1 \), and finally in \( M \). 
\end{proof}

\subsection{A random matrix estimate and the quadratic terms}\label{section:rmt}

In the proof of Proposition \ref{measure_var:proposition_main}, we will encounter expressions such as 
\begin{equation}\label{measure_rmt:eq_motivation}
\int_{\T^3} \bigg( \big( V * (\Wt \It[u])\big)(x) \Wt(x) \It[u](x) - ( \MT \It[u])(x) \It[u](x) \bigg) \dx. 
\end{equation}
This term no longer involves an explicit stochastic object, such as \( \lcol (\Wt)^2 \rcol (x) \), at a single point \( x \in \T^3 \). By expanding the convolution, we can capture stochastic cancellations in terms of two spatial variables \( x \in \T^3 \) and \( y\in \T^3\), which has already been studied in Lemma \ref{measure_re:lemma_stochastic_objects_II}. The most natural way to capture stochastic cancellations in \eqref{measure_rmt:eq_motivation}, however, is through random operator bounds. This is the object of the next lemma.

\begin{proposition}[Random matrix estimate]\label{measure_rmt:proposition_estimate}
Let \( \gamma > \max(1-\beta,1/2) \) and let \( 1 \leq r \leq \infty \). We define
\begin{equation*}
\OpT \defe \sup_{ \substack{ f_1, f_2 \colon \\ \| f_1 \|_{\mathbb{W}_x^{\gamma,r}(\T^3)}\leq 1, \\ \| f_2 \|_{\mathbb{W}_x^{\gamma,r^\prime}(\T^3)}\leq 1  .}} \bigg[ \int_{\T^3} V* (\Wt f_1 ) ~ \Wt f_2 \dx - \int_{\T^3} \big( \MT f_1 \big) f_2 \dx \bigg]. 
\end{equation*}
Then, we have for all \( 1 \leq p  <\infty \) that 
\begin{equation}\label{measure_rmt:eq_estimate}
\sup_{T,t\geq 0} \| \OpT \|_{L^p_\omega(\Omega)} \lesssim p. 
\end{equation}
\end{proposition}

\begin{remark}
Aside from Fourier support considerations, the proof below mainly proceeds in physical space. If \( r=2 \), an alternative approach is to view \( \OpTtwo \) as the operator norm of  a random matrix acting on the Fourier coefficients. Using a non-trivial amount of combinatorics, one can then bound \( \OpTtwo \) using the moment method (see also \cite[Proposition 2.8]{DNY20}).  This alternative approach is closer to the methods in the literature on random dispersive equations but more complicated. The estimate for \( r \neq 2 \), which is not needed in this paper, is useful in the study of the stochastic heat equation with Hartree nonlinearity. 
\end{remark}

\begin{proof}
Since this will be important in the proof, we now indicate the dependence of the multiplier on the interaction potential by writing \( \MT[V] \). We use a Littlewood-Paley decomposition of \( \Wt , f_1, \) and \( f_2 \). We then have that  
\begin{align*}
&\int_{\T^3} V* (\Wt f_1 ) ~ \Wt f_2 \dx - \int_{\T^3} \big( \MT[V]f_1 \big) f_2 \dx \\
&= \sum_{K_1,K_2,N_1,N_2} \bigg[ \int_{\T^3} V* (P_{N_1}\Wt \, P_{K_1} f_1 ) ~ P_{N_2}\Wt P_{K_2} f_2 \dx - \int_{\T^3} \big( \MT[V;N_1,N_2] P_{K_1}f_1 \big)P_{K_2} f_2 \dx\bigg].
\end{align*}
To control this sum, we first define a frequency-localized version of \( \OpT \) by 
\begin{equation*}
\begin{aligned}
&\OpTN \\
&\defe \sup_{\substack{f_1,f_2 \colon\\ \| f_1\|_{L^r_x}\leq 1,\\ \|f_2 \|_{L^{r^\prime}_x}\leq 1}}\bigg[ \int_{\T^3} V* (P_{N_1}\Wt \, P_{K_1} f_1 ) ~ P_{N_2}\Wt P_{K_2} f_2 \dx - \int_{\T^3} \big( \MT[V;N_1,N_2] P_{K_1}f_1 \big)P_{K_2} f_2 \dx\bigg].
\end{aligned}
\end{equation*}
We emphasize the change from \( \mathbb{W}_x^{\gamma,r}(\T^3) \) to \( L^r_x(\T^3) \), which simplifies the notation below. By proving the estimate for a slightly smaller \( \gamma \), \eqref{measure_rmt:eq_estimate} reduces to
\begin{equation}
\sup_{T,t\geq 0} \| \OpTN \|_{L^p_\omega(\Omega)} \lesssim p (N_1 N_2)^{-\delta} (K_1 K_2)^\gamma. 
\end{equation}
By using Lemma \ref{measure_re:lemma_stochastic_objects_I} and Lemma \ref{measure_re:lemma_stochastic_objects_II}, it suffices to prove for a small \( \delta >0 \) that 
\begin{equation}
\begin{aligned}
\OpTN &\lesssim (N_1 N_2)^{-\delta} (K_1 K_2)^\gamma \\
&\times \Big( 1 + \| \Wt \|_{\cC_x^{-\frac{1}{2}-\delta}}^2 + \sup_{y\in \T^3} \sup_{N_1,N_2} \| \lcol (\tau_y P_{N_1} \Wt) P_{N_2} \Wt \rcol \|_{\cC^{-1-\delta}_x} \Big). 
\end{aligned}
\end{equation}
By interpolation, we can further reduce to \( r=1 \) or \( r=\infty \). Using the self-adjointness of the convolution with \( V \) and the multiplier \( \MT[V;N_1,N_2] \), it suffices to take \( r=1 \). We now separate the cases \( N_1 \sim N_2 \) and \( N_1 \not \sim N_2 \). \\

\emph{Case 1: \( N_1 \not \sim N_2 \).} This is the easier (but slightly tedious) case and it does not contain any probabilistic resonances. We note that \( \MT[V;N_1,N_2]=0 \) and hence we only need to control the convolution term. From Fourier support considerations, we also see that this term vanishes unless \( \max(K_1,K_2) \gtrsim \max(N_1,N_2) \). While our conditions on \( f_1 \) and \( f_2 \) are not completely symmetric and we already used the self-adjointness to restrict to \( r=1 \), we only treat the case \( K_1 \gtrsim K_2 \). Since our proof only relies on Hölder's inequality and Young's inequality, the case \( K_1 \lesssim K_2 \) can be treated similarly. We now estimate
\begin{align*}
&\bigg| \int_{\T^3} V* (P_{N_1}\Wt \, P_{K_1} f_1 ) ~ P_{N_2}\Wt P_{K_2} f_2 \dx \bigg| \\
&\lesssim \sum_{L\lesssim K_1} \bigg| \int_{\T^3} P_L\Big(V* (P_{N_1}\Wt \, P_{K_1} f_1 ) \Big)  ~ \widetilde{P}_L \Big(P_{N_2}\Wt P_{K_2} f_2\Big) \dx \bigg|\\
&\lesssim \sum_{L\lesssim K_1} \big \| (P_L V) * (P_{N_1}\Wt \, P_{K_1} f_1) \big\|_{L_x^1}  \big \| \widetilde{P}_L (P_{N_2}\Wt P_{K_2} f_2 ) \big\|_{L_x^\infty} \\
&\lesssim \| P_{N_1} \Wt \|_{L_x^\infty} \| f_1 \|_{L_x^1}  \sum_{L\lesssim K_1} \| P_L V \|_{L_x^1}  \big \| \widetilde{P}_L (P_{N_2}\Wt P_{K_2} f_2 ) \big\|_{L_x^\infty} \\
&\lesssim N_1^{\frac{1}{2}+\delta} \| \Wt \|_{\cC^{-\frac{1}{2}-\delta}_x}  \sum_{L\lesssim K_1} L^{-\beta} \big \| \widetilde{P}_L (P_{N_2}\Wt P_{K_2} f_2 ) \big\|_{L_x^\infty} . 
\end{align*}
We now split the last sum into the cases \( L \ll N_2 \) and \( N_2 \lesssim L \lesssim K_1 \). If \( L \ll N_2 \), we only obtain a non-zero contribution when \( N_2 \sim K_2 \). Thus, the corresponding contribution is bounded by 
\begin{align*}
&1\{K_2\sim N_2\} N_1^{\frac{1}{2}+\delta} \| \Wt \|_{\cC^{-\frac{1}{2}-\delta}_x}  \sum_{L\lesssim N_2 } L^{-\beta} \big \| \widetilde{P}_L (P_{N_2}\Wt P_{K_2} f_2 ) \big\|_{L_x^\infty}  \\
&\lesssim 1\{K_2\sim N_2\} N_1^{\frac{1}{2}+\delta} \| \Wt \|_{\cC^{-\frac{1}{2}-\delta}_x} \Big(  \sum_{L\lesssim N_2 } L^{-\beta} \Big) \| f_2 \|_{L_x^\infty} \| P_{N_2} \Wt \|_{L_x^\infty} \\ 
&\lesssim 1\{K_2\sim N_2\} N_1^{\frac{1}{2}+\delta} N_2^{\frac{1}{2}+\delta} \| \Wt \|_{\cC^{-\frac{1}{2}-\delta}_x}^2 \\
&\lesssim (N_1 N_2)^{-\delta} K_1^\gamma K_2^\gamma \| \Wt \|_{\cC^{-\frac{1}{2}-\delta}_x}^2 . 
\end{align*}
In the last line, we also used  \( N_1 \lesssim K_1 \) and \( \gamma > 1/2 \). If \( L \gtrsim N_2 \), we simply estimate 
\begin{align*}
&N_1^{\frac{1}{2}+\delta} \| \Wt \|_{\cC^{-\frac{1}{2}-\delta}_x}  \sum_{N_2 \lesssim L\lesssim K_1} L^{-\beta} \big \| \widetilde{P}_L (P_{N_2}\Wt P_{K_2} f_2 ) \big\|_{L_x^\infty} \\
&\lesssim N_1^{\frac{1}{2}+\delta}  \| \Wt \|_{\cC^{-\frac{1}{2}-\delta}_x} \Big(  \sum_{N_2 \lesssim L\lesssim K_1} L^{-\beta} \Big) \| P_{N_2} \Wt \|_{L_x^\infty} \| P_{K_2} f_2 \|_{L_x^\infty} \\
&\lesssim N_1^{\frac{1}{2}+\delta} N_2^{\frac{1}{2}-\beta+\delta}  \| \Wt \|_{\cC^{-\frac{1}{2}-\delta}_x}^2 \\
&\lesssim (N_1 N_2)^{-\delta} K_1^\gamma   \| \Wt \|_{\cC^{-\frac{1}{2}-\delta}_x}^2,
\end{align*}
provided that \( \gamma > \max( 1-\beta, 1/2 ) \). This completes the estimate in Case 1, i.e., \( N_1 \not \sim N_2 \). \\

\emph{Case 2: \( N_1 \sim N_2 \).} This is the more difficult case. Guided by the uncertainty principle, we decompose the interaction potential by writing  \( V = P_{\ll N_1} V + P_{\gtrsim N_1} V \). Using the linearity of the multiplier \( \MT[V;N_1,N_2] \) in \( V \), we decompose 
\begin{align*}
&\int_{\T^3} V* (P_{N_1}\Wt \, P_{K_1} f_1 ) ~ P_{N_2}\Wt P_{K_2} f_2 \dx - \int_{\T^3} \big( \MT[V;N_1,N_2] P_{K_1}f_1 \big)P_{K_2} f_2 \dx \\
&= \int_{\T^3} (P_{\ll N_1 } V)* (P_{N_1}\Wt \, P_{K_1} f_1 ) ~ P_{N_2}\Wt P_{K_2} f_2 \dx - \int_{\T^3} \big( \MT[P_{\ll N_1} V;N_1,N_2] P_{K_1}f_1 \big)P_{K_2} f_2 \dx \\
&+ \int_{\T^3} (P_{\gtrsim N_1 } V)* (P_{N_1}\Wt \, P_{K_1} f_1 ) ~ P_{N_2}\Wt P_{K_2} f_2 \dx - \int_{\T^3} \big( \MT[P_{\gtrsim N_1} V;N_1,N_2] P_{K_1}f_1 \big)P_{K_2} f_2 \dx.
\end{align*}
We now split the proof into two subcases corresponding to the contributions of \( P_{\ll N_1} V \) and \( P_{\gtrsim N_1} V \). \\

\emph{Case 2.a: \( N_1 \sim N_2\), contribution of \( P_{\ll N_1} V \).} Similar as in Case 1, we do not rely on any cancellation between the convolution term and its renormalization. As a result, we estimates both terms separately. \\
We first estimate the convolution term. Due to the convolution with \( P_{\ll N_1} V \), we only obtain a non-zero contribution if \( N_1 \sim K_1 \). Using \( N_1 \sim N_2 \) in the second inequality below, we obtain that 
\begin{align*}
&\Big|  \int_{\T^3} (P_{\ll N_1 } V)* (P_{N_1}\Wt \, P_{K_1} f_1 ) ~ P_{N_2}\Wt P_{K_2} f_2 \dx \Big| \\
&\lesssim 1\{N_1 \sim K_1\}  \| (P_{\ll N_1 } V)* (P_{N_1}\Wt \, P_{K_1} f_1 ) \|_{L_x^1} \| \widetilde{P}_{\ll N_1} ( P_{N_2}\Wt P_{K_2} f_2  ) \|_{L_x^\infty}  \\ 
&\lesssim 1\{N_1 \sim K_1\}  1\{N_2 \sim K_2\} \| P_{N_1} \Wt \|_{L_x^\infty}  \| f_1 \|_{L_x^1} \| P_{N_2} \Wt \|_{L_x^\infty} \| f_2 \|_{L_x^\infty} \\
&\lesssim 1\{N_1 \sim K_1\}  1\{N_2 \sim K_2\} (N_1 N_2)^{\frac{1}{2}+\delta} \| \Wt \|_{\cC_x^{-\frac{1}{2}-\delta}}^2 \\
&\lesssim (N_1 N_2)^{-\delta} (K_1 K_2)^\gamma  \| \Wt \|_{\cC_x^{-\frac{1}{2}-\delta}}^2 . 
\end{align*}
Second, we turn to the multiplier term. From the definition of \( \MT[P_{\ll N_1} V; N_1,N_2] \) (see Definition \ref{measure_re:definition_MT_decomposition}), we see that the corresponding symbol is supported on frequencies \( |n|\sim N_1 \). As a result, we only obtain a non-zero contribution if \( K_1 \sim K_2 \sim N_1 \). Using Lemma \ref{measure_re:lemma_physical_MT}, Hölder's inequality, Young's inequality, and the trivial estimate $\|\Cort \|_{L_x^\infty} \lesssim N_1$, we obtain
\begin{align*}
&\Big|\int_{\T^3} \big( \MT[P_{\ll N_1} V;N_1,N_2] P_{K_1}f_1 \big)P_{K_2} f_2 \dx \Big|\\
&= 1\{K_1\sim K_2 \sim N_1\} \Big| \int_{\T^3} \Big( ( \Cort P_{\ll N_1} V ) * P_{K_1} f_1 \Big) P_{K_2} f_2 \dx \Big| \\
&\lesssim 1\{K_1\sim K_2 \sim N_1\} \| ( \Cort P_{\ll N_1} V ) * P_{K_1} f_1\|_{L_x^1}  \|  P_{K_2} f_2 \|_{L_x^\infty} \\
&\lesssim 1\{K_1 \sim K_2\sim N_1\} \| \Cort P_{\ll N_1} V\|_{L_x^1} \| f_1\|_{L_x^1} \| f_2 \|_{L_x^\infty} \\
&\lesssim 1\{K_1 \sim K_2 \sim N_1\} \|  \Cort \|_{L_x^\infty}  \| V \|_{L_x^1} \\
&\lesssim 1\{K_1 \sim K_2 \sim N_1\}  N_1 \lesssim (N_1 N_2)^{-\delta} (K_1 K_2)^\gamma. 
\end{align*}
This completes the estimate of the contribution from \( P_{\ll N_1} V \). \\

\emph{Case 2.b: \( N_1 \sim N_2 \), contribution of \( P_{\gg N_1} V \).} The estimate for this case relies on the cancellation between the convolution and multiplier term, i.e., the renormalization. One important ingredient lies in the estimate \( \| P_{\gg N_1} V \|_{L_x^1} \lesssim N_1^{-\beta} \), which yields an important gain. \\
Using the translation operator \( \tau_y \), we rewrite the convolution term as 
\begin{align*}
& \int_{\T^3} (P_{\gtrsim N_1 } V)* (P_{N_1}\Wt \, P_{K_1} f_1 ) ~ P_{N_2}\Wt P_{K_2} f_2 \dx\\
&=\int_{\T^3} P_{\gtrsim N_1} V(y) \bigg[  \int_{\T^3} P_{K_1} f_1(x-y)  P_{K_2} f_2(x)  P_{N_1} \Wt(x-y) P_{N_2} \Wt(x) \dx \bigg] \dy \\
&= \int_{\T^3} P_{\gtrsim N_1} V(y)  \bigg[ \int_{\T^3} \big(\tau_y P_{K_1} f_1\, P_{K_2} f_2 \big)(x) \big( \tau_y P_{N_1} \Wt \, P_{N_2} \Wt\big)(x) \dx \bigg] \dy. 
\end{align*}
Using Lemma \ref{measure_re:lemma_physical_MT}, we obtain that 
\begin{align*}
 &\int_{\T^3} \big( \MT[P_{\gtrsim N_1} V;N_1,N_2] P_{K_1}f_1 \big)P_{K_2} f_2 \dx \\
&= \int_{\T^3} \Big( \big( \Cort P_{\gg N_1} V \big) * P_{K_1} f_1\Big)(x) P_{K_2} f_2(x) \dx \\
&= \int_{\T^3} P_{\gg N_1} V(y) \bigg[ \int_{\T^3} \big(\tau_y P_{K_1} f_1\, P_{K_2} f_2 \big)(x) \Cort(y) \dx \bigg] \dy. 
\end{align*}
By recalling Definition \ref{measure_re:definition_renormalization_translation} and combining both identities, we obtain that
\begin{align*}
 &\int_{\T^3} (P_{\gtrsim N_1 } V)* (P_{N_1}\Wt \, P_{K_1} f_1 ) ~ P_{N_2}\Wt P_{K_2} f_2 \dx- \int_{\T^3} \big( \MT[P_{\gtrsim N_1} V;N_1,N_2] P_{K_1}f_1 \big)P_{K_2} f_2 \dx \\
&= \int_{\T^3} P_{\gg N_1} V(y) \bigg[ \int_{\T^3} \big(\tau_y P_{K_1} f_1\, P_{K_2} f_2 \big)(x) \lcol (\tau_y P_{N_1} \Wt) P_{N_2} \Wt \rcol \hspace{-0.5ex}(x)\dx \bigg] \dy. 
\end{align*}
Using that \( \lcol (\tau_y P_{N_1} \Wt) P_{N_2} \Wt \rcol \hspace{-0.5ex}(x) \) is supported on frequencies \( \lesssim N_1 \), we obtain that 
\begin{align*}
&\bigg| \int_{\T^3} P_{\gg N_1} V(y) \bigg[ \int_{\T^3} \big(\tau_y P_{K_1} f_1\, P_{K_2} f_2 \big)(x) \lcol (\tau_y P_{N_1} \Wt) P_{N_2} \Wt \rcol \hspace{-0.5ex}(x)\dx \bigg] \dy \bigg| \\
&\lesssim \| P_{\gtrsim N_1} V(y) \|_{L_y^1} \sup_{y\in \T^3} \Big( \sum_{L\lesssim N_1} L^{1+\delta} \| P_L\big( (\tau_y P_{K_1} f_1) P_{K_2} f_2\big) \|_{L_x^1} \Big) \sup_{y \in \T^3} \| \lcol (\tau_y P_{N_1} \Wt) P_{N_2} \Wt \rcol \|_{\cC^{-1-\delta}_x}  \\
&\lesssim N_1^{-\beta} \,   \Big( \hspace{-2ex} \sum_{ \substack{L\lesssim N_1 \\ L \lesssim \max(K_1,K_2)}} \hspace{-2ex} L^{1+\delta} \Big)  \,  \| f_1 \|_{L_x^1} \| f_2 \|_{L_x^\infty}  \sup_{y \in \T^3} \| \lcol (\tau_y P_{N_1} \Wt) P_{N_2} \Wt \rcol \|_{\cC^{-1-\delta}_x} \\
&\lesssim (N_1 N_2)^{-\delta} \max(K_1,K_2)^\gamma \sup_{y \in \T^3} \| \lcol (\tau_y P_{N_1} \Wt) P_{N_2} \Wt \rcol \|_{\cC^{-1-\delta}_x} . 
\end{align*}
This completes the estimate of the contribution from \( P_{\gg N_1} V \) and hence the proof of the proposition.
\end{proof}

\subsection{Proof of Proposition \ref{measure_var:proposition_main} and  Corollary \ref{measure_var:corollary_cT}}\label{section:proof_variational}

In this subsection, we reap the benefits of our previous work and prove the main results of this section. 

\begin{proof}[Proof of Proposition \ref{measure_var:proposition_main}:] In this proof, we treat \( Q_T= Q_T(W,\varphi,\lambda) \) like an implicit constant and omit the dependence on \( W,\varphi\), and \( \lambda\). In particular, its precise definition may change throughout the proof.\\
From the quartic binomial formula (Lemma \ref{measure_re:lemma_binomial}), it follows that
\begin{equation*}
\begin{aligned}
 &  \varphi( W + I(u)) + \lcol \cVT( \Winf + \Iinf(u) )\rcol + \frac{1}{2} \| u \|_{L^2}^2  \\
  &= \lambda \int_{\T^3}  \lcol (V*(\Winf)^2) (\Winf) \rcol \Iinf[u] \dx +    \frac{\lambda}{4}\int_{\T^3} (V*(\Iinf[u])^2) (\Iinf[u])^2 \dx 
 + \frac{1}{2} \| u \|_{L^2}^2 \\
&+\frac{\lambda}{4} \int_{\T^3} \lcol (V*(\Winf)^2) (\Winf)^2 \rcol \dx + \cT +  \varphi( W + I(u))+ \frac{\lambda}{2}\int_{\T^3} ( V* \lcol (\Winf)^2 \rcol) (\Iinf[u] )^2 \dx  \\
&+ \lambda \int_{\T^3} \hspace{-0.5ex}\Big[ (V* (\Winf \Iinf[u] )) \Winf\Iinf[u] - (\MT \Iinf[u]  ) \Iinf[u] \Big] \hspace{-0.5ex}\dx + \lambda \int_{\T^3} \hspace{-1ex} (V* (\Iinf[u] )^2) \Iinf[u]  \Winf \dx . 
\end{aligned}
\end{equation*}
We have grouped the terms according to their importance and their degree in \( \Iinf[u] \). The first line consists of the main terms, whereas the second and third line consist of less important terms of increasing degree in \( \Iinf[u] \). We will split them further in  \eqref{measure_var:eq_E0}-\eqref{measure_var:eq_E3} below and introduce notation for the individual terms. \\
Since \( \lcol (V*(\Winf)^2) \Winf \rcol \) has regularity \( \min(- \frac{3}{2}+\beta, -\frac{1}{2})-\) and \( \Iinf[u] \) has regularity \( 1 \), the term 
\begin{equation*}
\lambda \int_{\T^3} \lcol (V*(\Winf)^2) \Winf \rcol \Iinf[u] \dx
\end{equation*}
is potentially unbounded as \( T \rightarrow \infty \). As in \cite{BG18}, we absorb it into the quadratic term \( \frac{1}{2} \|u\|_{L^2}^2 \). To this end, we want to remove the integral in \( \Iinf[u] \) and obtain an expression in the drift \( u \). From Itô's formula, it holds that 
\begin{align*}
&\lambda \int_{\T^3} \lcol (V*(\Winf)^2) \Winf \rcol \Iinf[u] \dx \\
&= \lambda \int_0^T \int_{\T^3} \lcol (V*(\Wt)^2) \Wt \rcol \Jt u_t \dx \mathrm{d}t + \lambda \int_0^T \int_{\T^3} \It[u] \, \mathrm{d}( \lcol (V*(\Wt)^2)\Wt\rcol).
\end{align*}
The second term is a martingale (in the upper limit of integration) and therefore has expectation equal to zero. Together with the self-adjointness of \( J_t \), it follows that
\begin{align*}
&\E_{\bP} \Big[ \lambda \int_{\T^3} \lcol (V*(\Winf)^2) \Winf \rcol \Iinf[u] \dx + \frac{1}{2} \| u \|_{L^2}^2  \Big] \\
&= \E_{\bP} \Big[ \lambda \int_0^T \int_{\T^3} \Jt \Big( \lcol (V*(\Wt)^2) \Wt \rcol \Big) u_t \dx \mathrm{d}t  + \frac{1}{2} \| u \|_{L^2}^2 \Big] \\
&= \E_{\bP}\Big[ \frac{1}{2}\Big \| \lP[u]\Big\|_{L^2}^2 - \frac{\lambda^2}{2} \Big\|  \Jt \Big( \lcol (V*(\Wt)) \Wt \rcol \Big) \Big\|_{L^2}^2 \Big],
\end{align*}
where \( l^T[u] \) is as in \eqref{measure_var:eq_lT}. To simplify the notation, we write 
\begin{equation}
w_t \defe l^T_t[u] = u_t + \lambda \Jt \Big( \lcol(V*(\Wt)^2) \Wt\rcol\Big). 
\end{equation}
With \( \bWt{[3]} \) as in \eqref{measure_re:eq_wt[3]}, it follows that 
\begin{equation}
\It[w] = \It[u] + \lambda \bWt{[3]}.
\end{equation}
By inserting this back into the quartic binomial formula, we obtain that 
\begin{equation}\label{measure_var:eq_identity}
\begin{aligned}
 &\E_{\bP} \big[ \varphi( W + I(u))+  \lcol \cVT( \Winf + \Iinf[u] )\rcol + \frac{1}{2} \| u \|_{\cH}^2 \big] \\
&= \E_{\bP} \big[\mathcal{E}_0 +\cT\big] + \E_{\bP}\big[\mathcal{E}_1 + \mathcal{E}_2 + \mathcal{E}_3\big] +\E_{\bP} \Big[ \frac{\lambda}{4} \int_{\T^3} (V* (\Iinf[w])^2) (\Iinf[w])^2 \dx + \frac{1}{2} \| w \|_{\cH}^2 \Big].  
\end{aligned}
\end{equation}
where the ``error'' terms \( \mathcal{E}_j \), with \( j=0,1,2,3\), are given by
\begin{align}
\mathcal{E}_0 &\defe \frac{\lambda}{4} \int_{\T^3} \lcol (V*(\Winf)^2) (\Winf)^2 \rcol \dx - \frac{\lambda^2}{2} \Big\| \Jt \big( \lcol (V * (\Wt)^2) \Wt \rcol \big) \Big\|_{L_t^2 L_x^2}^2   \label{measure_var:eq_E0}\\
&\hspace{3ex}+ \frac{\lambda^3}{2} \int_{\T^3} (V*\lcol (\Winf)^2 \rcol) (\bWinf{[3]})^2 \dx \notag \\
 &\hspace{3ex}+\lambda^3 \int_{\T^3} \Big( V* (\Winf \bWinf{[3]}) \Winf \bWinf{[3]} - ( \Minf \bWinf{[3]}) \bWinf{[3]} \Big) \dx  ,\notag \\
\mathcal{E}_1 &\defe  \varphi( W + I[u])  -\lambda^2 \int_{\T^3} (V* \lcol (\Winf)^2\rcol ) \bWinf{[3]} \Iinf[w] \dx \label{measure_var:eq_E1} \\
&\hspace{3ex}- 2 \lambda^2 \int_{\T^3} \Big( \big( V* (\Winf \bWinf{[3]})\big) \Winf - \Minf  \bWinf{[3]} \Big) \Iinf[w] \dx  , \notag \\
\mathcal{E}_2 &\defe \lambda \int_{\T^3} \Big( \big( V* (\Winf \Iinf[w])\big) \Winf \Iinf[w] - (\Minf \Iinf[w]) \Iinf[w] \Big) \dx \label{measure_var:eq_E2}  \\
&\hspace{3ex}+ \frac{\lambda}{2} \int_{\T^3} ( V * \lcol (\Winf)^2 \rcol) (\Iinf[w])^2 \dx,  \notag\\
\mathcal{E}_3 &\defe \lambda \int_{\T^3} \big( V * (\Iinf[w]-\lambda \bWinf{[3]})^2 \big) (\Iinf[w]- \lambda \bWinf{[3]}) \Winf \dx \label{measure_var:eq_E3}\\
&\hspace{3ex}+ \frac{\lambda}{4} \int_{\T^3} \Big( (V*(\Iinf[w]-\lambda \bWinf{[3]})^2) (\Iinf[w]-\lambda \bWinf{[3]})^2 - (V* (\Iinf[w])^2) (\Iinf[w])^2) \Big) \dx. \notag
\end{align}

Since \( \mathcal{E}_0\) does not depend on \( w \), we can define 
\begin{equation}\label{measure_var:eq_cT}
\cT \defe -\E_{\bP} \big[ \mathcal{E}_0 \big]. 
\end{equation}
The behavior of \( \cT \) as \( T \rightarrow \infty \)  is irrelevant for the rest of the proof. However, it determines whether the Gibbs measure is singular or absolutely continuous with respect to the Gaussian free field (see Section \ref{section:singularity}). From the estimates \eqref{appendix:eq_trilinear_I} and \eqref{appendix:eq_trilinear_II}, it is easy to see that
\begin{equation*}
 -Q_T + \frac{1}{2} \Big( \frac{\lambda}{4} \cV(\Iinf[u]) + \frac{1}{2} \| w \|_{\cH}^2  \Big)\leq \frac{\lambda}{4} \cV(\Iinf[w]) + \frac{1}{2} \| w \|_{\cH}^2 \leq Q_T + 2 \Big( \frac{\lambda}{4} \cV(\Iinf[u]) + \frac{1}{2} \| w \|_{\cH}^2  \Big).
\end{equation*}
Thus, it suffices to bound the terms in \( \mathcal{E}_1, \mathcal{E}_2, \) and \( \mathcal{E}_3 \)  pointwise by
\begin{equation*}
Q_T +  \frac{1}{8} \Big( \frac{\lambda}{4} \cV(\Iinf[w]) + \frac{1}{2} \| w \|_{\cH}^2 \Big).
\end{equation*}
We treat the individual summands separately. \\
\emph{Contribution of \(\mathcal{E}_1\):} 
For the first summand in \( \mathcal{E}_1 \), the linear growth of \( \varphi \), Sobolev embedding, a minor modification of \eqref{measure_re:eq_so_III_1}, and Lemma \ref{appendix:lemma_It} imply that 
\begin{equation}\label{measure_var:eq_e1_estimate}
\begin{aligned}
&| \varphi( W + I[u] )| \lesssim \| W \|_{\cC_t^0 \cC_x^{-\frac{1}{2}-\kappa}} + \| I[ J_t^T \big( \lcol (V * (\Wt)^2) \Wt \rcol  \big)]  \|_{\cC_t^0 \cC_x^{-\frac{1}{2}-\kappa}} + \| I[w] \|_{\cC_t^0 \cC^{-\frac{1}{2}-\kappa}_x}\\
&\lesssim Q_T + \| I[w] \|_{C_t H_x^1} 
\lesssim \frac{1}{\delta} Q_T  + \delta \| w \|_{\cH}^2. 
\end{aligned}
\end{equation}
For the second summand in \( \mathcal{E}_1\), we have  from Lemma \ref{measure_re:lemma_stochastic_objects_III} that
\begin{equation*}
\lambda \Big| \int_{\T^3} (V* \lcol (\Winf)^2 \rcol) \bWinf{[3]} \Iinf[w] \dx \Big| \lesssim  \lambda \|  (V* \lcol (\Winf)^2 \rcol) \bWinf{[3]}  \|_{C^{-1}_x} \| \Iinf[w] \|_{H^{1}_x} \lesssim \frac{1}{\delta} Q_T + \delta \| w \|_{\cH}^2. 
\end{equation*}
For the third summand in \( \mathcal{E}_1 \), we have from Lemma \ref{measure_re:lemma_stochastic_objects_III} and Lemma \ref{appendix:lemma_It} that
\begin{align*}
 &\lambda^2 \int_{\T^3} \Big( \big( V* (\Winf \bWinf{[3]})\big) \Winf - \Minf \bWinf{[3]} \Big) \Iinf[w] \dx \\
&\lesssim \lambda^2 \| \big( V* (\Winf \bWinf{[3]})\big) \Winf - \Minf \bWinf{[3]} \|_{C^{-1}_x} \| \Iinf[w]\|_{H^{1}_x} \lesssim \frac{1}{\delta} Q_T + \delta \| w \|_{\cH}^2.
\end{align*}
\emph{Contribution of \( \mathcal{E}_2 \):} For the first summand in \( \mathcal{E}_2\),  the random matrix estimate (Proposition \ref{measure_rmt:proposition_estimate}) implies for every \( 0 < \gamma < \min(\beta,\frac{1}{2}) \) that 
\begin{align*}
&\lambda \Big| \int_{\T^3} \Big( \big( V* (\Winf \Iinf[w])\big) \Winf\Iinf[w] - (\Minf \Iinf[w])\Iinf[w] \Big) \dx \Big|
 \lesssim Q_T \| \Iinf[w] \|_{H^{1-\gamma}_x}^2 \\
&\lesssim \frac{1}{\delta} Q_T + \delta \big( \lambda \| \Iinf[w] \|_{L^2_x}^4 + \| \Iinf[w]\|_{H^1_x}^2 \big) \lesssim  \frac{1}{\delta} Q_T + \delta \big( \lambda \cV(\Iinf[w])+ \| \Iinf[w]\|_{H^1_x}^2 \big).
\end{align*}
The second summand in \( \mathcal{E}_2\) can easily be controlled using Lemma \ref{measure_re:lemma_stochastic_objects_I}. \\
\emph{Contribution of \( \mathcal{E}_3 \):} We estimate the first summand in \( \mathcal{E}_3 \) by 
\begin{align*}
&\lambda \Big| \int_{\T^3} \big( V * (\Iinf[w]-\lambda \bWinf{[3]})^2 \big) (\Iinf[w]- \lambda \bWinf{[3]}) \Winf \dx \Big| \\
& \lesssim \lambda \| \Winf \|_{C^{-\frac{1}{2}-\delta}_x} \Big\| \langle \nabla \rangle^{\frac{1}{2}+\delta} \Big(   \big( V * (\Iinf[w]-\lambda \bWinf{[3]})^2 \big) (\Iinf[w]- \lambda \bWinf{[3]}) \Big) \Big\|_{L^1_x}. 
\end{align*} 
In the second factor, we bound the contribution of \( ( V* \Iinf[w]^2) \Iinf[w] \) using Lemma \ref{measure_var:lemma_cubic}. In contrast, the terms containing at least one factor of \( \bWt{[3]} \) can be controlled using Lemma \ref{measure_re:lemma_stochastic_objects_III}, \eqref{appendix:eq_trilinear_I} and \eqref{appendix:eq_trilinear_II}.  This leads to 
\begin{align*}
& \| \Winf \|_{C^{-\frac{1}{2}-\epsilon}_x} \Big\| \langle \nabla \rangle^{\frac{1}{2}+\delta} \Big(   \big( V * (\Iinf[w]-\lambda \bWinf{[3]})^2 \big) (\Iinf[w]- \lambda \bWinf{[3]}) \Big) \Big\|_{L^1_x} \\
&\lesssim \lambda Q_T\Big(1+ \cV(\Iinf[w]])^{\frac{1}{2}} \| \Iinf[w]\|_{L^2}^{1-\theta} \| \Iinf[w]\|_{H^1_x}^{\theta} + \| \Iinf[w]\|_{H^{\frac{1}{2}+2\delta}}^2 \Big) \\
&\lesssim Q_T + \delta \Big( \lambda \cV(\Iinf[w]) + \| w \|_{\cH}^2 \Big). 
\end{align*}
The second summand in \( \mathcal{E}_3 \) can be controlled using the same (or simpler) arguments. 
\end{proof}

Based on the proof of Proposition \ref{measure_var:proposition_main}, we can also determine the behavior as \( T \rightarrow \infty \) of the renormalization constants \( \cT \). In particular, we obtain a short proof of Corollary \ref{measure_var:corollary_cT}. 

\begin{proof}[Proof of Corollary \ref{measure_var:corollary_cT}:]
We  let \( \beta > 1/2 \) and choose any \( 1/2 < \gamma < \min( \beta,1) \). Using the definition of \( \cT \) in \eqref{measure_var:eq_cT}, it remains to control the expectation of \( \mathcal{E}_0 \), which is defined in \eqref{measure_var:eq_E0}. We treat the four terms in \( \mathcal{E}_0 \) separately. \\
The first term has zero expectation by Proposition \ref{measure_re:proposition_renormalization}. 
For the second term, we obtain from Corollary \ref{measure_re:corollary_stochastic_objects_I} that 
\begin{equation*}
\E_{\bP} \bigg[  \Big\| \Jt \big( \lcol (V * (\Wt)^2) \Wt \rcol \big) \Big\|_{L_t^2 L_x^2}^2 \bigg] \lesssim \sum_{n\in \Z^3} \int_0^\infty \frac{\sigmaT{t}(n)^2}{\langle n \rangle^2} \frac{1}{\langle n \rangle^{2\gamma}} \dt \lesssim \sum_{n\in \Z^3} \frac{1}{\langle n \rangle^{2+2\gamma}} \lesssim 1. 
\end{equation*}
For the third term, we obtain from Lemma \ref{measure_re:lemma_stochastic_objects_I} and Lemma \ref{measure_re:lemma_stochastic_objects_III} that 
\begin{equation*}
\bigg| \E_{\bP} \bigg[ \int_{\T^3} (V*\lcol (\Winf)^2 \rcol) (\bWinf{[3]})^2 \dx \bigg] \bigg| \lesssim \E_{\bP} \Big[ \| V*\lcol (\Winf)^2 \|_{\cC_x^{-1/2}} \| \bWinf{[3]} \|_{\cC_x^{\gamma}}^2 \Big] \lesssim 1. 
\end{equation*}
For the fourth term, we obtain from Lemma \ref{measure_re:lemma_stochastic_objects_III} and the random matrix estimate (Proposition \ref{measure_rmt:proposition_estimate}) that
\begin{align*}
\bigg| \E_{\bP} \bigg[ \int_{\T^3} \Big( V* (\Winf \bWinf{[3]}) \Winf \bWinf{[3]} - ( \Minf \bWinf{[3]}) \bWinf{[3]} \Big) \dx  \bigg] \bigg| \lesssim \E_{\bP} \Big[ \operatorname{Op}^T_\infty(\gamma,2)  \| \bWinf{[3]}\|_{H_x^\gamma}^2 \Big] \lesssim 1. 
\end{align*}
This completes the argument. 
\end{proof}

\section{\protect{The reference and drift measures}}\label{section:drift_measure}

In this section, we prove Theorem \ref{theorem:reference}, which contains information regarding the reference measures. In this paper, we will use the reference measure $\nu_\infty$ to prove the singularity of the Gibbs measure (Theorem \ref{theorem:singularity}). In the second part of this series, the reference measures will play an essential role in the probabilistic local well-posedness theory. \\

As in previous sections, we replace the truncation parameter $N$ by $T$. 
Due to its central importance, let us provide an informal description of the terms in the representation of $\nu_T$. The first summand follows the distribution of the Gaussian free field, which has independent Fourier coefficients and regularity \( -1/2-\). The second summand is a cubic Gaussian chaos with regularity \( {\min(1/2 + \beta,1)-} \). Finally, the third summand is a Gaussian chaos of order \( n \) with regularity \( 5/2-\). 

The statement of Theorem \ref{theorem:reference} is concerned with measures on \( \cC_x^{-1/2-\kappa}(\T^3) \). At this point, it should not be surprising to the reader that the proof mostly uses the lifted measures \( \tmuT \) and \( \tmu_\infty \). We will construct a reference measure \( \QT \) for \( \tmuT \), and the reference measure \( \REFT \) will be given by the pushforward of \( \QT \) under \( W_\infty \). Since the main tool in the construction of \( \QT \) is Girsanov's theorem, we call \( \QT \) the drift measure. 
This section is a modification of the arguments in Barashkov and Gubinelli's paper \cite{BG20}. Since \( \lP[u] \) in Proposition \ref{measure_var:proposition_main} is simpler than in the \( \Phi^4_3\)-model, however, we obtain slightly stronger results. For instance, we prove \( L^q\)-bounds for the density \( \DT \) in \eqref{measure_drift:eq_DT}, whereas the analogous density in \cite{BG20} only satisfies ``local'' \( L^q\)-bounds.

\subsection{Construction of the drift measure}

We define the forcing term 
\begin{equation}\label{measure_drift:eq_xi}
\XiT(\WP)_t \defe - \lambda \Jt \Big( \lcol (V * (\Wt)^2 ) \Wt \rcol \Big) + \Jt \langle \nabla \rangle^{-\frac{1}{2}} \lcol (\langle \nabla \rangle^{-\frac{1}{2}} \Wt)^n \rcol,  
\end{equation}
where \( n \) is a large odd integer depending on \( \beta \). The first summand in \eqref{measure_drift:eq_xi} is the main term. The second summand in \eqref{measure_drift:eq_xi} yields necessary coercivity in the proof of Lemma \ref{measure_drift:lemma_no_explosion_P} and Proposition \ref{measure_drift:proposition_Lp}, but can be safely ignored for most of the argument. We define the drift \( \uP \) through the integral equation
\begin{equation}\label{measure_drift:eq_integraleq}
\begin{aligned}
\ut &= \XiT(\WP-\IP[\uP])_t \\
&= - \lambda \Jt \Big(\hspace{-0.8ex} \lcol (V*(\Wt - \It[\uP])^2) (\Wt - \It[\uP]) \rcol \hspace{-0.8ex}\Big) \hspace{-0.3ex} +  \hspace{-0.3ex} \Jt   \hspace{-0.3ex} \langle \nabla \rangle^{-\frac{1}{2}} \lcol \hspace{-0.2ex} \Big( \langle \nabla \rangle^{-\frac{1}{2}} \big( \Wt - \It[\uP]\big)\Big)^{n} \hspace{-0.3ex} \rcol \hspace{-0.4ex}. 
\end{aligned}
\end{equation}
We also define the drift \( u \), which does not contain any regularization in the interaction, by 
\begin{equation}\label{measure_drift:eq_integraleq_u}
u_t= - \lambda J_t \Big( \lcol (V*(W_t - I_t[u])^2) (W_t - I_t[u]) \rcol \Big) \hspace{-0.3ex} +  J_t   \langle \nabla \rangle^{-\frac{1}{2}} \lcol \Big( \langle \nabla \rangle^{-\frac{1}{2}} \big( W_t - I_t[u]\big)\Big)^{n} \rcol . 
\end{equation}
Using the binomial formulas (Lemma \ref{measure_re:lemma_binomial} and Lemma \ref{measure_re:lemma_high_power}), we see that the integral equation has smooth coefficients on every compact subset of \( [0,\infty) \times \T^3 \). As a result, it can be solved locally in time using standard ODE-theory. Due to the polynomial nonlinearity, however, we will need to rule out finite-time blowup. To this end, we introduce  the blow-up time \( T_{\exp}[\uP] \in (0,\infty] \), which we will later show to be infinite almost surely with respect to both \( \bP \) and \( \QT \).  The reason is that the highest-degree term in \eqref{measure_drift:eq_integraleq}, which is given by 
\( - \Jt \langle \nabla \rangle^{-1/2} ( \langle \nabla \rangle^{-1/2}  \It[\uP])^{n} \hspace{-0.3ex}  \), is defocusing. 
We also introduce the stopping time
\begin{equation}
\tauTN \defe \inf \Big \{ t \in [0,\infty) \colon \int_0^t \| \us  \|_{L^2_x}^2 \ds = N \Big\}. 
\end{equation}
From the integral equation, it is clear that \( \ut(\cdot) \) is supported in frequency space on the finite set \( \{ n\in \Z^3\colon \| n\|\lesssim \langle t \rangle \} \). As a result, the \( L_t^2 L_x^2\)-norm can be used  as a blow-up criterion and the solution \( \ut \)  exists for all times \( t \leq \tauTN \), i.e., \( T_{\exp}[\uP]>\tauTN \). We then define the truncated solution by
\begin{equation}
\uNt \defe 1\{t\leq \tauTN\}\, \ut. 
\end{equation}
From the definition of \( \tauTN \), it follows that 
\begin{equation*}
\int_0^\infty \| \uNs \|_{L^2_x}^2 \ds \leq N.
\end{equation*}
Thus, \( \uNP \) satisfies Novikov's condition and we can define the shifted probability measure \( \QTN \) by
\begin{equation}\label{measure_drift:eq_QuN}
 \frac{\d \QTN~}{\d \bP~ }  = \exp\Big( \int_0^\infty \int_{\T^3} \uNs \d B_s - \frac{1}{2} \int_0^\infty \| \uNs \|_{L^2}^2 \ds \Big). 
\end{equation}
Here, the \( L^2_x\)-pairing in the integral \( \int_0^\infty \int_{\T^3} \uNs \d B_s  \) is implicit, i.e., 
\begin{equation*}
\int_0^\infty \int_{\T^3} \uNs \d B_s  = \int_0^\infty \langle \uNs , \d B_s \rangle_{L^2_x(\T^3)} = \sum_{\substack{n_1,n_2\in\Z^3\colon \\ n_1+n_2=0}} \int_0^\infty \widehat{\uNs}(n_1) \, \d B_s^{n_2}. 
\end{equation*}
We emphasize that the stochastic integral  \(  \int_0^\infty \int_{\T^3} \uNs \d B_s \) only depends on the Brownian process \( B \) through the Gaussian process \(  W \). This is important in order to view \( \QTN \) as a measure on \( \cC_t^0 \cC_x^{-1/2-\kappa}([0,\infty]\times \T^3) \) without changing the expression for the density. To make this direct dependence on \( W \) clear, we note that \( \uP \) and hence \( \tauTN \) are functions of \( \WP \), and hence \( W \), directly from their definition. By using the definition of \( \uP \), the self-adjointness of \( \Jt \), and \( \d \Ws = \Js \d B_s \), we obtain that 
\begin{align*}
&\int_0^\infty \int_{\T^3} \uNs \d B_s \\
&= \int_0^\infty \int_{\T^3} \Big( - \lambda \lcol ( V* (\Wt - \It[\uP])^2)  (\Wt - \It[\uP]) \rcol +\langle \nabla \rangle^{-\frac{1}{2}} \lcol \big( \langle \nabla \rangle^{-\frac{1}{2}}  (\Wt - \It[\uP]) \big)^n \rcol \Big) \d \Ws. 
\end{align*}
The expression on the right-hand side clearly is a function of \( \WP \) and hence \( W \). With a slight abuse of notation, we will keep writing the integral with respect to \( \mathrm{d}B_s \), since it is more compact. \\

By Girsanov's theorem, the process 
\begin{equation}
\BtuTN \defe B_t - \int_0^t \uNs \ds
\end{equation}
is a cylindrical Brownian motion under \( \QTN \). In particular, the law of \( \BtuTN \) under \( \QTN \) coincides with the law of \( B_t \) under \( \bP \). As a consequence, the process 
\begin{equation}\label{measure_drift:eq_WtuTN}
\WtuTN \defe W_t  - \int_0^t J_s \uNs \ds = W_t - I_t[\uNP]
\end{equation}
satisfies   
\begin{equation}\label{measure_drift:eq_law_WuN}
\Law_{\QTN}(\WtuTN) = \Law_{\bP}(W). 
\end{equation}
To avoid confusion, let us remark on a technical detail. In the definition \eqref{measure_drift:eq_WtuTN}, the drift \( \uNs \) is supported on frequencies \( |n|\lesssim \langle T \rangle \). The right-hand side of \eqref{measure_drift:eq_WtuTN}, however, does not contain a further frequency projection. In particular, \( W \) and hence \( \WuTN \) contain arbitrarily high frequencies. This is related to the definition of the truncated Gibbs measure \( \muT \), where the density only depends on frequencies \( \lesssim \langle T \rangle \), but whose samples contain arbitrarily high frequencies. Put differently, we regularize the interaction but not the samples themselves. To make notational matters even worse, while \( \WuTN \) contains all frequencies, we will often work with \( \rho_T(\nabla) \WuTN \), which only contains frequencies \( \lesssim \langle T \rangle \). Similar as in Section \ref{section:stochastic_control}, we define the truncated process \( \WtTuTN \) by 
\begin{equation}
\WtTuTN \defe \rho_T(\nabla) \WtuTN . 
\end{equation}
Due to the integral equation \eqref{measure_drift:eq_integraleq}, we have that 
\begin{equation}\label{measure_drift:eq_integraleq_N}
\uNt= 1\{ t \leq \tauTN\} \Big[ -\lambda \Jt \Big( \lcol ( V* (\WtTuTN)^2)\WtTuTN \rcol \Big) + \Jt \langle \nabla \rangle^{-\frac{1}{2}} \lcol \big( \langle \nabla \rangle^{-\frac{1}{2}} \WtTuTN \big)^n \rcol \Big].
\end{equation}
We intend to use \( \QTN \) (and the limit as \( N \rightarrow \infty \)) as a reference measure for \( \tmuT \). Due to \eqref{measure_drift:eq_law_WuN}, the law of \( \WtTuTN\) under \( \QTN \) does not depend on \( N \). In our estimates of \( \uNt \) through the integral equation, it is therefore natural to view \(\WtTuTN\) as given. Under this perspective, the right-hand side of \eqref{measure_drift:eq_integraleq_N} no longer  depends on \( \uP \) and yields an explicit expression for \( \uP \). For comparison, the corresponding equation in the \( \Phi^4_3\)-model (cf. \cite[(14)]{BG20}) is a linear integral equation. 
We now start to estimate the drift \(  \uP \). 

\begin{lemma}\label{measure_drift:lemma_uN_1}
For all \( 1 \leq M \leq N \), all \( S \geq 0 \), and all \( 0 < \gamma < \min(1,\beta) \), it holds that 
\begin{equation}
\E_{\QTN} \Big[ \int_0^{\tau_M \wedge S} \| \us \|_{L^2}^2 \ds \Big] \lesssim \max( S^{1-2\gamma},1). 
\end{equation}
In particular, it holds that 
\begin{equation}
\QTN \big( \tauTM \leq S \big) \lesssim \frac{\max(S^{1-2\gamma},1)}{M}. 
\end{equation}
\end{lemma}

\begin{proof}
We recall from the definition of the drift measure that
\begin{equation*}
\Law_{\QTN}(\WuTN) = \Law_{\bP}(W) \qquad \text{and} \qquad \Law_{\QTN}(\WTuTN) = \Law_{\bP}(\WP)
\end{equation*}
As a result, we obtain that 
\begin{align*}
&\E_{\QTN} \Big[ \int_0^{\tau_M \wedge S} \| \us \|_{L^2}^2 \ds \Big]\\
&\leq \E_{\bP} \Big[ \int_0^S  \Big\| \lambda \Js \big( \lcol (V*(\Ws)^2) \Ws \rcol \big) + \Js \langle \nabla \rangle^{-\frac{1}{2}} \lcol \big( \langle \nabla \rangle^{-\frac{1}{2}} \Ws\big)^n \rcol \Big\|_{L^2}^2 \ds \Big] \\
&\lesssim \lambda^2  \E_{\bP} \Big[ \int_0^S   \Big\| \lambda \Js \big( \lcol (V*(\Ws)^2) \Ws \rcol \big) \Big\|_{L^2}^2 \ds \Big] + \E \Big[ \int_0^S  \Big \| \Js \langle \nabla \rangle^{-\frac{1}{2}} \lcol \big( \langle \nabla \rangle^{-\frac{1}{2}} \Ws\big)^n \rcol \Big\|_{L^2}^2 \ds \Big].
\end{align*}
For the first summand, we obtain from the definition of \( \Js \) and Lemma \ref{measure_re:lemma_stochastic_objects_I} that
\begin{align*}
\E_{\bP} \Big[ \int_0^S  \Big\| \lambda \Js \big( \lcol (V*(\Ws)^2) \Ws \rcol \big) \Big\|_{L^2}^2 \ds \Big] 
&\lesssim \Big(\int_0^S \langle t \rangle^{-2\gamma} \dt \Big) \sup_{t\geq 0} \E \big[ \big\| \lcol  (V* (\Wt)^2) \Wt \rcol \big\|_{H^{-\frac{3}{2}+\gamma}}^2 \big] \\
&\lesssim \max( S^{1-2\gamma},1). 
\end{align*}
For the second summand, we obtain from Lemma \ref{measure_re:lemma_high_power} that 
\begin{align*}
 \E \Big[ \int_0^S \Big \| \Js \langle \nabla \rangle^{-\frac{1}{2}} \lcol \big( \langle \nabla \rangle^{-\frac{1}{2}} \Ws\big)^n \rcol \Big\|_{L^2}^2 \ds \Big] 
\lesssim \Big ( \int_0^S \langle t \rangle^{-4+2\epsilon} \dt \Big) \sup_{t\geq 0} \E \big[ \big\| \lcol (\langle \nabla \rangle^{-\frac{1}{2}} W_t)^n \rcol \big\|_{H^{-\epsilon}}^2 \big] 
\lesssim 1. 
\end{align*}
This yields the desired estimate. 
\end{proof}

\begin{lemma}\label{measure_drift:lemma_uN_2}
For all \( 1 \leq M \leq N \), \( 1 \leq p <\infty\), and \( \gamma< \min(1/2,\beta)\), it holds that 
\begin{equation*}
\sup_{T,t \geq 0} \Big( \E_{\QTN} \big[ \| I_t[\uMP] \|_{\cC^{\frac{1}{2}+\gamma}_x(\T^3)}^p \big] \Big)^{\frac{1}{p}} \lesssim_p 1 . 
\end{equation*}
Furthermore, we have that for any \( 0<\alpha < 1 \) and \( 0< \eta < 1/2 \) that 
\begin{equation}\label{measure_drift:eq_uN_2_2}
\sup_{T\geq 0} \Big( \E_{\QTN}\big[ \| I[\uMP]\|_{\cC_t^{\alpha,\eta} \cC_x^0 ([0,\infty]\times \T^3)}^p \big] \Big)^{\frac{1}{p}} \lesssim_p 1,
\end{equation}
where the \( \cC_t^{\alpha,\eta} \cC_x^0 \)-norm is as in \eqref{notation:eq_Cweighted}. 
\end{lemma}

The proof of Lemma \ref{measure_drift:lemma_uN_2} is easier than its counterpart \cite[(16)]{BG20} in the \( \Phi^4_3\)-model, which requires a Gronwall argument. The second estimate \eqref{measure_drift:eq_uN_2_2} is needed for technical reasons related to tightness, and we encourage the reader to skip its proof on first reading.

\begin{proof}
The argument is similar to the proof of Lemma \ref{measure_drift:lemma_uN_1}. 
From the definition of \( \uMP \) and \( \uNP \), we have that 
\begin{equation}
\uMs = 1 \{ s \leq \tauTM \} \uNs. 
\end{equation}
Thus, we obtain that 
\begin{equation}
\| I_t[\uMP] \|_{\cC_x^{\frac{1}{2}+\gamma}} \leq \int_0^{t \wedge \tauTM} \| J_s \uNs \|_{\cC_x^{\frac{1}{2}+\gamma}} \ds \leq \int_0^t \| J_s \uNs \|_{\cC_x^{\frac{1}{2}+\gamma}} . 
\end{equation}
Using the integral equation \eqref{measure_drift:eq_integraleq} again, we obtain that 
\begin{equation}\label{measure_drift:eq_uN_2_p1}
\begin{aligned}
\| I_t[\uMP] \|_{\cC_x^{\frac{1}{2}+\gamma}}  &\leq \lambda \int_0^t \| J_s \Js \lcol (V*(\WsTuTN)^2 ) \WsTuTN\rcol \|_{\cC_x^{\frac{1}{2}+\gamma}}\ds\\
& + \int_0^t \| J_s \Js \langle \nabla \rangle^{-\frac{1}{2}} \lcol \big( \langle \nabla \rangle^{-\frac{1}{2}} \WsTuTN\big)^n \rcol \|_{\cC_x^{\frac{1}{2}+\gamma}} \ds. 
\end{aligned}
\end{equation}
Using that
\begin{equation*}
\Law_{\QTN}(\WuTN) = \Law_{\bP}(W),
\end{equation*}
we obtain from Lemma \ref{measure_re:lemma_stochastic_objects_III} and Lemma \ref{measure_re:lemma_high_power} that
\begin{align*}
&\Big( \E_{\QTN} \big[ \| I_t[\uMP] \|_{\cC_x^{\frac{1}{2}+\gamma}}^p \big] \Big)^{\frac{1}{p}}\\
&\lesssim \lambda \int_0^t \big( \E_{\bP} \| J_s \Js \lcol (V* (\Ws)^2) \Ws \rcol \|_{\cC_x^{\frac{1}{2}+\gamma}}^p \big)^{\frac{1}{p}} \ds + \int_0^t \big( \E_{\bP} \| J_s \Js \langle \nabla \rangle^{-\frac{1}{2}} \lcol \big( \langle \nabla \rangle^{-\frac{1}{2}} \Ws \big)^n \rcol \|_{\cC_x^{\frac{1}{2}+\gamma}}^p \big)^{\frac{1}{p}} \ds \\
&\lesssim_p \int_0^t \langle  s \rangle^{-1+\gamma-\min(1/2,\beta)+\delta} \ds + \int_0^t \langle s \rangle^{-3+\gamma+\delta} \ds \\
&\lesssim_p 1 .
\end{align*} 
This completes the proof of the first estimate. The second estimate \eqref{measure_drift:eq_uN_2_2} follows from a minor modification of the proof. To simplify the notation, we set 
\begin{equation*}
A(s) \defe \| J_s \Js \lcol (V*(\WsTuTN)^2 ) \WsTuTN\rcol \|_{L_x^\infty} + \| J_s \Js \langle \nabla \rangle^{-\frac{1}{2}} \lcol \big( \langle \nabla \rangle^{-\frac{1}{2}} \WsTuTN\big)^n \rcol \|_{L_x^\infty}
\end{equation*}
 For any \( K \geq 1 \), we have from a similar argument as in \eqref{measure_drift:eq_uN_2_p1} that 
\begin{align*}
\sup_{ \substack{0 \leq t^\prime \leq t \colon \\ t,t^\prime \sim K}} \frac{\| I_t[\uMP] - I_{t^\prime}[\uMP]\|_{L_x^\infty}}{1\wedge |t-t^\prime|^\alpha} 
&\lesssim \sup_{ \substack{0 \leq t^\prime \leq t \colon \\ t,t^\prime \sim K}} \frac{1}{1\wedge |t-t^\prime|^\alpha} \int_{t^\prime}^{t} A(s) \ds \\
&\lesssim \int_{s\sim K} A(s) \ds + \Big( \int_{s \sim K} A(s)^{\frac{1}{1-\alpha}} \ds \Big)^{1-\alpha}. 
\end{align*}
Proceeding as in the first estimate, this implies that 
\begin{equation*}
\Big( \E_{\QTN} \Big[ \Big(\sup_{ \substack{0 \leq t^\prime \leq t \colon \\ t,t^\prime \sim K}} \frac{\| I_t[\uMP] - I_{t^\prime}[\uMP]\|_{L_x^\infty}}{1\wedge |t-t^\prime|^\alpha}  \Big)^p \Big] \Big)^{\frac{1}{p}} \lesssim K^{-\frac{1}{2}-\gamma}. 
\end{equation*}
The desired estimate of the \( \cC_t^{\alpha,\eta} \cC_x^0\)-norm then follows by summing over dyadic scales and using a telescoping series if the times are not comparable. 
\end{proof}
In Lemma \ref{measure_drift:lemma_uN_1} and Lemma \ref{measure_drift:lemma_uN_2}, we controlled the process \( \uP \) with respect to the measures \( \QTN \). Unfortunately, the proof of Proposition \ref{measure_drift:proposition} below also requires the absence of finite-time blowup for \( \uP \) with respect \( \bP \). This is the subject of the next lemma. 

\begin{lemma}\label{measure_drift:lemma_no_explosion_P}
For any \( T \geq 1 \), it holds that \( T_{\exp}[\uP] = \infty ~ \bP\)-almost surely. 
\end{lemma}

The proof of the analogue for the \( \Phi^4_3\)-model (cf. \cite[Lemma 5]{BG20}) extends verbatim to our situation and we omit the minor modifications. To ease the reader's mind, let us briefly explain why the same argument applies here. In most of this section, the most important term in the integral equation \eqref{measure_drift:eq_integraleq} is the first summand. It has the lowest regularity and is closely tied to the interactions in the Hamiltonian. The result of Lemma \ref{measure_drift:lemma_no_explosion_P}, however, is essentially a soft statement. If we fix a time \( S \geq 1 \) and only want to rule out \( T_{\exp}[\uP] \leq S \), the low regularity is inessential and only leads to a loss in powers of \( S \). The main term is then given by the (auxiliary) second summand, which is defocusing and exactly the same as in the \( \Phi^4_3\)-model. 

The next proposition eliminates the stopping time from our drift measures.

\begin{proposition}\label{measure_drift:proposition}
The family of measures \( (\QTN)_{T,N\geq0} \) is tight on \( \cC_t^0 \cC_x^{-1/2-\kappa}([0,\infty]\times \T^3) \). For any fixed \( T \geq 0 \),   the sequence of measures \( (\QTN)_{N\geq 0} \) weakly converges to a measure \( \QT \) as \( N \rightarrow \infty\). For any \( S \geq 0 \), the limiting measure \( \QT \) satisfies 
\begin{equation}\label{measure_drift:eq_T_density}
\frac{\d \QT|_{\mathcal{F}_S}}{\d \bP|_{\mathcal{F}_S}} = \exp \Big( \int_0^S \int_{\T^3}\us  \d B_s - \frac{1}{2} \int_0^S \| \us \|_{L^2}^2 \ds \Big). 
\end{equation}
\end{proposition}
Our argument differs from the proof of \cite[Lemma 7]{BG20}, which is the analog for the \( \Phi^4_3 \)-model. The argument in \cite{BG20} relies on Kolmogorov's extension theorem, whereas we rely on tightness and Prokhorov's theorem. This is important in the proof of Corollary \ref{measure_drift:corollary_Qinf} below, since the measures \( \QT \) are not (completely) consistent. We also believe that this clarifies the mode of convergence. Before we begin with the proof, we state the following corollary.

\begin{corollary}\label{measure_drift:corollary_Qinf}
The measures \( \QT \) weakly convergence to a measure \( \Qinf \) on  \( \cC_t^0 \cC_x^{-1/2-\kappa}([0,\infty]\times \T^3) \) as \( T \rightarrow \infty \). For any \(S \geq 0 \), it holds that 
\begin{equation}\label{measure_drift:eq_inf_density}
\frac{\d \Qinf|_{\mathcal{F}_S}}{\d \bP|_{\mathcal{F}_S}} = \exp \Big( \int_0^S \int_{\T^3} u_s  \d B_s - \frac{1}{2} \int_0^S \| u_s \|_{L^2}^2 \ds \Big),
\end{equation}
where \( u_s \) is as in \eqref{measure_drift:eq_integraleq_u}. 
\end{corollary} 

\begin{proof}[Proof of Proposition \ref{measure_drift:proposition}:] 
We first prove that the family of measures \( (\QTN)_{T,N\geq 0} \), viewed as measures for \( W \), are tight on \( \cC_t^0 \cC_x^{-1/2-\kappa}([0,\infty]\times \T^3)  \). From \eqref{measure_drift:eq_WtuTN}, we have that 
\begin{equation}\label{measure_drift:eq_W_Wu}
W = \WuTN + I[\uNP].
\end{equation}
Since the law of \( \WuTN \) under \( \QTN \) agrees with the law of \( W \) under \( \bP \), an application of Kolmogorov's continuity theorem (cf. \cite[Theorem 4.3.2]{Stroock11}) yields for any \( p \geq 1 \), \( 0 < \alpha < \frac{1}{2} \), and \( 0 < \eta <\kappa/2 \) that 
\begin{equation*}
\E_{\QTN} \Big[ \| \WuTN \|_{\cC_t^{\alpha,\eta} \cC_x^{-(1+\kappa)/2}}^p\Big]  = \E_{\bP}  \Big[ \| W \|_{\cC_t^{\alpha,\eta} \cC_x^{-(1+\kappa)/2}}^p\Big] \lesssim_{p} 1 . 
\end{equation*}
Together with \eqref{measure_drift:eq_W_Wu} and Lemma \ref{measure_drift:lemma_uN_2}, this implies 
\begin{equation*}
\E_{\QTN} \Big[ \| W \|_{\cC_t^{\alpha,\eta} \cC_x^{-(1+\kappa)/2}}^p\Big] \lesssim_p 1 
\end{equation*}
Since the embedding \( \cC_t^{\alpha,\eta} \cC_x^{-(1+\kappa)/2} \hookrightarrow \cC_t^{0} \cC_x^{-1/2-\kappa} \) is compact, this implies the tightness of the family of measures \( (\QTN)_{T,N\geq 0} \). \\

By Prokhorov's theorem, a subsequence of \( (\QTN)_N \)  weakly converges to a measure \( \QT \).  Once we proved \eqref{measure_drift:eq_T_density}, this can be upgraded to weak convergence of the full sequence, since \eqref{measure_drift:eq_T_density} uniquely identifies the limit. With a slight abuse of notation, we therefore ignore this distinction between a subsequence and the full sequence. \\
Let \( S \geq 0 \) and let \( f\colon \cC_t^0 \cC_x^{-1/2-\kappa}([0,S]\times \T^3) \rightarrow \R \) be continuous, bounded, and nonnegative. We write \( f(W) \) for \( f(W|_{[0,S]}) \). Using the weak convergence of \( \QTN \) to \( \QT \), we have that 
\begin{equation*}
\E_{\QT}[f(W)] = \lim_{N\rightarrow \infty} \E_{\QTN}[f(W)] = \lim_{N\rightarrow \infty} \Big( \E_{\QTN}[ 1\{\tauTN \geq S\} f(W)] + \E_{\QTN}[ 1\{\tauTN< S\} f(W)] \Big). 
\end{equation*}
Using Lemma \ref{measure_drift:lemma_uN_2}, the second term is controlled by 
\begin{equation*}
\E_{\QTN}[ 1\{\tauTN< S\} f(W)] \leq \| f \|_{\infty} \QTN( \tauTN < S) \lesssim \| f\|_\infty \frac{\max(S^{1-2\gamma},1)}{N},
\end{equation*}
which converges to zero as \( N \rightarrow \infty \). Together with the definition of \( \QTN \) and the martingale property of the Girsanov density, this implies 
\begin{align*}
\E_{\QT}[f(W)] &=  \lim_{N\rightarrow \infty} \E_{\QTN}[ 1\{\tauTN \geq S\} f(W)]  \\
&= \lim_{N\rightarrow \infty} \E_{\bP} \Big[ f(W) 1\{ \tauTN \geq S\} \exp\Big( \int_0^{\tauTN} \us \d B_s - \frac{1}{2} \int_0^{\tauTN} \| \us\|_{L^2}^2 \ds \Big) \Big] \\
&= \lim_{N\rightarrow \infty} \E_{\bP} \Big[ f(W)1\{ \tauTN \geq S\} \exp\Big( \int_0^{S} \us \d B_s - \frac{1}{2} \int_0^{S} \| \us\|_{L^2}^2 \ds \Big) \Big]. 
\end{align*}
Using monotone convergence and Lemma \ref{measure_drift:lemma_no_explosion_P}, we obtain
\begin{align*}
 &\lim_{N\rightarrow \infty} \E_{\bP} \Big[ f(W)1\{ \tauTN \geq S\} \exp\Big( \int_0^{S} \us \d B_s - \frac{1}{2} \int_0^{S} \| \us\|_{L^2}^2 \ds \Big) \Big] \\
&=  \E_{\bP} \Big[ f(W)1\{ T_{\exp}[\uP]> S\} \exp\Big( \int_0^{S} \us \d B_s - \frac{1}{2} \int_0^{S} \| \us\|_{L^2}^2 \ds \Big) \Big]\\
&=  \E_{\bP} \Big[ f(W) \exp\Big( \int_0^{S} \us \d B_s - \frac{1}{2} \int_0^{S} \| \us\|_{L^2}^2 \ds \Big) \Big].
\end{align*}
\end{proof}

\begin{proof}[Proof of Corollary \ref{measure_drift:corollary_Qinf}:]
Due to Proposition \ref{measure_drift:proposition}, the family of measures \( (\QT)_{T\geq 0} \) is tight. By Prokhorov's theorem, it follows that a subsequence weakly converges to  a measure \( \Qinf \). Once \eqref{measure_drift:eq_inf_density} is proven, it uniquely identifies the limit \( \Qinf \). With a slight abuse of notation, we therefore assume as before that the whole sequence \( \QT \) converges weakly to \( \Qinf \). \\
Since \( \Wt = W_t \) and \( \It = I_t \) for all \( 0 \leq t \leq T/4\) (by our choice of $\rho$), it follows from the integral equation \eqref{measure_drift:eq_integraleq} that \( \us = u_s \) for all  \( 0 \leq s \leq T/4 \). Using \eqref{measure_drift:eq_T_density}, it follows for all \(S \leq T/4 \) that 
\begin{equation}
\frac{\d \QT|_{\mathcal{F}_S}}{\d \bP|_{\mathcal{F}_S}} = \exp \Big( \int_0^S \int_{\T^3}u_s  \d B_s - \frac{1}{2} \int_0^S \| u_s \|_{L^2}^2 \ds \Big). 
\end{equation}
The corresponding identity \eqref{measure_drift:eq_inf_density} for \( \Qinf \) then follows by taking \( T \rightarrow \infty \). 
\end{proof}

\begin{corollary}\label{measure_drift:corollary_bounds_u}
For any \( T \geq 1 \), \( S \geq 1 \), and any \( 0<\gamma< \min(\beta,1/2) \),  the measure \( \QT \) satisfies the two estimates
\begin{align*}
\E_{\QT} \Big[ \int_0^{S} \| \us \|_{L^2}^2 \ds \Big] &\lesssim \max( S^{1-2\gamma},1), \\
\sup_{t\geq 0} \Big( \E_{\QT} \big[ \| I_t[\uP] \|_{\cC_x^{\frac{1}{2}+\gamma}}^p \big] \Big)^{\frac{1}{p}} &\lesssim_p 1. 
\end{align*}
\end{corollary}
The corollary directly follows from Lemma \ref{measure_drift:lemma_uN_1}, Lemma \ref{measure_drift:lemma_uN_2}, and Proposition \ref{measure_drift:proposition}.

\subsection{\protect{Absolutely continuity with respect to the drift measure}}\label{section:absolute_continuity}

We recall the definition of the measure \( \tmuT \) from  \eqref{sc:eq_tmu}, which states that
 \begin{equation}
\frac{\d \tmu_T}{\d \bP} =  \frac{1}{\cZTl}  \exp\Big( - \lcol \cVT( \Winf) \rcol \Big).
\end{equation}
Using Proposition \ref{measure_drift:proposition}, we obtain that 
\begin{equation}\label{measure_drift:eq_DT}
D_T \defe \frac{\d \tmuT}{\d \QT} = \frac{\d \tmu_T}{\d \bP}  \frac{\d \bP}{\d \QT} =  \frac{1}{\cZTl} \exp\Big( - \lcol \cVT(\Winf)\rcol  - \int_0^\infty \int_{\T^3} \ut \, \d B_t + \frac{1}{2} \int_0^\infty \| \ut  \|_{L^2}^2 \dt \Big). 
\end{equation}
Since \( \d B_t = \d B_t^{\uP} + \ut  \d t \), we also obtain that 
\begin{equation}
D_T =  \frac{1}{\cZTl}  \exp\Big( - \lcol \cVT(\Winf)\rcol - \int_0^\infty \int_{\T^3} \ut \, \d B_t^{\uP} - \frac{1}{2} \int_0^\infty \| \ut  \|_{L^2}^2 \dt \Big).
\end{equation}

\begin{proposition}[$L^q$-bounds]\label{measure_drift:proposition_Lp}
If \( n\in \mathbb{N} \) in the definition of \( \uP\) is odd and sufficiently large, there exists a \( q > 1 \) such that 
\begin{equation}\label{measure_drift:eq_Lp}
\sup_{T\geq 0} \E_{\QT} \Big[ |D_T|^q  \Big] \lesssim_{n,q} 1. 
\end{equation}
\end{proposition}

\begin{remark}\label{measure_drift:remark_Lp}
We point out two important differences between Proposition \ref{measure_drift:proposition_Lp} and the corresponding result for the \( \Phi^4_3\)-model in \cite[Lemma 9]{BG20}. The first difference is a consequence of working with \( \tmuT \) instead of \( \bmu_T \) as described in Section \ref{section:stochastic_control}. Barashkov and Gubinelli define and bound the density \( D_T \) with respect to the same measure \( \Qinf \) for all \( T \geq 1 \). In contrast, our density is defined with respect to \( \QT \) and we make no statements about the behavior of \( D_T \) with respect to \( \Q^u_S \) for any \( S \neq T \). Since the increments of \( T \mapsto \rho_T(\nabla) W_\infty \) are not independent, such a statement would be especially difficult if \( S \) and \( T \) are close. The second difference is a result of the smoothing effect of the interaction potential \( V \). While the Hartree-nonlinearity allows us to prove the full \( L^q \)-bound \eqref{measure_drift:eq_Lp}, the corresponding result in the \( \Phi^4_3\)-model requires the localizing factor
\( \exp( - \| W_\infty\|_{\cC_x^{-1/2-\epsilon}}^n)\).
\end{remark}

The rest of this subsection is dedicated to the proof of the \( L^q\)-bounds (Proposition \ref{measure_drift:proposition_Lp}). Since we intend to apply the Boué-Dupuis formula to bound the density \( D_T \) in \( L^q(\QT) \), we first study the effect of shifts in \( B^{\uP} \) on the integral equation \eqref{measure_drift:eq_integraleq}.  For any \( w \in \bH \), we define
\begin{align*}
u^{\scriptscriptstyle{T},w}_s&\defe \Xi(\WPTuT+ w)_s \\
&= - \lambda \lcol (V* (\WsTuT + \Is[w])^2) (\WsTuT  + \Is[w])\rcol + \Js \langle \nabla \rangle^{-\frac{1}{2}} \lcol ( \langle \nabla \rangle^{-\frac{1}{2}} (\WsTuT+\Is[w]))^n \rcol. 
\end{align*}

Using the cubic binomial formula (Lemma \ref{measure_re:lemma_binomial}), we obtain that 
\begin{equation}\label{measure_drift:eq_integral_u_r}
u^{\scriptscriptstyle{T},w}_s = - \lambda \Js \lcol (V*(\WsTuT)^2) \WsTuT  \rcol + r_s^{\scriptscriptstyle{T},w}, 
\end{equation}
where the remainder \( r_s^{\scriptscriptstyle{T},w}\) is given by 
\begin{align*}
r_s^{\scriptscriptstyle{T},w}& = - \lambda \Js \Big( (V* \lcol(\WsTuT)^2\rcol) \Is[w]\Big)- 2 \lambda \Js \Big( (V*(\WsTuT \Is[w])) \WsTuT - \Ms \Is[w] \Big)\\
&- 2 \lambda \Js \Big( (V* (\WsTuT \Is[w])) \Is[w] \Big) - \lambda \Js \Big( (V*\Is[w]^2) \tW_s \Big) - \lambda \Js \Big( (V*\Is[w]^2) \Is[w] \Big)\\
& + \Js \langle \nabla \rangle^{-\frac{1}{2}} \lcol ( \langle \nabla \rangle^{-\frac{1}{2}} (\WsTuT+\Is[w]))^n \rcol . 
\end{align*} 
We also define \( \hTw = w + \uTw \). 
We further decompose 
\begin{equation*}
r_s^{\scriptscriptstyle{T},w} = \widetilde{r}_s^{\scriptscriptstyle{T},w} + \Js \langle \nabla \rangle^{-\frac{1}{2}} \lcol ( \langle \nabla \rangle^{-\frac{1}{2}} (\WsTuT+\Is[w]))^n\rcol. 
\end{equation*}
 Before we begin the main argument, we prove the following auxiliary lemma.

\begin{lemma}[Estimate of $\widetilde{r}_t^{\scriptscriptstyle{T},w}$]\label{measure_drift:lemma_tr}
Let \( \epsilon,\delta>0 \) be small absolute constants and let \( n\geq n(\delta,\beta) \) be sufficiently large. Then, we have for all \( t \geq 0 \) that 
\begin{equation}
\langle t \rangle^{1+\delta} \| \widetilde{r}^{\scriptscriptstyle{T},w}_t \|_{L^2_x}^2 \lesssim_{n,\delta,\beta,\lambda} C_\epsilon Q_t(\tW) + \epsilon \Big( \| \It[w] \|_{\bW^{-\frac{1}{2},n+1}_x}^{n+1} + \int_0^t \| w_s \|_{L^2_x}^2 \ds \Big). 
\end{equation}
\end{lemma}
\begin{remark}
We emphasize that the implicit constant does not depend on \( \epsilon \). In the application of Lemma \ref{measure_drift:lemma_tr}, we will choose \( \epsilon >0 \) sufficiently small depending on \( \delta,n,\beta,\lambda\).
\end{remark}
\begin{proof} In the following argument, the implicit constants are allowed to depend on \( n,\delta,\beta, \) and \( \lambda \) but not on \( \epsilon \). 
We estimate the five terms in \( \widetilde{r}_t^{\scriptscriptstyle{T},w} \) separately and do not require any new ingredients. We only rely on Lemma \ref{measure_re:lemma_stochastic_objects_I}, Proposition \ref{measure_rmt:proposition_estimate}, Hölder's inequality, and Bernstein's inequality. \\
For the first term, we have from the definition of \( \Jt \) and Lemma \ref{measure_re:lemma_stochastic_objects_I} that 
\begin{align*}
\Big\| \Jt  \Big( (V* \lcol (\WtTuT)^2\rcol) \It[w] \Big) \Big\|_{L^2_x}^2 
&\lesssim \langle t \rangle^{-1-2\delta}  \Big\|  \Big( (V* \lcol(\WtTuT)^2\rcol) \It[w] \Big) \Big\|_{H^{-1+\delta}_x}^2 \\
&\lesssim \langle t \rangle^{-1-2\delta} \Big\| V *  \lcol(\WtTuT)^2\rcol \Big\|_{\cC_x^{-1+2\delta}}^2 \| \It[w]  \|_{H^{1-\delta}_x}^2  \\
&\lesssim  \langle t \rangle^{-1-2\delta} \Big\| V *  \lcol(\WtTuT)^2\rcol \Big\|_{\cC_x^{-1+2\delta}}^2 \| \It[w]  \|_{\bW^{-\frac{1}{2},n+1}_x}^{2\delta} \| \It[w] \|_{H^1_x}^{2-\delta} \\
&\lesssim  \langle t \rangle^{-1-2\delta} C_\epsilon \, Q_t(\tW) + \langle t \rangle^{-1-2\delta} \epsilon \Big( \| \It[w]  \|_{\bW^{-\frac{1}{2},n+1}_x}^{n+1}  + \| \It[w]\|_{H^1_x}^2 \Big). 
\end{align*}
For the second term, we have from duality and Proposition \ref{measure_rmt:proposition_estimate} for all \( 0 < \gamma < \min(\beta,1) \) that 
\begin{align*}
&\big \| J_t \Big( (V*(\WtTuT\It[w] )) \WtTuT\ - \MT \It[w] \Big) \big\|_{L^2_x}^2 \\
&\leq \langle t \rangle^{-1-2\gamma} \| \Jt  \Big( (V*(\WtTuT\ \It[w] )) \WtTuT\ - \MT \It[w] \Big) \big\|_{H^{-(1-\gamma)}_x}^2 \\
&\lesssim \langle t \rangle^{-1-2\gamma}  Q_t(\tW) \| \It[w] \|_{H^{1-\gamma}_x}^2 \\
&\lesssim  \langle t \rangle^{-1-2\gamma} C_\epsilon \, Q_t(\tW) +  \langle t \rangle^{-1-2\gamma} \epsilon \Big(  \| \It[w]  \|_{\bW^{-\frac{1}{2},n+1}_x}^{n+1}  + \|\It[w] \|_{H^1_x}^2 \Big). 
\end{align*}
For the third term, we estimate
\begin{align*}
\| \Jt  \Big( (V* (\tW_t \It[w] )) \It[w]  \Big) \|_{L^2_x}^2 &\lesssim \langle t \rangle^{-3} \| V * (\tW_t \It[w] )\|_{L^4_x}^2 \| \It[w]  \|_{L^4_x}^2 \\
&\lesssim \langle t \rangle^{-3} \| \tW_t \|_{L^\infty_x}^2 \|\It[w] \|_{L^4_x}^4 \\
&\lesssim \langle t \rangle^{-3+1+2\delta} \| \tW_t \|_{\cC_x^{-\frac{1}{2}-\delta}}^2 \| \It[w]  \|_{L^4_x}^4 \\ 
&\lesssim  C_\epsilon  \langle t \rangle^{-2+2\delta} \, \| \tW_t \|_{\cC_x^{-\frac{1}{2}-\delta}}^{2 \frac{4+\delta}{\delta}}  +  \epsilon \langle t \rangle^{-2+2\delta}  \| \It[w]  \|_{L^4_x}^{4+\delta}  \\ 
&\lesssim C_\epsilon  \langle t \rangle^{-2+2\delta}  Q_t(\tW) + \epsilon \langle t \rangle^{-2+2\delta}\Big(  \| \It[w]  \|_{\bW^{-\frac{1}{2},n+1}_x}^{n+1}  + \|\It[w] ]\|_{H^1_x}^2 \Big). 
\end{align*}
In the last line, we use \cite[Lemma 20]{BG20}. \\
The fourth term can be estimated exactly like the third term. To estimate the fifth term, we only rely on Hölder's  inequality, Bernstein's inequality, and the Fourier support condition of \( \It[w]  \). We have that
\begin{align*}
&\| \Jt  \Big( (V* \It[w] ^2) \It[w] \Big) \|_{L_x^2}^2 
\lesssim \langle t \rangle^{-3} \| (V* \It[w] ^2) \It[w]   \|_{L_x^2}^2 
\lesssim \langle t \rangle^{-3} \| \It[w]  \|_{L_x^6}^6 
\lesssim \langle t \rangle^{-3+\frac{\delta}{2}} \| \It[w]  \|_{L_x^{\frac{6}{6+\delta}}}^6 \\
&\lesssim \langle t\rangle^{-3+\frac{\delta}{2}} \| \It[w]  \|_{\bW_x^{-\frac{1}{2},\frac{4}{\delta}}}^{4} \| \It[w]  \|_{H_x^1}^2 
\lesssim \langle t \rangle^{-3+2\delta}   \| \It[w]  \|_{\bW_x^{-\frac{1}{2},\frac{4}{\delta}}}^{4+\delta} \| \It[w]  \|_{H_x^1}^{2-\delta} \\
&\lesssim C_\epsilon \langle t \rangle^{-3+2\delta}  + \epsilon \langle t \rangle^{-3+2\delta}   \Big( \| \It[w]  \|_{\bW_x^{-\frac{1}{2},n+1}}^{n+1} + \| \It[w] \|_{H_x^1}^{2}\Big). 
\end{align*}
In the second last inequality, we used that \( \| \It[w] \|_{H^1_x} \lesssim \langle t \rangle^{\frac{3}{2}} \| \It[w] \|_{H_x^{-1/2}} \). This completes the estimate of all five terms in \( \widetilde{r}^{\scriptscriptstyle{T},w}_t \) and hence the proof. 
\end{proof}

Equipped with Lemma \ref{measure_drift:lemma_tr}, we can now prove the $L^q$-bound for $D_T$. 

\begin{proof}[Proof of Proposition \ref{measure_drift:proposition_Lp}:]
The proof splits into two steps. \\
\emph{Step 1: Formulation as a variational problem.}
In order to prove the desired estimate \eqref{measure_drift:eq_Lp}, it suffices to obtain a lower bound on
\( -\log \E_{\QT} [ D_T^q ] \). Using the Boué-Dupuis formula, we obtain
\begin{align*}
&-\log \E_{\QT}[ D_T^q]  -q \log(\cZTl) \\
&= -\log \E_{\QT} \bigg[ \exp\bigg( -q \Big( \lcol \cVT(\WinfTuT+\Iinf[u]) \rcol - \int_0^\infty \int_{\T^3} \ut  \, \d B^{\uP}_t - \frac{1}{2} \int_0^\infty \| \ut \|_{L^2}^2 \dt \Big) \bigg) \bigg] \\
&= \inf_{w \in \bH} \E \bigg[ q \bigg(\lcol \cVT(\WinfTuT+ \Iinf[w]+ \Iinf[\uTw])\rcol + \int_0^\infty \int_{\T^3} u_t^{\scriptscriptstyle{T},w} \d B^{\uP}_t + \int_0^\infty \int_{\T^3} u_t^{\scriptscriptstyle{T},w} w_t \dx \dt \\
 &\hspace{10ex}+\frac{1}{2} \int_0^\infty \| u_t^{\scriptscriptstyle{T},w} \|_{L^2}^2 \dt \bigg) 
 + \frac{1}{2} \int_0^\infty \| w_t \|_{L^2}^2 \dt \bigg]. 
\end{align*}
Since \( T\mapsto \int_0^T \int_{\T^3} u_t^{\scriptscriptstyle{T},w}\d B^{\uP}_t \) is a martingale, its expectation vanishes. We now insert the change of variables \( \uTw = \hTw -w  \) into the formula above, and obtain that
\begin{align*}
&-\log \E_{\QT}[ D_T^q]    -q \log(\cZTl) \\
&=\inf_{w \in \cH_a} \E_{\QT}\bigg[ q \Big( \lcol \cVT(\WinfTuT+ \Iinf[\hTw ])\rcol + \frac{1}{2} \int_0^\infty \| h_t^w \|_{L^2}^2 \dt - \frac{1}{2} \int_0^\infty \| w_t \|_{L^2}^2 \dt \Big) 
+ \frac{1}{2} \int_0^\infty \| w_t \|_{L^2}^2 \dt \bigg] \\
&=\inf_{w \in \bH} \E_{\QT} \bigg[ q \Big( \lcol \cVT(\WinfTuT+ \Iinf[\hTw ])\rcol + \frac{1}{2} \int_0^\infty  \| h_t^w \|_{L^2}^2 \dt \Big) - \frac{q-1}{2} \int_0^\infty \| w_t\|_{L^2}^2 \dt  \bigg]. 
\end{align*}
Since we want to obtain a lower bound, the most dangerous term in the expression above is  \( -\frac{q-1}{2} \int_0^\infty \| w_t\|_{L^2}^2 \dt \). Using our previous information about the variational problem (Proposition \ref{measure_var:proposition_main} and Proposition \ref{measure_var:proposition}) and the nonnegativity of \( \cV(\Iinf[\hTw ]) \), we obtain that 
\begin{equation}\label{measure_drift:eq_lp_p1}
-\log \E_{\QT}[ D_T^q] 
\geq -C  + \inf_{w\in \bH} \E_{\QT} \bigg[ \frac{1}{4} \int_0^\infty \| l^T_t(\hTw )\|_{L^2}^2 \dt  - \frac{q-1}{2} \int_0^\infty \| w_t\|_{L^2}^2 \dt \bigg].
\end{equation}
Recalling the definition of \( l_t^T(\hTw ) \) from Proposition \ref{measure_var:proposition_main} and \eqref{measure_drift:eq_integral_u_r}, we obtain that 
\begin{align*}
 l_t^T(\hTw ) &=  \hTw_t + \lambda  \Jt  \lcol (V*(\WtTuT)^2) \tW_t \rcol \\
&=  (u_t^{\scriptscriptstyle{T},w} + w_t) + \Jt  \lcol (V*(\WtTuT)^2) \tW_t \rcol \\
&=  (\rTw_t+ w_t). 
\end{align*}
Together with our previous estimate, this leads to 
\begin{equation*}
-\log \E_{\QT}[ D_T^q] \geq -C + \inf_{w\in \bH} \E_{\QT} \bigg[  \frac{1}{4} \int_0^\infty \| w_t + \rTw_t\|_{L^2}^2 \dt - \frac{q-1}{2} \int_0^\infty \| w_t \|_{L^2}^2 \dt \bigg]. 
\end{equation*}
By choosing \( q \) sufficiently close to one, it only remains to establish 
\begin{equation}
\E \int_0^\infty \| w_t \|_{L^2}^2 \dt \lesssim 1+ \E \int_0^\infty \| w_t + \rTw_t \|_{L^2}^2 \dt. 
\end{equation}
This bound is proven via a Gronwall-type argument. \\

\emph{Step 2: Gronwall-type argument.} 
This step crucially relies on the smoother term in the definition of the drift \eqref{measure_drift:eq_integraleq}. We essentially follow the proof of \cite[Lemma 11]{BG20}. As in \cite{BG20}, we introduce the auxiliary process
\begin{equation}
\Aux_s(\tW,w) = \sum_{i=0}^n {n \choose i} \langle \nabla \rangle^{-\frac{1}{2}} \Js  \Big( \lcol ( \langle \nabla \rangle^{-\frac{1}{2}} \tW_s)^i\rcol (\langle \nabla \rangle^{-\frac{1}{2}} \Is[w] )^{n-i}  \Big). 
\end{equation}
With this notation, it holds that \( \rTw = \widetilde{r}^{\scriptscriptstyle{T},w} + \Aux(\tW,w) \). We then expand 
\begin{equation}\label{measure_drift:eq_identity_w_r}
\begin{aligned}
w_s^2 &= 2 (w_s + \rTw_s)^2 - 4 w_s r_s^{\scriptscriptstyle{T},w} - 2 (r_s^{\scriptscriptstyle{T},w})^2 - w_s^2 \\
&= 2 (w_s+r_s^{\scriptscriptstyle{T},w})^2 -4 w_s \widetilde{r}^{\scriptscriptstyle{T},w}_s - 2 (r_s^w)^2 - w_s^2 - 4 \Aux_s(\tW,w). 
\end{aligned}
\end{equation}
Using Itô's integration by parts formula, we have for all \( s \leq t \) that 
\begin{align*}
&4 \int_0^t \int_{\T^3} \Aux_s(\tW,w) w_s \dx \ds \\
&= 4 \sum_{i=0}^n { n \choose i} \int_0^t \int_{\T^3} \lcol ( \langle \nabla \rangle^{-\frac{1}{2}} \tW_s)^i \rcol \, (\langle \nabla \rangle^{-\frac{1}{2}} \Is[w] )^{n-i} \,  (\langle \nabla \rangle^{-\frac{1}{2}} \Js  w_s) \dx \ds \\
&=4 \sum_{i=0}^n  \frac{1}{n+1-i} { n \choose i} \int_0^t \int_{\T^3} \lcol ( \langle \nabla \rangle^{-\frac{1}{2}} \tW_s)^i\rcol \, \frac{\partial}{\partial s}(\langle \nabla \rangle^{-\frac{1}{2}} \Is[w] )^{n+1-i}\dx \ds \\
&= 4 \sum_{i=0}^n \frac{1}{n+1-i} {n \choose i} \int_{\T^3} \lcol (\langle \nabla \rangle^{-\frac{1}{2}} \tW_t)^{i}\rcol (\langle \nabla \rangle^{-\frac{1}{2}} \It[w])^{n+1-i} \dx \\
&- 4 \sum_{i=0}^n \frac{1}{n+1-i} {n \choose i} \int_0^t \int_{\T^3}  (\langle \nabla \rangle^{-\frac{1}{2}} \Is[w] )^{n+1-i} \d \big( \lcol (\langle \nabla \rangle^{-\frac{1}{2}} \tW_s)^i \rcol \big).
\end{align*}
Due to the martingale property, the second summand has zero expectation. After setting 
\begin{equation}
\overline{\Aux}_t(\tW,w) \defe \sum_{i=0}^n \frac{1}{n+1-i} {n \choose i} \int_{\T^3} \lcol (\langle \nabla \rangle^{-\frac{1}{2}} \tW_t)^{i}\rcol (\langle \nabla \rangle^{-\frac{1}{2}} \It[w] )^{n+1-i} \dx,
\end{equation}
we obtain from \eqref{measure_drift:eq_identity_w_r} that 
\begin{equation}\label{measure_drift:eq_gronwall_1}
\begin{aligned}
&\E \Big[ \int_0^t \| w_s \|_{L^2}^2 \ds + 4 \overline{\Aux}_t(\tW,w) \Big] \\
&= \E \Big[ \int_0^t \big( 2 \| w_s +\rTw_s \|_{L^2}^2 - 4 \langle w_s, \widetilde{r}_s^{\scriptscriptstyle{T},w}\rangle - \| w_s \|_{L^2}^2 - 2 \|r_s^{\scriptscriptstyle{T},w} \|_{L^2}^2 \big) \ds \Big] \\
&\leq \E \Big[ 2 \int_0^t  \| w_s +\rTw_s \|_{L^2}^2 \ds + 4 \int_0^T \| \widetilde{r}_s^{\scriptscriptstyle{T},w}\|_{L^2}^2 \ds \Big]. 
\end{aligned}
\end{equation}
We perform the Gronwall-type argument based on the quantity \( \Phi(t) \), which is defined by
\begin{equation}
\Phi(t)\defe \E \int_0^t \| w_s \|_{L^2}^2 \ds + \| \It[w] \|_{\bW^{-\frac{1}{2},n+1}_x}^{n+1}. 
\end{equation}
By \cite[Lemma 12]{BG20} and \eqref{measure_drift:eq_gronwall_1}, we have that 
\begin{align*}
\Phi(t) \lesssim 1 + \E \bigg[ \int_0^t \| w_s \|_{L^2}^2 \ds + \overline{\Aux}_t(\tW,w) \bigg] \lesssim 1+ \E \bigg[ \int_0^t \| r_s^{\scriptscriptstyle{T},w} + w_s \|_{L^2}^2 \ds + \int_0^t \| \widetilde{r}_s^{\scriptscriptstyle{T},w} \|_{L^2}^2 \ds \bigg]. 
\end{align*}
From Lemma \ref{measure_drift:lemma_tr}, we obtain for \( \epsilon, \delta >0 \) that 
\begin{align*}
\Phi(t) &\lesssim_\delta 1+ \E \bigg[\int_0^t \|r_s^{\scriptscriptstyle{T},w} + w_s \|_{L^2}^2 \ds + C_{\epsilon}  \int_0^t \langle s \rangle^{-1-\delta} Q_s(\bW,\lambda)  \ds \bigg] + \epsilon \int_0^t \langle s \rangle^{-1-\delta} \Phi(s) \ds \\
&\lesssim_\delta C_{\epsilon} + \E \bigg[\int_0^t \|r_s^{\scriptscriptstyle{T},w} + w_s \|_{L^2}^2 \ds\bigg] + \epsilon \sup_{0\leq s\leq t} \Phi(s). 
\end{align*}
By choosing \( \epsilon >0 \) sufficiently small depending on \( \delta\), this implies the desired estimate. 
\end{proof}

\subsection{The reference measure}

Using our construction of the drift measures \( \QT \), we now provide a short proof of Theorem \ref{theorem:reference}. As in the rest of this section, we use the truncation parameter $T$. 

\begin{proof}[Proof of Theorem \ref{theorem:reference}:]
For any $1\leq T\leq \infty$, we define the reference measure $\nu_T$ as 
\begin{equation*}
\nu_T \defe (W_\infty)_\# \QT. 
\end{equation*}
By using the $L^q$-bound (Proposition \ref{measure_drift:proposition_Lp}), we have that for all Borel sets $A\subseteq \cC_x^{-1/2-\kappa}(\T^3)$ that 
\begin{equation*}
\mu_T(A) = \tmu_T(W_\infty \in A) = \E_{\QT} \big[ 1\big\{ W_\infty \in A \big\} \, \DT \big] \leq \Big( \E_{\QT} \big[  \DT^q \big] \Big)^{\frac{1}{q}} \QT(W_\infty \in A)^{1-\frac{1}{q}} \lesssim \nu_T(A)^{1-\frac{1}{q}}. 
\end{equation*}
This proves the first part of Theorem  \ref{theorem:reference}. Regarding the representation of $\nu_T$, which forms the second part of Theorem  \ref{theorem:reference}, we have that 
\begin{align*}
&\nu_T \\
=& \Law_{\QT}(W_\infty) \\
	=& \Law_{\QT}( W^u_\infty + I_\infty[u^{\scriptscriptstyle{T}}])  \\
	=& \Law_{\QT} \Big( W_\infty^u - \lambda  \rho_T(\nabla) \int_0^\infty \hspace{-1ex} J_s^2 \lcol ( V* (W_s^{\scriptscriptstyle{T},u})^2 ) W_s^{\scriptscriptstyle{T},u} \rcol  \hspace{-0.5ex} \ds + \rho_T(\nabla) \int_0^\infty  \hspace{-1ex} \langle \nabla \rangle^{-\frac{1}{2}}  J_s^2
 \lcol \big( \langle \nabla \rangle^{-\frac{1}{2}} W_s^{\scriptscriptstyle{T},u}  \big)^n \rcol  \hspace{-0.5ex} \ds \Big) \\
=& \Law_{\bP} \Big( W_\infty - \lambda \rho_T(\nabla) \int_0^\infty  J_s^2 \lcol ( V* (W_s^{\scriptscriptstyle{T}})^2 ) W_s^{\scriptscriptstyle{T}} \rcol \ds + \rho_T(\nabla) \int_0^\infty \langle \nabla \rangle^{-\frac{1}{2}}  J_s^2
 \lcol \big( \langle \nabla \rangle^{-\frac{1}{2}} W_s^{\scriptscriptstyle{T}}  \big)^n \rcol \ds  \Big). 
\end{align*}
This completes the proof. 
\end{proof}

\section{Singularity}\label{section:singularity}

In this section, we prove Theorem  \ref{theorem:singularity}. The majority of this section deals with the singularity for \( 0 < \beta < 1/2 \). The absolute continuity for \( \beta > 1/2 \) will be deduced from Corollary \ref{measure_var:corollary_cT} and requires no new ingredients. Theorem \ref{theorem:singularity} is important for the motivation of this series of papers, since we provide the first proof of invariance for a Gibbs measure which is singular with respect to the corresponding Gaussian free field. The methods of this section, however, will not be used in the rest of this two-paper series. 

We prove the singularity of the Gibbs measure \( \mu_\infty \) and the Gaussian free field \( \cg \) through the explicit event in Proposition  \ref{singularity:proposition_main}.

\begin{proposition}[Singularity]\label{singularity:proposition_main}
Let \( 0 < \beta < \frac{1}{2} \) and let \( \delta >0 \) be sufficiently small. Then, there exists a (deterministic) sequence \( (S_m)_{m=1}^\infty \subseteq \mathbb{R}_{>0} \) converging to infinity such that 
\begin{equation}\label{singularity:eq_GFF}
\lim_{m\rightarrow \infty} \frac{1}{S_m^{1-2\beta-\delta}} \int_{\T^3} \lcol (V* (\rho_{S_m}(\nabla) \phi)^2) (\rho_{S_m}(\nabla) \phi)^2 \rcol \dx = 0 \quad \cg \text{-a.s.}
\end{equation}
and 
\begin{equation}\label{singularity:eq_Gibbs}
\lim_{m\rightarrow \infty} \frac{1}{S_m^{1-2\beta-\delta}} \int_{\T^3} \lcol (V* (\rho_{S_m}(\nabla) \phi)^2) (\rho_{S_m}(\nabla) \phi)^2 \rcol \dx =  -\infty \quad \mu_\infty\text{-a.s.}
\end{equation}
Here, \( \cg \) is the Gaussian free field, \( \mu_\infty\) is the Gibbs measure, and \( \phi \in \cC_x^{-\frac{1}{2}-\kappa}(\T^3) \) denotes the random element.  
\end{proposition}

\begin{remark}
In the statement of the proposition, the reader may wish to replace \( \phi \) by \( W_\infty \), \( \cg \) by \( \bP \), and \( \mu_\infty \) by \( \tmu_\infty \). We choose the notation \( \phi \) to emphasize that this is a property of \( \cg \) and \( \mu_\infty \) only and does not rely on the stochastic control perspective. Of course, the stochastic control perspective is heavily used in the proof. 
\end{remark}

To simplify the notation, we define 
\begin{equation}
\bWS{3}{s}  \defe \lcol (V * (\WS{}{s})^2) \WS{}{s} \rcol \qquad \text{and} \qquad 
\bWS{4}{s} \defe \lcol ( V* ( \WS{}{s})^2) (\WS{}{s})^2 \rcol. 
\end{equation}
We note that the dependence on the interaction potential \( V \) is not reflected in this notation. We first study the behavior of the integral of \( \WS{4}{\infty} \) with respect to \( \bP \). This is the easier part of the proof and the statement \eqref{singularity:eq_GFF} follows from the following lemma.

\begin{lemma}[Quartic power under the Gaussian free field]\label{singularity:lemma_GFF}
Let \( 0 < \beta < 1/2 \). Then, we have that 
\begin{equation}
\sup_{S\geq 1} \E_{\bP} \bigg[ \Big( \frac{1}{S^{\frac{1}{2}-\beta}} \int_{\T^3} \bWS{4}{\infty} \dx \Big)^2 \bigg] \lesssim 1. 
\end{equation}
\end{lemma}

\begin{proof}
From Proposition \ref{measure_re:proposition_renormalization}, we obtain that 
\begin{equation*}
\int_{\T^3} \bWS{4}{\infty} \dx = \sum_{\substack{n_1,n_2,n_3,n_4\in\Z^3 \colon\\ n_{1234}=0 } }  \Big( \sum_{\pi\in S_4} \widehat{V}(n_{\pi(1)}+n_{\pi(2)}) \Big) \int_0^\infty \int_0^{s_1} \int_0^{s_2} \int_0^{s_3} \d \WS{n_4}{s_4} \d \WS{n_3}{s_3}\d \WS{n_2}{s_2}\d \WS{n_1}{s_1}. 
\end{equation*}
Since the iterated stochastic integrals are uncorrelated, we obtain that 
\begin{align*}
&\E_{\bP} \bigg[ \Big( \int_{\T^3} \bWS{4}{\infty} \dx \Big)^2 \bigg] \\
&\lesssim \sum_{\substack{n_1,n_2,n_3,n_4\in\Z^3 \colon\\ n_{1234}=0 } }   \Big( \sum_{\pi\in S_4} \widehat{V}(n_{\pi(1)}+n_{\pi(2)}) \Big)^2 \prod_{j=1}^4 \frac{\rhoS{s}(n_j)^2}{\langle n_j \rangle^2} \\
&\lesssim \sum_{\substack{n_1,n_2,n_3,n_4\in\Z^3 \colon\\ n_{1234}=0 } }   \langle n_{12} \rangle^{-2\beta} \prod_{j=1}^4 \frac{\rhoS{s}(n_j)^2}{\langle n_j \rangle^2} \\
&\lesssim \sum_{n_1,n_2,n_3 \in \Z^3 } \langle n_{123} \rangle^{-2} \langle n_{12} \rangle^{-2\beta} \prod_{j=1}^3  \frac{\rhoS{s}(n_j)^2}{\langle n_j \rangle^2}.
\end{align*}
It now only remains to estimate the sum. Provided that \( \beta < 1/2 \), we first sum in \( n_3 \), then \( n_2 \), and finally \( n_1 \) to obtain 
\begin{align*}
\sum_{n_1,n_2,n_3 \in \Z^3 } \hspace{-2ex} \langle n_{123} \rangle^{-2} \langle n_{12} \rangle^{-2\beta} \prod_{j=1}^3  \frac{\rhoS{s}(n_j)^2}{\langle n_j \rangle^2} 
\lesssim \sum_{n_1,n_2\in \Z^3 }\hspace{-1ex}   \langle n_{12} \rangle^{-1-2\beta} \prod_{j=1}^2  \frac{\rhoS{s}(n_j)^2}{\langle n_j \rangle^2} 
\lesssim \sum_{n_1 \in \Z^3} \frac{\rhoS{s}(n_1)^2}{\langle n_1 \rangle^{2+2\beta}} 
\lesssim S^{1-2\beta}. 
\end{align*}
\end{proof}
We now begin our study of the integral \(  \int_{\T^3} \bWS{4}{\infty} \dx  \) under \( \Qinf \). Naturally, we would like to replace (most) occurrences of \( \WS{}{} \) by \( \WSu{}{} \), since the law of \( \WSu{}{} \) under \( \Qinf  \) is explicit. This is the objective of our first (algebraic) lemma. 

\begin{lemma}\label{singularity:lemma_representation}
For any \( S \geq 1 \), it holds that 
\begin{align}
 \int_{\T^3} \bWS{4}{\infty} \dx 
 =& - 4 \lambda \int_0^\infty \int_{\T^3} ( \JS_s \bWSu{3}{s} )  \cdot J_s \bWu{3}{s} \dx \ds \label{singularity:eq_representation_1}\\
&+ 4 \int_0^\infty \int_{\T^3} (\JS_s \bWSu{3}{s} ) \d B^u_s 
- 4 \lambda \sum_{j=1}^3 \int_0^\infty \int_{\T^3} \AS{j}{s}[u] \cdot J_s \bWu{3}{s} \dx \ds  \label{singularity:eq_representation_2}\\
& +  4 \sum_{j=1}^3 \int_0^\infty \int_{\T^3} \AS{j}{s}[u]  \d B^u_s 
+ 4 \int_0^\infty \int_{\T^3} \bWS{3}{s} \Big( J_s \langle \nabla \rangle^{-\frac{1}{2}} \lcol \big( \langle \nabla \rangle^{-\frac{1}{2}} \Wu{}{s} \big)^n \rcol \Big) \dx \ds, \label{singularity:eq_representation_3}
\end{align}
where 
\begin{align}
\AS{1}{s}[u] &\defe \JS_s \Big( ( V * \lcol (\WSu{}{s})^2 \rcol ) \IS_s[u] \Big) + 2 \JS_s \Big( V* ( \WSu{}{s} \IS_s[u]) \cdot \WSu{}{s} -   \MS_s \IS_s[u]  \Big), \label{singularity:eq_A1} \\
\AS{2}{s}[u] &= \JS_s \Big( (V* (\IS_s[u])^2 ) \WSu{}{s} \Big) + 2 \JS_s \Big( (V*( \WSu{}{s} \IS_s[u] )) \IS_s[u] \Big) , \label{singularity:eq_A2}\\
\AS{3}{s}[u] &= \JS_s \Big( (V* \IS_s[u]^2 ) \IS_s[u] \Big). \label{singularity:eq_A3}
\end{align}
\end{lemma}

\begin{proof}
Using \eqref{measure_re:eq_potential_derivative} from Proposition \ref{measure_re:proposition_renormalization} together with the integral equation for \( u \), i.e. \eqref{measure_drift:eq_integraleq_u}, we obtain that 
\begin{align}
 \int_{\T^3} \bWS{4}{\infty} \dx   
=& 4 \int_0^\infty \int_{\T^3} \bWS{3}{s} \d \WS{}{s}  \notag \\
=& 4 \int_0^\infty \int_{\T^3} \bWS{3}{s} (\JS_s u_s) \dx \ds + 4 \int_0^\infty \int_{\T^3} \bWS{3}{s} \d \WSu{}{s}  \notag \\
=& - 4 \lambda \int_0^\infty \int_{\T^3} ( \JS_s \bWS{3}{s}) ( J_s \bWu{3}{s} ) \dx \ds  + 4 \int_0^\infty \int_{\T^3} ( \JS_s \bWS{3}{s} ) \d B^u_s   \label{singularity:eq_representation_p1}\\
+& 4 \int_0^\infty \int_{\T^3} ( \JS_s \bWS{3}{s} ) ( J_s \langle \nabla\rangle^{-\frac{1}{2}} \lcol ( \langle \nabla \rangle^{-\frac{1}{2}} W^u_s )^n \rcol \dx \ds  \notag
\end{align}
From the cubic binomial formula \eqref{measure_re:eq_cubic_binomial} and the definition of \( \AS{j}{s}\), it follows that 
\begin{equation*}
\JS_s \bWS{3}{s} = \JS_s \bWSu{3}{s} + \sum_{j=1}^3 \AS{j}{s}[u]. 
\end{equation*}
Inserting this into \eqref{singularity:eq_representation_p1} leads to the desired identity. 
\end{proof}

We begin by studying the right-hand side of \eqref{singularity:eq_representation_1}, which is the main term. Our first lemma controls the expectation, which  will  be upgraded to a pointwise estimate later.
\begin{lemma}\label{singularity:lemma_expectation} If \( 0 < \beta < 1/2 \) and \( S \geq 1 \) is sufficiently large, then 
\begin{equation}
\E_{\Qinf} \Big[ \int_0^\infty \int_{\T^3} ( \JS_s \bWSu{3}{s} ) \cdot J_s \bWu{3}{s} \dx \ds  \Big] \gtrsim S^{1-2 \beta}. 
\end{equation}
\end{lemma}

\begin{proof}
Since the law of \( \Wu{}{} \) under \( \Qinf\) coincides with the law of \( W \) under \( \bP \), 
it holds that 
\begin{equation}
\E_{\Qinf} \Big[ \int_0^\infty \int_{\T^3} ( \JS_s \bWSu{3}{s} ) \cdot J_s \bWu{3}{s} \dx \ds  \Big] = 
\E_{\bP} \Big[ \int_0^\infty \int_{\T^3} ( \JS_s \bWS{3}{s} ) \cdot J_s \bW^{3}_{s} \dx \ds  \Big].
\end{equation}
The rest of the proof consists of a tedious but direct calculation. Using the real-valuedness of \( W \) and the stochastic integral representation \eqref{measure_re:eq_renormalization_3}, we have that 
\begin{align*}
&\int_{\T^3} ( \JS_s \bWS{3}{s} ) \cdot J_s \bW^{3}_{s} \dx \\
=&\int_{\T^3} ( \JS_s \bWS{3}{s} ) \cdot \overline{J_s \bW^{3}_{s}} \dx \\
=& \sum_{n\in \Z^3} \frac{\sigmaS{s}(n) \sigma_s(n)}{\langle n \rangle^2} \sum_{ \substack{n_1,n_2,n_3 \in \Z^3, \\ m_1,m_2,m_3 \in \Z^3 \\ n_{123}=m_{123}=n }} 
 \bigg[ \Big( \sum_{\pi\in S_3} \widehat{V}(n_{\pi(1)}+n_{\pi(2)}) \Big)\Big( \sum_{\tau \in S_3} \widehat{V}(m_{\tau(1)}+m_{\tau(2)}) \Big) \\
&\times \Big( \int_0^s \int_0^{s_1} \int_0^{s_2} \d \WS{n_3}{s_3} \d \WS{n_2}{s_2} \d \WS{n_1}{s_1} \Big)  \overline{\Big( \int_0^s \int_0^{s_1} \int_0^{s_2} \d W^{\scriptscriptstyle{T},s_3}_{m_3} \d W^{\scriptscriptstyle{T},s_2}_{m_2}\d W^{\scriptscriptstyle{T},s_1}_{m_1}\Big)} \bigg]. 
\end{align*}
Taking expectations, we only obtain a non-trivial contribution for \( (n_1,n_2,n_3)=(m_1,m_2,m_3) \), and it follows that 
\begin{align*}
& \E_{\bP} \Big[ \int_{\T^3} ( \JS_s \bWS{3}{s} ) \cdot J_s \bW^{3}_{s} \dx \Big] \\
=& \sum_{n_1,n_2,n_3\in \Z^3} \bigg[ \frac{\sigmaS{s}(n_{123}) \sigma_s(n_{123})}{\langle n_{123} \rangle^2} \Big( \sum_{\pi\in S_3} \widehat{V}(n_{\pi(1)}+n_{\pi(2)}) \Big)^2 \Big(\prod_{j=1}^3  \frac{1}{\langle n_j \rangle^2} \Big) \\
&\times \int_0^s \int_0^{s_1} \int_0^{s_2} \Big( \prod_{j=1}^3 \big( \sigma_{s_j}(n_j) \sigmaS{s_j}(n_j)\big)  \Big) \d s_3 \d s_2 \d s_1 \bigg] \\
=& \frac{1}{6} \hspace{-0.5ex} \sum_{n_1,n_2,n_3\in \Z^3} \hspace{-0.5ex}   \bigg[ \frac{\sigmaS{s}(n_{123}) \sigma_s(n_{123})}{\langle n_{123} \rangle^2} \Big( \sum_{\pi\in S_3} \widehat{V}(n_{\pi(1)}+n_{\pi(2)}) \Big)^2 \Big(\prod_{j=1}^3  \frac{1}{\langle n_j \rangle^2} \Big) 
\Big( \prod_{j=1}^3 \int_0^s \sigma_{s_j}(n_j) \sigmaS{s_j}(n_j) \d s_j \Big) \bigg]. 
\end{align*}
By recalling that \( \sigmaS{s} = \rho_S \cdot \sigma_s \), integrating in \( s \), using Lemma \ref{appendix:lemma_asymptotics}, and symmetry considerations, we obtain that 
\begin{align*}
& \E_{\bP} \Big[ \int_0^\infty \int_{\T^3} ( \JS_s \bWS{3}{s} ) \cdot J_s \bW^{3}_{s} \dx ds \Big] \\
&= \frac{1}{6} \sum_{n_1,n_2,n_3\in \Z^3}   \frac{\rho_S(n_{123}) }{\langle n_{123} \rangle^2} \Big( \sum_{\pi\in S_3} \widehat{V}(n_{\pi(1)}+n_{\pi(2)}) \Big)^2 \Big(\prod_{j=1}^3  \frac{\rho_S(n_j)}{\langle n_j \rangle^2} \Big) 
\int_0^\infty \sigma_s(n_{123})^2 \Big( \prod_{j=1}^3 \rho_s(n_j)^2\Big) \ds  \\
&\geq c \sum_{n_1,n_2,n_3\in \Z^3}   \frac{\rho_S(n_{123}) }{\langle n_{123} \rangle^2} \frac{1}{\langle n_{12} \rangle^{2\beta}}  \Big(\prod_{j=1}^3  \frac{\rho_S(n_j)}{\langle n_j \rangle^2} \Big) 
\int_0^\infty \sigma_s(n_{123})^2 \Big( \prod_{j=1}^3 \rho_s(n_j)^2\Big) \ds \\
&- C \sum_{n_1,n_2,n_3\in \Z^3}   \frac{\rho_S(n_{123}) }{\langle n_{123} \rangle^2} \frac{1}{\langle n_{12} \rangle^{1+2\beta}}  \Big(\prod_{j=1}^3  \frac{\rho_S(n_j)}{\langle n_j \rangle^2} \Big) 
\int_0^\infty \sigma_s(n_{123})^2 \Big( \prod_{j=1}^3 \rho_s(n_j)^2\Big) \ds,
\end{align*}
where \( c,C>0 \) are small and large constants depending only on $V$. The only difference between the two terms lies in the power of \( \langle n_{12} \rangle \). The minor term can easily be estimated from above by 
\begin{align*}
&\sum_{n_1,n_2,n_3\in \Z^3}   \frac{\rho_S(n_{123}) }{\langle n_{123} \rangle^2} \frac{1}{\langle n_{12} \rangle^{1+2\beta}}  \Big(\prod_{j=1}^3  \frac{\rho_S(n_j)}{\langle n_j \rangle^2} \Big) 
\int_0^\infty \sigma_s(n_{123})^2 \Big( \prod_{j=1}^3 \rho_s(n_j)^2\Big) \ds \\
&\lesssim \sum_{n_1,n_2,n_3\in \Z^3} \frac{1}{\langle n_{123} \rangle^2 \langle n_{12} \rangle^{1+2\beta} \langle n_1 \rangle^2 \langle n_2 \rangle^2 \langle n_3 \rangle^2}  \\
&\lesssim 1. 
\end{align*}
Using Lemma \ref{appendix:lemma_well_behaved_truncations}, the main term can be estimated from below by 
\begin{align*}
& \sum_{n_1,n_2,n_3\in \Z^3}   \frac{\rho_S(n_{123}) }{\langle n_{123} \rangle^2} \frac{1}{\langle n_{12} \rangle^{2\beta}}  \Big(\prod_{j=1}^3  \frac{\rho_S(n_j)}{\langle n_j \rangle^2} \Big) 
\int_0^\infty \sigma_s(n_{123})^2 \Big( \prod_{j=1}^3 \rho_s(n_j)^2\Big) \ds  \\
&\gtrsim  \sum_{ \substack{n_1,n_2,n_3\in \Z^3\colon \\ |n_j - S e_j |\leq S/20} } \frac{1}{\langle n_{123} \rangle^2 \langle n_{12} \rangle^{2\beta} \langle n_1 \rangle^2 \langle n_2 \rangle^2 \langle n_3 \rangle^2} \\
&\gtrsim S^{-8-2\beta}  \# \{ (n_1,n_2,n_3) \in (\Z^3)^3 \colon |n_j - S e_j |\leq S/20 \text{ for } j=1,2,3 \} \\
&\gtrsim S^{1-2\beta}. 
\end{align*}
 This completes the proof of the lemma. 
\end{proof}

Before we can upgrade Lemma \ref{singularity:lemma_expectation}, we need the following estimate of the \( \AS{j}{} \). 
\begin{lemma}\label{singularity:lemma_Aj}
Let \( 0 < \beta < 1/2 \), let \( \delta>0 \) sufficiently small, and let \( k \geq 1 \) be sufficiently large depending on \( \beta \). For any \( v \colon \R_{>0} \times \T^3 \rightarrow \R \) and any \( j=1,2,3\), it then holds that 
\begin{equation}\label{singularity:eq_Aj}
\| \AS{j}{s}[v] \|_{L_x^2 }^2 \lesssim \langle s \rangle^{-1-2\beta+20\delta} \Big( Q_s(\bWu{}{})  + \| \IS_s[v] \|_{\cC_x^{-\frac{1}{2}-\delta}}^k + \| \IS_s[v] \|_{H_x^1}^2\Big).
\end{equation}
\end{lemma}

\begin{remark}
As is clear from the proof, this estimate can be slightly refined. Ignoring \( \delta\)-losses, the worst power \( \langle s \rangle^{-1-2\beta} \) only occurs with \( \| \IS_s[v]\|_{H_x^{1-\beta}}^2 \) instead of \( \| \IS_s[v]\|_{H_x^{1}}^2 \). However, \eqref{singularity:eq_Aj} is sufficient for our purposes. 
\end{remark}

\begin{proof}
We treat the estimates for \( j=1,2,3 \) separately. We first estimate \( \AS{1}{s} \), which consists of two terms. For the first summand, we have that 
\begin{align*}
& \Big\| \JS_s \Big( ( V * \lcol (\WSu{}{s})^2 \rcol ) \IS_s[v] \Big)  \Big\|_{L_x^2}^2 \\
&\lesssim \langle s \rangle^{-1-2\beta+4\delta}  \Big\|  \Big( ( V * \lcol (\WSu{}{s})^2 \rcol ) \IS_s[v] \Big) \Big \|_{H_x^{-1+\beta-2\delta}}^2 \\
&\lesssim \langle s \rangle^{-1-2\beta+4\delta} \|  V * \lcol (\WSu{}{s})^2 \rcol  \|_{\cC_x^{-1+\beta-\delta}}^2 \| \IS_s[v] \|_{H_x^{1-\beta}}^2  \\
&\lesssim \langle s \rangle^{-1-2\beta+4\delta} \|  V * \lcol (\WSu{}{s})^2 \rcol  \|_{\cC_x^{-1+\beta-\delta}}^2 \| \IS_s[v] \|_{H_x^{-1}}^{\beta} \| \IS_s[v] \|_{H_x^{1}}^{2-\beta}. 
\end{align*}
Provided that \( k \gtrsim \beta^{-1} \), the desired statement follows from Young's inequality. The estimate for the second summand is similar, except that in the second inequality above we use the random matrix estimate (Proposition \ref{measure_rmt:proposition_estimate}) instead of Hölder's inequality. \\
Next, we estimate \( \AS{2}{s} \). Let \( \eta >0 \) remain to be chosen. Using \eqref{appendix:eq_trilinear_IV} from Lemma \ref{appendix:lemma_trilinear}, we can control the first term in \( \AS{2}{s} \) by 
\begin{align*}
&\Big\| \JS_s \Big( (V* (\IS_s[v])^2 ) \WSu{}{s} \Big)  \Big\|_{L_x^2}^2 \\
&\lesssim \langle s \rangle^{-2+4\delta} \Big \| \langle \nabla \rangle^{-\frac{1}{2}-2\delta}  \Big( (V* (\IS_s[v])^2 ) \WSu{}{s} \Big) \Big\|_{L_x^2}^2 \allowdisplaybreaks[4]\\
&\lesssim \langle s \rangle^{-2+4\delta} \| \WSu{}{s} \|_{\cC_x^{-\frac{1}{2}-\delta}}^2 \| \IS_s[v] \|_{\cC_x^{-\frac{1}{2}-\delta}}^2 \| \IS_s[v] \|_{H_x^{1+4\delta}}^2  \allowdisplaybreaks[4]\\
&\lesssim \langle s \rangle^{-2+12\delta} \| \WSu{}{s} \|_{\cC_x^{-\frac{1}{2}-\delta}}^2 \| \IS_s[v] \|_{\cC_x^{-\frac{1}{2}-\delta}}^2 \| \IS_s[v] \|_{H_x^{1}}^2  \allowdisplaybreaks[4] \\
&\lesssim  \langle s \rangle^{-2+12 \delta + 8 \eta} \| \WSu{}{s} \|_{\cC_x^{-\frac{1}{2}-\delta}}^2   \| \IS_s[v] \|_{\cC_x^{-\frac{1}{2}-\delta}}^{2+\eta}  \| \IS_s[v] \|_{H_x^{1}}^{2-\eta} \\
&\lesssim   \langle s \rangle^{-2+12 \delta + 8 \eta} \Big( \| \WSu{}{s} \|_{\cC_x^{-\frac{1}{2}-\delta}}^{\frac{8}{\eta}} +  \| \IS_s[v] \|_{\cC_x^{-\frac{1}{2}-\delta}}^{\frac{4(2+\eta)}{\eta}}+  \| \IS_s[v] \|_{H_x^{1}}^{2} \Big). 
\end{align*}
After choosing \( \eta = 10 k^{-1} \), the desired estimate follows provided  that \( k \gtrsim (1/2-\beta)^{-1} \). The only difference in the estimate of the second term in \( \AS{2}{s} \) is that we use \eqref{appendix:eq_trilinear_III} instead of \eqref{appendix:eq_trilinear_IV}. \\
We now turn to the estimate of \( \AS{3}{s} \). Arguing exactly as in our estimate for \( \AS{2}{s} \), we obtain that 
\begin{align*}
\Big\| \JS_s \Big( (V* (\IS_s[v])^2 ) \IS_s[v] \Big)  \Big\|_{L_x^2}^2 
\lesssim  \langle s \rangle^{-2+12 \delta + 8 \eta}    \| \IS_s[v] \|_{\cC_x^{-\frac{1}{2}-\delta}}^{4+\eta}  \| \IS_s[v] \|_{H_x^{1}}^{2-\eta}.
\end{align*}
Using Young's inequality, this contribution is acceptable. 
\end{proof} 

We are now ready to upgrade our bound on the expectation from Lemma \ref{singularity:lemma_expectation} into a pointwise statement. The main tool will be the Boué-Dupuis formula. 

\begin{lemma}\label{singularity:lemma_pointwise}
For any \( \delta >0 \), there exists a sequence \( (S_m)_{m=1}^\infty \) converging to infinity such that 
\begin{equation}
\lim_{m\rightarrow \infty} \frac{1}{S_m^{1-2\beta-\delta}} \int_0^\infty \int_{\T^3} \Big( \JSm_s \WSmu{3}{s} \Big) \Big( J_s \Wu{3}{s} \Big) \dx \ds = \infty \qquad \Qinf\text{-a.s.} 
\end{equation}
\end{lemma}

\begin{proof}
Let \( k \geq 1 \) remain to be chosen. We define the auxiliary function 
\begin{equation}
G_S = \frac{1}{S^{1-2\beta-\delta}} \int_0^\infty \int_{\T^3} \Big( \JS_s \WSu{3}{s} \Big) \Big( J_s \Wu{3}{s} \Big) \dx \ds + \sup_{0\leq s < \infty} \| \Wu{}{s} \|_{\cC_x^{-\frac{1}{2}-\delta}(\T^3)}^k. 
\end{equation}
We will now show that 
\begin{equation}
\lim_{S\rightarrow \infty} \E_{\Qinf} \Big[ e^{-G_S}\Big] =0 ,
\end{equation}
which implies the desired result. We could switch from \( (\Wu{}{},\Qinf) \) to \( (W,\bP) \), which we have done several times above. Since the \( \AS{j}{s} \) in \eqref{singularity:eq_A1}-\eqref{singularity:eq_A3} are defined in terms of \( \Wu{}{} \), however, we decided not to change the measure.  \\
We define \( A^j_s \) similar as in \eqref{singularity:eq_A1}-\eqref{singularity:eq_A3}, but with \( \JS_s \) replaced by \( J_s \), \( \IS_s \) replaced by \( I_s \), and \( \WSu{}{} \) replaced by \( \Wu{}{} \). Since all our estimates for \( \AS{j}{} \) were uniform in \(S \geq 1 \), they also hold for \( A^j \). Using the Boué-Dupuis formula (Theorem \ref{thm:bd_formula}) and the cubic binomial formula,
we have that 
\begin{align}
-& \log \E_{\Qinf} \Big[ e^{-G_S} \Big] \notag \\
=& \inf_{v\in \bH}  \E_{\Q^u} \bigg[ \frac{1}{S^{1-2\beta-\delta}} \int_0^\infty \int_{\T^3} \bigg( \JS_s \Big( \lcol (V* (\WSu{}{s}+ \IS_s[v])^2 ) (\WSu{}{s}+\IS_s[v]) \rcol \Big)   \notag \\
\times& J_s \Big( \hspace{-0.5ex} \lcol (V* (\Wu{}{s}+ I_s[v])^2 ) (\Wu{}{s}+I_s[v]) \rcol \Big) \bigg) \dx \ds + \sup_{0\leq s < \infty} \| \Wu{}{s} + I_s[v] \|_{\cC_x^{-\frac{1}{2}-\delta}}^k + \frac{1}{2} \| v\|_{L_s^2 L_x^2}^2 \bigg]  \notag \\
=& \E_{\Qinf} \bigg[ \frac{1}{S^{1-2\beta-\delta}} \int_0^\infty \int_{T^3}  \JS_s \Big( \lcol (V* (\WSu{}{s})^2 )\WSu{}{s} \rcol \Big) J_s \Big( \lcol (V* (\Wu{}{s})^2 ) \Wu{}{s} \rcol \Big)  \dx \ds \bigg]  \label{singularity:eq_ptw_p1}\\
+& \inf_{v\in \bH} \E_{\Qinf} \bigg[  \sup_{0\leq s < \infty} \| \Wu{}{s} + I_s[v] \|_{\cC_x^{-\frac{1}{2}-\delta}}^k + \frac{1}{2} \int_0^\infty \| v_s \|_{L_x^2}^2 \ds   \label{singularity:eq_ptw_p2}\\
+& \frac{1}{S^{1-2\beta-\delta}} \sum_{j=1}^3 \int_0^\infty \int_{\T^3} (\JS_s \bWSu{3}{s}) A^j_s[v]\dx \ds  + \frac{1}{S^{1-2\beta-\delta}} \sum_{j=1}^3 \int_0^\infty \int_{\T^3} (J_s \bWu{3}{s}) \AS{j}{s}[v]\dx \ds  \label{singularity:eq_ptw_p3}\\
+&\frac{1}{S^{1-2\beta-\delta}} \sum_{i,j=1}^3 \int_0^\infty \int_{\T^3} \AS{i}{s}[v] A^j_s[v] \dx \ds \bigg].  \label{singularity:eq_ptw_p4}
\end{align}
The main term is given by \eqref{singularity:eq_ptw_p1}. By Lemma \ref{singularity:lemma_expectation}, we see that \eqref{singularity:eq_ptw_p1} converges to infinity as \( S \rightarrow \infty \). Thus, it remains to obtain a lower bound on the variational problem in \eqref{singularity:eq_ptw_p2}-\eqref{singularity:eq_ptw_p4}. The terms in \eqref{singularity:eq_ptw_p2} are nonnegative and help with the lower bound. In contrast, the terms in \eqref{singularity:eq_ptw_p3} and \eqref{singularity:eq_ptw_p4} are viewed as errors and will be estimated in absolute value. \\
Regarding \eqref{singularity:eq_ptw_p2}, we briefly note that 
\begin{equation*}
\E_{\Q^u}\bigg[  \sup_{0\leq s < \infty} \| \Wu{}{s} + I_s[v] \|_{\cC_x^{-\frac{1}{2}-\delta}}^k  \bigg] \geq \frac{1}{2} \E_{\Qinf}\bigg[  \sup_{0\leq s < \infty} \| I_s[v] \|_{\cC_x^{-\frac{1}{2}-\delta}}^k  \bigg] - C. 
\end{equation*}
In the estimates below, we will often use that \( \AS{j}{s}[v] = 0 \) for all \( s \gg S \). We begin with the first term in  \eqref{singularity:eq_ptw_p3}. We have that 
\begin{align}
&\Big|  \frac{1}{S^{1-2\beta-\delta}} \int_0^\infty \int_{\T^3} (\JS_s \bWSu{3}{s}) A^j_s[v]\dx \dt  \Big|  \notag \\
&\leq  \frac{1}{S^{1-2\beta-\delta}} \int_0^\infty 1\{ s\lesssim S\} \langle s \rangle^{-\frac{1}{2}} \| \JS_s \bWSu{3}{s}\|_{L^2}^2 \ds + \frac{1}{S^{1-2\beta-\delta}} \int_0^\infty 1\{ s\lesssim S\} \langle s \rangle^{\frac{1}{2}} \| A^j_s[v] \|_{L^2}^2 \ds.
\label{singularity:eq_ptw_p5}
\end{align}
For the first term in \eqref{singularity:eq_ptw_p5}, we obtain from Lemma \ref{measure_re:lemma_stochastic_objects_III}   that 
\begin{equation}\label{singularity:eq_ptw_p6}
\begin{aligned}
&\E_{\Qinf} \bigg[  \frac{1}{S^{1-2\beta-\delta}} \int_0^\infty 1\{ s\lesssim S\} \langle s \rangle^{-\frac{1}{2}} \| \JS_s \bWSu{3}{s}\|_{L^2}^2 \ds\bigg] \\
&\lesssim \frac{1}{S^{1-2\beta-\delta}} \int_0^\infty 1\{ s\lesssim S\} \langle s \rangle^{-\frac{1}{2}-2\beta+2\delta}  \E_{\Q^u} \Big[ \| \bWSu{3}{s}\|_{H_x^{-\frac{3}{2}+\beta-\delta}}^2 \Big] \ds \\
&\lesssim \frac{1}{S^{1-2\beta-\delta}} \int_0^\infty 1\{ s\lesssim S\} \langle s \rangle^{-\frac{1}{2}-2\beta+2\delta}  \ds  \\
&\lesssim 1. 
\end{aligned}
\end{equation}
For the second term in \eqref{singularity:eq_ptw_p5}, we obtain from Lemma \ref{singularity:lemma_Aj} that
\begin{equation}\label{singularity:eq_ptw_p7}
\begin{aligned}
&\E_{\Qinf} \bigg[ \frac{1}{S^{1-2\beta-\delta}} \int_0^\infty 1\{ s\lesssim S\} \langle s \rangle^{\frac{1}{2}} \| A^j_s[v] \|_{L^2}^2 \ds \bigg] \\
&\lesssim \frac{1}{S^{1-2\beta-\delta}} \int_0^\infty 1\{ s \lesssim S \} \langle s \rangle^{-\frac{1}{2}-2\beta} \E_{\Qinf} \Big[ Q_s(\Wu{}{}) \Big] ds  \\
&+ \frac{1}{S^{1-2\beta-\delta}} \E_{\Q^u} \bigg[ \int_0^\infty 1\{s \lesssim S\} \langle s \rangle^{-\frac{1}{2}-2\beta} \Big( \| I_s[v]\|_{\cC_x^{-\frac{1}{2}-\delta}}^k + \| I_s[v] \|_{H_x^1}^2 \Big) \ds \bigg]  \\
&\lesssim 1 + S^\delta \max( S^{-\frac{1}{2}}, S^{2\beta-1}) \E_{\Qinf} \bigg[ \sup_{0\leq s <\infty}  \Big( \| I_s[v]\|_{\cC_x^{-\frac{1}{2}-\delta}}^k + \| I_s[v] \|_{H_x^1}^2 \Big)  \bigg]  \\
&\lesssim 1+ S^\delta \max( S^{-\frac{1}{2}}, S^{2\beta-1}) \E_{\Qinf} \bigg[ \sup_{0\leq s <\infty}   \| I_s[v]\|_{\cC_x^{-\frac{1}{2}-\delta}}^k + \| v\|_{L_s^2 L_x^2}^2 \bigg] . 
\end{aligned}
\end{equation}
In the last line, we also Lemma \ref{appendix:lemma_It}. Since \( S \rightarrow \infty \), this contribution can be absorbed in the coercive term \eqref{singularity:eq_ptw_p5}. The estimate of the second summand in \eqref{singularity:eq_ptw_p3} is exactly the same.\\

Regarding the error terms in \eqref{singularity:eq_ptw_p4}, we have that 
\begin{align*}
&\Big| \frac{1}{S^{1-2\beta-\delta}} \sum_{i,j=1}^3 \int_0^\infty \int_{\T^3} \AS{i}{s}[v] A^j_s[v] \dx \ds\Big| \\
&\lesssim  \frac{1}{S^{1-2\beta-\delta}} \sum_{j=1}^3 \int_0^\infty 1\{ s\lesssim S \} \Big(  \|\AS{i}{s}[v] \|_{L_x^2}^2 + \| A^j_s[v] \|_{L_x^2}^2 \Big) \dx \ds.
\end{align*}
The right-hand side can now be controlled using the same (or simpler) estimates as for the second summand in \eqref{singularity:eq_ptw_p5}. This completes the proof. 
\end{proof}

Essentially the same estimates as in the previous proof can also be used to control  the minor terms in \eqref{singularity:eq_representation_2} and \eqref{singularity:eq_representation_3}. We record them in the following lemma. 

\begin{lemma}\label{singularity:lemma_minor}
Let \( 0 < \beta <1/2 \), let \( \delta >0 \) and let \(j=1,2,3\). Then, it holds that 
\begin{align}
\lim_{S\rightarrow \infty} \E_{\Qinf} \bigg[ \bigg( \frac{1}{S^{\frac{1}{2}-\beta+\delta}} \int_0^\infty \int_{\T^3} \JS_s \WSu{3}{s} \d B^u_s \bigg)^2 \bigg] &=0, \label{singularity:eq_minor_1}\\
\lim_{S\rightarrow \infty} \E_{\Qinf} \bigg[  \frac{1}{\max( S^{1-3\beta+\delta},1)} \bigg|  \int_0^\infty \int_{\T^3} \AS{j}{s}[u] \cdot J_s \bWu{3}{s} \dx \ds \bigg| \bigg]  &=0,\label{singularity:eq_minor_2}\\
\lim_{S\rightarrow \infty} \E_{\Qinf} \bigg[ \bigg( \frac{1}{\max(S^{\frac{1}{2}-2\beta+\delta},1)} \int_0^\infty \int_{\T^3} \AS{j}{s}[u] \d B^u_s \bigg)^2 \bigg] &=0, \label{singularity:eq_minor_3} \\
\lim_{S\rightarrow \infty}  \E_{\Qinf} \bigg[  \frac{1}{S^\delta} \bigg|  \int_0^\infty \int_{\T^3} ( \JS_s \bWS{3}{s} ) ( J_s \langle \nabla\rangle^{-\frac{1}{2}} \lcol ( \langle \nabla \rangle^{-\frac{1}{2}} \Wu{}{s})^n \rcol  ) \dx \ds \bigg| \bigg]&=0.  \label{singularity:eq_minor_4}
\end{align}
\end{lemma}
\begin{proof}
We begin with the proof of \eqref{singularity:eq_minor_1}. Using Itô's isometry, we have that 
\begin{align*}
 \E_{\Q^u} \bigg[ \bigg( \frac{1}{S^{\frac{1}{2}-\beta+\delta}} \int_0^\infty \int_{\T^3} \JS_s \WSu{3}{s} \d B^u_s \bigg)^2 \bigg]   = \frac{1}{S^{1-2\beta+2\delta}} \int_0^\infty  \E_{\Qinf} \Big[ \| \JS_s \WSu{3}{s}  \|_{L^2_x}^2 \Big] \ds . 
\end{align*}
Arguing essentially as in \eqref{singularity:eq_ptw_p6}, we obtain that 
\begin{equation*}
 \frac{1}{S^{1-2\beta+2\delta}} \int_0^\infty  \E_{\Qinf} \Big[ \| \JS_s \WSu{3}{s}  \|_{L^2_x}^2 \Big] \ds \lesssim  \frac{1}{S^{1-2\beta+2\delta}} \int_0^\infty 1\{ s \lesssim S \} \langle s\rangle^{-2\beta+\delta} \ds \lesssim S^{-\delta},
\end{equation*}
which yields \eqref{singularity:eq_minor_1}.  \\

We now turn to \eqref{singularity:eq_minor_2}. Using Lemma \ref{singularity:lemma_Aj} and Corollary \ref{measure_drift:corollary_bounds_u}, we have for all \( \epsilon >0 \) that 
\begin{equation}\label{singularity:eq_minor_p1}
\E_{\Qinf} \Big[ \| \AS{j}{s}[u] \|_{L^2_x}^2 \Big] \lesssim \langle s \rangle^{-1-2\beta+20\epsilon} ( 1+ \big( \langle s \rangle^{1 - (\frac{1}{2}+\beta)+\epsilon}\big)^2 ) \lesssim \langle s \rangle^{-4\beta+40\epsilon}. 
\end{equation}
Using Lemma \ref{measure_re:lemma_stochastic_objects_I} and \eqref{singularity:eq_minor_p1}, we obtain that 
\begin{align*}
&\E_{\Qinf} \bigg[  \bigg|  \int_0^\infty \int_{\T^3} \AS{j}{s}[u] \cdot J_s \bWu{3}{s} \dx \ds \bigg| \bigg] \\
&\lesssim \E_{\Qinf} \bigg[ \int_0^\infty 1\{ s \lesssim S\} \langle s \rangle^{-\beta} \| J_s \bWu{3}{s} \|_{L_x^2}^2 \ds \bigg]  + \E_{\Qinf} \bigg[ \int_0^\infty 1\{ s\lesssim S\} \langle s\rangle^{\beta} \| \AS{j}{s}[u] \|_{L^2_x}^2 \ds \bigg] \\
&\lesssim \int_0^\infty 1\{ s \lesssim S \} \langle s \rangle^{-3\beta+40\epsilon} \ds \lesssim S^{-\frac{\delta}{2}} \max(1,S^{1-3\beta+\delta}). 
\end{align*}
Next, we prove \eqref{singularity:eq_minor_2}. Using Itô's isometry and \eqref{singularity:eq_minor_p1}, we have that 
\begin{align*}
\E_{\Qinf} \bigg[ \bigg(   \int_0^\infty \int_{\T^3} \AS{j}{s}[u] \d B^u_s \bigg)^2 \bigg]  
&\lesssim \E_{\Qinf}  \bigg[ \int_0^\infty 1\{ s\lesssim S\}  \| \AS{j}{s}[u] \|_{L^2_x}^2 \ds \bigg] \\
&\lesssim \int_0^\infty 1\{ s \lesssim S \} \langle s \rangle^{-4\beta+40\epsilon} \ds \\
&\lesssim S^{-\delta} \max( S^{\frac{1}{2}-2\beta+\delta}, 1)^2. 
\end{align*}
Finally, we turn to \eqref{singularity:eq_minor_4}, which is the most regular term. We first recall the algebraic identity \( \JS_s \bWS{3}{s} = \JS \bWSu{3}{s} + \sum_{j=1}^3 \AS{j}{s}[u] \). Then, Lemma \ref{measure_re:lemma_stochastic_objects_I} and \eqref{singularity:eq_minor_p1} yield 
\begin{equation}\label{singularity:eq_minor_p2}
\E_{\Qinf} \Big[ \| \JS_s \bWS{3}{s} \|_{L_x^2}^2 \Big] \lesssim \langle s \rangle^{-2\beta +2\epsilon}. 
\end{equation}
From Lemma \ref{measure_re:lemma_high_power}, we have that 
\begin{equation}\label{singularity:eq_minor_p3}
\E_{\Qinf} \Big[ \| J_s \langle \nabla \rangle^{-\frac{1}{2}} \lcol \big( \langle \nabla \rangle^{-\frac{1}{2}} \Wu{}{s} \big)^n \rcol \|_{L_x^2}^2 \Big] \lesssim \langle s \rangle^{-4+2\epsilon}. 
\end{equation}
By combining \eqref{singularity:eq_minor_p2} and \eqref{singularity:eq_minor_p3}, we obtain
\begin{align*}
  &\E_{\Qinf} \bigg[  \bigg|  \int_0^\infty \int_{\T^3} ( \JS_s \bWS{3}{s} ) ( J_s \langle \nabla\rangle^{-\frac{1}{2}} \lcol ( \langle \nabla \rangle^{-\frac{1}{2}} \Wu{}{s})^n \rcol  ) \dx \ds \bigg| \bigg] \\
&\lesssim \E_{\Qinf} \bigg[ \int_0^\infty 1\{ s \lesssim S\}   \langle s \rangle^{-1} \| \JS_s \bWS{3}{s} \|_{L_x^2}^2 ds \bigg] + \E_{\Qinf}  \bigg[ \int_0^\infty \langle s \rangle \| J_s \langle \nabla \rangle^{-\frac{1}{2}} \lcol \big( \langle \nabla \rangle^{-\frac{1}{2}} \Wu{}{s} \big)^n \rcol \|_{L_x^2}^2 \ds \bigg] \\
&\lesssim \int_0^\infty \Big( \langle s \rangle^{-1-2\beta+2\epsilon} + \langle s \rangle^{-3+\epsilon} \Big) \ds \lesssim 1. 
\end{align*}
\end{proof}

We are now ready to prove the main result of this section. 

\begin{proof}[Proof of Proposition \ref{singularity:proposition_main}:] 
We recall from Lemma \ref{singularity:lemma_representation} that
\begin{equation}\label{singularity:eq_p1}
 \frac{1}{S^{1-2\beta-\delta}} \int_{\T^3} \bWS{4}{\infty} \dx  = -  \frac{4\lambda}{S^{1-2\beta-\delta}} \int_0^\infty \int_{\T^3} ( \JS_s \bWSu{3}{s} )  \cdot J_s \bWu{3}{s} \dx \ds + \mathcal{R}^{\scriptscriptstyle{S}}(\bWu{}{},u), 
\end{equation}
where the remainder \( \mathcal{R}(\bWu{}{},u) \) contains the terms from \eqref{singularity:eq_representation_2} and \eqref{singularity:eq_representation_3} with an additional \( S^{-1+2\beta+\delta}\). By Lemma  \ref{singularity:lemma_pointwise}, there exists a deterministic sequence \( S_m \) such that the first summand in \eqref{singularity:eq_p1} converges to \( - \infty \) almost surely with respect to \( \Qinf \). Since \( 0<\beta<1/2 \), we have that  
\begin{equation*}
1-2\beta > \max\Big( \frac{1}{2} -\beta, 1 - 3\beta, \frac{1}{2}-2\beta, 0 \Big) .
\end{equation*}

Using Lemma  \ref{singularity:lemma_minor}, this implies that the remainder \(  \mathcal{R}^{\scriptscriptstyle{S}}(\bWu{}{},u) \) converges to zero in \( L^1(\Qinf) \). By passing to a subsequence if necessary, we can assume that \( \mathcal{R}^{\scriptscriptstyle{S_m}}(\bWu{}{},u) \) converges to zero almost surely with respect to \( \Qinf \). Using \eqref{singularity:eq_p1}, this implies that 
\begin{equation*}
\lim_{m\rightarrow \infty}\frac{1}{S_m^{1-2\beta-\delta}}  \int_{\T^3} \mathbb{W}^{\scriptscriptstyle{S_m},4}_\infty \dx = -\infty \qquad \Qinf\text{-a.s.}
\end{equation*}
Using \( \beta < 1/2 \) and  Lemma \ref{singularity:lemma_GFF}, the integral \(  S^{-1+2\beta+\delta} \int_{\T^3} \bWS{4}{\infty} \dx \)  converges to zero in \( L^2(\bP) \). By passing to another subsequence if necessary, we obtain that 
\begin{equation*}
\lim_{m\rightarrow \infty}\frac{1}{S_m^{1-2\beta-\delta}}  \int_{\T^3} \mathbb{W}^{\scriptscriptstyle{S_m},4}_\infty \dx = 0 \qquad \bP\text{-a.s.}
\end{equation*}
Since \( \mu_\infty \) is absolutely continuous with respect to \( \nu_\infty= (W_\infty)_{\#} \Qinf \) and  \( \cg= \Law_{\bP}(W_\infty)  \), this implies \eqref{singularity:eq_GFF} and \eqref{singularity:eq_Gibbs}. 
\end{proof}

Equipped with Corollary \ref{measure_var:corollary_cT} and Proposition \ref{singularity:proposition_main}, we now provide a short proof of Theorem \ref{theorem:singularity}. 

\begin{proof}[Proof of Theorem \ref{theorem:singularity}:]
If \( 0< \beta < 1/2 \), then the mutual singularity of the Gibbs measure \( \mu_\infty \) and the Gaussian free field \( \cg \) directly follows from Proposition \ref{singularity:proposition_main}.\\
If \( \beta > 1/2 \),  we claim that for all \( p \geq 1 \) that
\begin{equation}\label{singularity:eq_Lp}
\frac{\d \muT}{\d \cg} \in L^p(\cg) 
\end{equation}
with uniform bounds in \( T\geq 1 \). Since \( \mu_T \) converges weakly to \( \mu_\infty \), this implies the absolute continuity \( \mu_\infty\ll \cg \). \\
In order to prove the claim, we recall that \( \muT = (W_\infty)_\# \tmuT \ \) and \( \cg = (W_\infty)_\# \bP \). Furthermore, we see from \eqref{sc:eq_tmu} that the density $\mathrm{d}\tmuT/\mathrm{d}\mathbb{P}$ is a function of $W_\infty$. As a result, we obtain for all $p \geq 1$ that
\begin{equation*}
\int \Big( \frac{\mathrm{d}\tmuT}{\mathrm{d}\mathbb{P}}\Big)^p \mathrm{d}\mathbb{P} = \int  \Big( \frac{\mathrm{d}\muT}{\mathrm{d}\cg}\Big)^p \mathrm{d}\cg 
\end{equation*}
Thus, it suffices to bound the density \( \d \tmuT/ \d \bP \) in  \( L^p(\bP) \).  From the definition of \( \tmuT \) (Definition \ref{sc:definition_tmu}) and the definition of the renormalized potential energy in \eqref{measure_var:eq_renormalized_V}, we have that 
\begin{align*}
\Big( \frac{\d \tmuT}{\d \bP} \Big)^p &=  \frac{1}{\big( \cZTl\big)^p} \exp\Big( - p \lcol \cVT( \Winf) \rcol \Big) \\
&=  \frac{1}{\big( \cZTl\big)^p} \exp\Big( - \frac{\lambda p}{4} \int_{\T^3} \lcol (V* (\Winf)^2) (\Winf)^2 \rcol \dx - p \cT\Big) \\
&= \frac{\cZ^{\scriptscriptstyle{T},p\lambda}}{\big( \cZTl\big)^p} \exp( c^{\scriptscriptstyle{T}}_{p\lambda} -  p \cT)  \cdot  \frac{1}{\cZ^{\scriptscriptstyle{T},p\lambda}} \exp\Big( -  \lcol \cV^{\scriptscriptstyle{T},p\lambda}( \Winf) \rcol \Big).
\end{align*}
The first two factors are uniformly bounded in \( T \) by Proposition \ref{measure_var:proposition} and Corollary \ref{measure_var:corollary_cT}. The last factor is uniformly bounded in \( L^1(\bP) \)  for all \( T \geq 1 \) since we only replaced the coupling constant \( \lambda\) by \( p \lambda \). This completes the proof of the claim \eqref{singularity:eq_Lp}. 

\end{proof}

\begin{appendix}

\section{Probability theory}

In this section we recall two concepts from probability theory, namely, Gaussian hypercontractivity and multiple stochastic integrals.

\subsection{Gaussian hypercontractivity}
In several places of this paper, we reduced probabilistic \( L^p\)-bounds to probabilistic \( L^2\)-bounds using Gaussian hypercontractivity, which is closely related to logarithmic Sobolev embeddings. In the dispersive PDE community, among others, the resulting estimates are known as Wiener chaos estimates. A version of the following lemma can be found in \cite[Theorem I.22]{Simon74}, \cite[Theorem 1.4.1]{Nualart06}, and most papers on random dispersive PDE.

\begin{lemma}\label{prelim:lemma_hypercontractivity}
Let \( k \geq 1 \) and let \( f\colon ( \R_{>0} \times \Z^3)^k \rightarrow \mathbb{C} \) be deterministic, bounded,  and measurable. For any \( t \geq 0 \), define the random variable 
\begin{equation}
X_t = \sum_{n_1,\hdots,n_k \in \Z^3}  \int_0^{t} \int_0^{t_1} \hdots \int_0^{t_{k-1}} f(t_1,n_1,\hdots,t_k,n_k) \d W_{t_k}^{n_k} \d W_{t_{k-1}}^{n_{k-1}} \hdots \d W_{t_{1}}^{n_1}. 
\end{equation} 
Then, it holds for all \( p \geq 2\) that 
\begin{equation}
\| X_t \|_{L^p(\Omega)} \leq (p-1)^{\frac{k}{2}} \| X_t \|_{L^2(\Omega)}. 
\end{equation}
\end{lemma}

\subsection{Multiple stochastic integrals}\label{section:multiple_stochastic_integrals}
This section is based on \cite[Section 1.1]{Nualart06} and we refer the reader to this excellent book for more details. Of particular importance to us is \cite[Example 1.1.2]{Nualart06}, which discuss the specific case of  a \( d \)-dimensional Brownian motion. \\

We identify \( \WP \) with a Gaussian process on \( \calH = L^2( \R_{>0} \times \Z^3, \dt \otimes \d n )\), where \( \dt \) is the Lebesgue measure and \( \d n \) is the counting measure. For any \( h \in \calH \), we define
\begin{equation}
\WP[h]= \sum_{n\in \Z^3} \int_0^\infty h(t,n) \d \Wt[n]. 
\end{equation}
For any \( h,h^\prime \in \calH \), we have that 
\begin{equation}\label{appendix:eq_covariance}
\E\Big[ \WP[h] \WP[h^\prime] \Big] = \sum_{n\in\Z^3} \int_0^\infty h(t,n) h^\prime(t,-n) \frac{\sigmaT{t}(n)^2}{\langle n\rangle^2} \dt. 
\end{equation}

Since we did not include a complex conjugate in the left-hand side of \eqref{appendix:eq_covariance}, we note that this does not yield a positive-definite bilinear form. We also did not include the weight \( \rhoT{t}(n)^2/\langle n \rangle^2\) in the definition of \( \calH \). Thus, the ``covariance'' in \eqref{appendix:eq_covariance} does not coincide with the inner product on \( \calH \) and instead is only dominated by it. As is clear from \cite[Section 1.1]{Nualart06}, this only requires minor modifications in both the arguments and formulas. \\

For any \( k \geq 1 \) and any function \( f \in \calH_k = L^2\big( ( \R_{>0} \times \Z^3)^k, \bigotimes_{j=1}^k ( \dt \otimes \d n)\big) \), the multiple stochastic integral 
\begin{equation}\label{appendix:eq_multiple_stochastic_integral}
\calI_k[f] = \sum_{n_1,\hdots,n_k \in \Z^3} \int_0^\infty \hdots \int_0^\infty f(t_1,n_1,\hdots,t_k,n_k) \d \W{t_k}{n_k} \hdots \d \W{t_1}{n_1}
\end{equation}
can be defined as in \cite[Section 1.1.2]{Nualart06}. If \( f \) is symmetric in the pairs \( (t_1,n_1), (t_2,n_2), \hdots, (t_k,n_k) \), we can relate the multiple stochastic integral to an iterated stochastic integral.

\begin{lemma}\label{appendix:lemma_iterated_vs_multiple}
Let \( k \geq 1 \) and let \( f \in \calH_k \) be symmetric. Then, it holds that 
\begin{equation}
\calI_k[f] = k! \sum_{n_1,\hdots,n_k \in \Z^3} \int_0^\infty \int_0^{t_1} \hdots \int_0^{t_{k-1}}f(t_1,n_1,\hdots,t_k,n_k) \d \W{t_k}{n_k} \hdots \d \W{t_1}{n_1},
\end{equation}
where the right-hand side is understood as an iterated Itô integral. 
\end{lemma}
This lemma follows from \cite[(1.27)]{Nualart06} and the discussion below it. The primary reason for working with multiple stochastic integrals instead of iterated stochastic integrals is the simpler representation of their products. In order to state the product  formula in Lemma \ref{appendix:lemma_product_formula} below, we need one further definition.

\begin{definition}[Contraction]\label{appendix:definition_contraction}
Let \( k, l \geq 1 \) and let \( f \in \calH_k \) and \( g \in \calH_l \) be symmetric. For any \( 0 \leq r \leq \min(k,l) \), we define the contraction of \( r \) indices by 
\begin{align*}
&(f \otimes_r g )(t_1,n_1,\hdots, t_{k+l-2r}, n_{k+l-2r}) \\
\defe&\sum_{m_1,\hdots,m_r \in \Z^3} \int_0^\infty \hdots \int_0^\infty \bigg[ f(t_1,n_1,\hdots,t_{k-r},n_{k-r},s_1,m_1,\hdots,s_r,m_r)  \\
&\times g(t_{k+1-r},n_{k+1-r},\hdots,t_{k+l-2r},n_{k+l-2r},s_1,-m_1,\hdots,s_r,-m_r) \prod_{j=1}^k \frac{\sigmaT{s_j}(m_j)^2}{\langle m_j \rangle^2} \bigg]\d s_r \hdots \d s_1. 
\end{align*}
\end{definition}
The reader should note the relationship to the covariance \eqref{appendix:eq_covariance}. If \( f,g \in \calH = \calH_1 \), then 
\begin{equation*}
\E \Big[ \WP[f] \WP[g] \Big] = f \otimes_1 g. 
\end{equation*}
A slight modification of \cite[Proposition 1.1.3]{Nualart06} then yields the following result. 

\begin{lemma}[Product formula]\label{appendix:lemma_product_formula}
For any \( k, l \geq 1 \) and any symmetric \( f \in \calH_k \) and \( g \in \calH_l \), it holds that 
\begin{equation}
\calI_k[f]  \cdot \calI_l[g] = \sum_{r=0}^{\min(k,l)} r! {k \choose r}  {l \choose r} \calI_{k+l-2r}[ f \otimes_r g]. 
\end{equation}
\end{lemma}

\section{Auxiliary analytic estimates}

In this section, we record several auxiliary results, which have been placed here to not interrupt the flow of the argument.

\subsubsection*{Harmonic analysis} 

We record a non-stationary phase argument and several standard trilinear product estimates. 
\begin{lemma}[Asymptotics of $\widehat{V}$]\label{appendix:lemma_asymptotics}
There exists a constant \( c=c_\beta \in \R \) such that 
\begin{equation}\label{appendix:eq_Vhat}
\Big|\widehat{V}(n)- \frac{c_\beta}{\langle n\rangle^\beta} \Big| \lesssim \frac{1}{\langle n \rangle^{\beta+1}}. 
\end{equation}
\end{lemma}

\begin{remark}
On the Euclidean space \( \R^3 \), instead of the periodic torus \( \T^3\), the Fourier transform of \( |x|^{\beta-3} \) is given exactly by \( c_\beta |\xi|^{-\beta} \). At high frequencies, the Fourier transform \( \widehat{V} \) is determined by the singularities of \( V \), and hence the difference between \( \R^3 \) and \( \T^3 \) should not be essential. In fact, a more precise description of the asymptotics of \( \widehat{V} \) is given by \( c_\beta |n|^{-\beta} 1\{ n \neq 0\} + \mathcal{O}_M( \langle n \rangle^{-M} ) \), but it is easier to work with \eqref{appendix:eq_Vhat}. 
\end{remark}
\begin{proof}
We denote by \( \cF_{\R^3} \) the Fourier transform on \( \R^3 \) given by
\begin{equation*}
\cF_{\R^3} f(\xi) = \int_{\R^3} f(x) e^{-i \langle \xi , x \rangle} \dx. 
\end{equation*}
Let \( \{ \chi_N\}_{N\geq 1} \) be as in \eqref{notation:eq_chi}, which we naturally extend from \( \Z^3 \) to \( \R^3 \). Because we require additional room, we define for any \( x \in \T^3 \) and \( N \geq 1 \) the function
\begin{equation}
\widetilde{\chi}_N(x)\defe \chi_N(100 x). 
\end{equation}
Let \( n \in \Z^3 \backslash \{ 0 \} \). Using the assumptions on the interaction potential \(V \), we obtain that 
\begin{align*}
\widehat{V}(n) &= \int_{\T^3} V(x) e^{-i \langle n,x\rangle} \dx \\
&= \int_{\T^3} V(x)\widetilde{\chi}_1(x) e^{-i \langle n , x \rangle} \dx + \int_{\T^3} V(x) (1-\widetilde{\chi}_1(x)) e^{-i \langle n  , x \rangle} \dx  \\
&= \int_{\R^3} |x|^{-(3-\beta)} \widetilde{\chi}_1(x) e^{-i \langle n, x \rangle} \dx + \int_{\T^3} V(x) (1-\widetilde{\chi}_1(x)) e^{-i \langle n  , x \rangle} \dx \\
&= \cF_{\R^3}\big[ |x|^{-(3-\beta)}\big](\xi) - \sum_{N\geq 2} \int_{\R^3} |x|^{-(3-\beta)} \widetilde{\chi}_N(x) e^{-i \langle n , x \rangle} \dx + \int_{\T^3} V(x) (1-\widetilde{\chi}_1(x)) e^{-i \langle n  , x \rangle} \dx . 
\end{align*}
The first summand is given exactly by \( c_\beta \| n\|_2^{-\beta} \). A non-stationary phase argument for the second and third term shows that they are bounded by \( \mathcal{O}_M( \langle n \rangle^{-M}) \) for all \( M \geq 1\). This implies that 
\begin{equation*}
\widehat{V}(n)  = c_\beta \| n \|_2^{-\beta} 1\{ n \neq 0\} + \mathcal{O}_M( \langle n \rangle^{-M}). 
\end{equation*}
Since \( \| n \|_2^{-\beta} = \langle n \rangle^{-\beta} + \mathcal{O}( \langle n \rangle^{-1-\beta} ) \), this leads to \eqref{appendix:eq_Vhat}. 
\end{proof}

The following estimates are used in the paper to control several minor error terms. 
\begin{lemma}[Trilinear estimates]\label{appendix:lemma_trilinear}
For any sufficiently small \( \delta >0\), we have for all \( f,g,h\in C_x^\infty(\T^3) \) the estimates
\begin{align}
\Big\| \langle \nabla \rangle^{\frac{1}{2}+\delta} \Big( (V*(fg)) \, h \Big) \Big\|_{L^1_x}  
&\lesssim \| f \|_{H^{\frac{1}{2}+2\delta}_x} \| g \|_{H^{\frac{1}{2}+2\delta}_x} \| h \|_{\cC^{\frac{1}{2}+2\delta}_x} \label{appendix:eq_trilinear_I}, \\
\Big\| \langle \nabla \rangle^{\frac{1}{2}+\delta} \Big( (V*(fg)) \, h \Big) \Big\|_{L^1_x} 
&\lesssim    \| f \|_{\cC^{\frac{1}{2}+2\delta}_x} \| g \|_{H^{\frac{1}{2}+2\delta}_x} \| h \|_{H^{\frac{1}{2}+2\delta}_x} \label{appendix:eq_trilinear_II}, \\
\Big \| \langle \nabla \rangle^{-\frac{1}{2}-2\delta} \Big( (V*(fg)) \, h \Big) \Big\|_{L^2_x} 
&\lesssim \| f \|_{\cC_x^{-\frac{1}{2}-\delta}}  \Big( \| g \|_{\cC_x^{-\frac{1}{2}-\delta}} \| h \|_{H_x^{1+4\delta}} +  \| g \|_{H_x^{1+4\delta}} \| h \|_{\cC_x^{-\frac{1}{2}-\delta}} \Big) \label{appendix:eq_trilinear_III} \\
\Big \| \langle \nabla \rangle^{-\frac{1}{2}-2\delta} \Big( (V*(fg)) \, h \Big) \Big\|_{L^2_x} 
&\lesssim \Big( \| f \|_{\cC_x^{-\frac{1}{2}-\delta}} \| g \|_{H_x^{1+4\delta}} +  \| f \|_{H_x^{1+4\delta}} \| g \|_{\cC_x^{-\frac{1}{2}-\delta}} \Big) \| h \|_{\cC_x^{-\frac{1}{2}-\delta}}   \label{appendix:eq_trilinear_IV}. 
\end{align}
\end{lemma}

These estimates are essentially an easier version of the fractional product formula. They can be proven using a paraproduct decomposition and Hölder's inequality and we omit the details. We always included \( \delta\)-loss on the right-hand side of \eqref{appendix:eq_trilinear_I}, so we can avoid all summability or endpoint issues. We also never rely on the smoothing effect of the interaction potential \( V \). 

\subsubsection*{The integral operator and truncations}

We now record two properties related to the integral operator \( I_t \) and the associated frequency truncations \( \rho \) and \( \sigma \). 
\begin{lemma}[{\cite[Lemma 2]{BG18}}]\label{appendix:lemma_It}
For any space-time function \( u\colon [0,\infty)\times \T^3 \rightarrow \R \) and any \( \delta>0\), it holds that 
\begin{equation}\label{appendix:eq_It_1}
\sup_{T,t\geq 0} \| \It[u] \|_{H^1_x(\T^3)} \lesssim \| u\|_{L_t^2 L_x^2([0,\infty)\times \T^3)} 
\end{equation}
and 
\begin{equation}\label{appendix:eq_It_2}
\sup_{T,t,s\geq 0} \| \Is[u] - \It[u] \|_{H^{1-\delta}_x(\T^3)}^2 \lesssim \min(s,t)^{-2\delta}   \min(1,|t-s|) \| u\|_{L_t^2 L_x^2([0,\infty)\times \T^3)}^2. 
\end{equation}
\end{lemma}
\begin{proof}
The first estimate \eqref{appendix:eq_It_1} follows directly from \cite[Lemma 2]{BG18}. Since \( \Is[u] - \It[u] \) is supported on frequencies \( \gtrsim \min(s,t) \), we have that 
\begin{equation*}
 \| \Is[u] - \It[u] \|_{H^{1-\delta}_x(\T^3)} \lesssim \min(t,s)^{-\delta}  \| \Is[u] - \It[u] \|_{H^{1}_x(\T^3)}. 
\end{equation*}
The rest of the statement then again follows from \cite[Lemma 2]{BG18}. 
\end{proof}
The result in \cite{BG18} is only stated for \( I_t \) instead of \( \It \), but the same argument applies. 

\begin{lemma}[Well-behaved truncations]\label{appendix:lemma_well_behaved_truncations}
If \( S \geq 1 \) and \( n_1,n_2,n_3 \in \Z^3 \) satisfy \( \| n_j - S e_j \|_2 \leq S/20 \) for all \( j=1,2,3\), where \( e_j \) is the \( j\)-th canonical basis vector, then
\begin{equation}
\rho_S(n_{123})  \Big( \prod_{j=1}^3 \rho_S(n_j) \Big) \int_0^\infty \sigma_s(n_{123})^2 \Big( \prod_{j=1}^3 \rho_s(n_j)^2 \Big) \ds \gtrsim 1. 
\end{equation}
\end{lemma}
While the proof is a bit technical and depends on the precise regions in the definition of \( \rho \), this lemma should not be taken too seriously. 
\begin{proof}
We recall the lower bound \( \min(\rho(y),-\rho^\prime(y)) \gtrsim 1 \) for all \( 1/2 \leq y \leq 2 \) from the definition of \( \rho \). 
From the assumptions, we directly obtain that 
\begin{equation*}
|\| n_{123}\|_2 - \sqrt{3} S | \leq \frac{3}{20} S. 
\end{equation*}
In particular, we obtain that \( 3/2 \cdot S \leq \| n_{123}\|_2 \leq 19/20 \cdot S \). Since \( 19/20 \cdot S \leq  \|n_j \|_2 \leq 21/20 \cdot S \) for all \( j=1,2,3 \), it follows that 
\begin{equation*}
\rho_S(n_{123})  \Big( \prod_{j=1}^3 \rho_S(n_j) \Big) \gtrsim 1. 
\end{equation*}
We estimate the integral by
\begin{align*}
&\int_0^\infty \sigma_s(n_{123})^2 \Big( \prod_{j=1}^3 \rho_s(n_j)^2 \Big) \ds \\
&\gtrsim \int_0^\infty \langle s \rangle^{-1}  1\Big\{ \frac{\langle s \rangle}{2} \leq \| n_{123}\| \leq 2 \langle s \rangle \Big\} \Big( \prod_{j=1}^3 1\Big\{ \| n_j \|_2 \leq 2 \langle s \rangle\Big\}  \Big) \ds \\
&\gtrsim  S^{-1} \Big( \int_0^\infty 1\Big\{ \frac{1}{2} \max( \|n_1 \|,\|n_2\|,\|n_3\|,\|n_{123}\|) \leq s \leq 2 \|n_{123}\| \Big\} \ds  - 2 \Big) \\
&= S^{-1} \Big( \frac{3}{2} \|n_{123}\|_2 - 2 \Big) \\
&\gtrsim 1,
\end{align*}
where we used that \( S \geq 1 \). 
\end{proof}

\subsubsection*{A basic counting estimate}

The following estimate has been used to control stochastic objects (see Lemma \ref{measure_re:lemma_stochastic_objects_III}). 

\begin{lemma}\label{appendix:lemma_sum_estimates}
Let \( v,w \in \Z^3 \) and let \( \alpha,\beta>0 \) satisfy \( 1 < \alpha +\beta < 3 \). Then, 
\begin{equation}\label{appendix:eq_sum_estimates}
\sum_{n \in \Z^3} \frac{1}{\langle n + v \rangle^{\alpha} \langle n + w \rangle^{\beta} \langle n \rangle^2} \lesssim \min( \langle v \rangle, \langle w \rangle)^{1-\alpha-\beta}. 
\end{equation}
\end{lemma}
\begin{remark}
The estimate \eqref{appendix:eq_sum_estimates} is not sharp if \( v \) and \( w \) have different magnitudes. For our purposes, however, \eqref{appendix:eq_sum_estimates} will be sufficient. 
\end{remark} 

\begin{proof}[Proof of Lemma \ref{appendix:lemma_sum_estimates}:] 
Using Young's inequality, we have that 
\begin{equation}
 \frac{1}{\langle n + v \rangle^{\alpha} \langle n + w \rangle^{\beta}} \lesssim \frac{1}{\langle n + v \rangle^{\alpha+\beta}} + \frac{1}{\langle n + w \rangle^{\alpha+\beta}}. 
\end{equation}
Using this inequality, the estimate \eqref{appendix:eq_sum_estimates} reduces to 
\begin{equation*}
\sum_{n \in \Z^3} \frac{1}{\langle n + v \rangle^{\alpha+\beta} \langle n \rangle^2} \lesssim \langle v \rangle^{1-\alpha-\beta}. 
\end{equation*}
This can  easily be proven by decomposing the sum into the regions \( |n|\ll |v|\), \(|n|\sim |v|\) and \( |n|\gg |v|\). 
\end{proof}

\section{Uniqueness of weak subsequential limits}\label{section:uniquess}

In this section, we sketch the proof of the uniqueness of weak subsequential limits of $(\mu_T)_{T\geq1}$, which has been obtained in \cite[Proposition 6.6]{OOT20}. For the convenience of the reader, we present the argument from \cite{OOT20} in our notation. 

\begin{proposition}\label{appendix:prop_uniqueness}
The limit 
\begin{equation}\label{appendix:eq_uniqueness}
\lim_{T\rightarrow \infty} \int \mathrm{d}\mu_T(\phi) \exp(-f(\phi)) 
\end{equation}
exists for all Lipschitz functions $f\colon  \cC_x^{-1/2-\kappa}(\T^3)\rightarrow \bR$. In particular, weak subsequential limits of $(\mu_T)_{T\geq 1}$ are unique. 
\end{proposition}

\begin{remark}
The only reason why Proposition \ref{appendix:prop_uniqueness} does not (immediately) yield the weak convergence of $(\mu_T)_{T\geq1}$ is that we do not prove that the limit in   \eqref{appendix:eq_uniqueness} corresponds to the Laplace transform of a limiting measure. As described in the proof of Theorem \ref{theorem:tightness}, this part follows from Prokhorov's theorem. 
\end{remark}

As was observed in \cite{OOT20}, Proposition \ref{appendix:prop_uniqueness} follows essentially from the same estimates as in the proof of uniform bounds on the variational problem (Proposition \ref{measure_var:proposition_main}). 

\begin{proof}
We recall from \eqref{sc:eq_def_muT} that 
\begin{equation}
\dmu_T(\phi) =  \frac{1}{\cZTl} \exp \Big( - \lcol \cVT(\rho_T(\nabla)\phi)\rcol \Big) \, \mathrm{d}\big( (W_\infty)_\# \bP \big)(\phi). 
\end{equation}
We now split the proof into two steps. \\

\emph{Step 1: Reduction.} Let $k\geq 1$ be a large integer. In this step, we reduce the existence of the limit in \eqref{appendix:eq_uniqueness} to the existence of the limit 
\begin{equation}\label{appendix:eq_regularized_limit}
\lim_{T\rightarrow \infty} \E_{\bP}\Big[ \exp\Big(-f(W_\infty) - \lcol \cVT(\Winf{})\rcol - \epsilon \| W_\infty \|_{\cC_x^{-1/2-\kappa}}^k \Big)\Big]
\end{equation}
for all Lipschitz functions  $f\colon  \cC_x^{-1/2-\kappa}(\T^3)\rightarrow \bR$ and all $\epsilon>0$. To this end, we first note that 
\begin{equation*}
1- \exp(-\epsilon x^k) \leq \epsilon x^k \leq k!  \epsilon \exp(x)
\end{equation*}
for all $x \geq 0$. Using Proposition \ref{measure_var:proposition}, this implies
\begin{align*}
&\Big|   \E_{\bP}\Big[ \exp\Big(-f(W_\infty) - \lcol \cVT(\Winf{})\rcol \Big)\Big] 
-  \E_{\bP}\Big[ \exp\Big(-f(W_\infty) - \lcol \cVT(\Winf{})\rcol - \epsilon \| W_\infty \|_{\cC_x^{-1/2-\kappa}}^k \Big)\Big] \Big|  \\
&=  \E_{\bP}\Big[ \Big( 1- \exp\Big( - \epsilon \| W_\infty \|_{\cC_x^{-1/2-\kappa}}^k \Big)\Big)  \ \exp\Big(-f(W_\infty) - \lcol \cVT(\rho_T(\nabla)W_\infty)\rcol \Big) \Big] \\ 
&\lesssim_k  \epsilon \cdot \E_{\bP}\Big[ \exp\Big(\| W_\infty \|_{\cC_x^{-1/2-\epsilon}}-f(W_\infty) - \lcol \cVT(\rho_T(\nabla)W_\infty)\rcol \Big)\Big]  \\
&\lesssim_{k,\lambda,f} \epsilon. 
\end{align*}
Thus, the existence of the limit in \eqref{appendix:eq_regularized_limit} implies the existence of the limit 
\begin{equation}\label{appendix:eq_P_limit}
\lim_{T\rightarrow \infty} \E_{\bP}\Big[ \exp\Big(-f(W_\infty) - \lcol \cVT(\Winf{})\rcol \Big)\Big]
\end{equation}
for all Lipschitz functions  $f\colon  \cC_x^{-1/2-\kappa}(\T^3)\rightarrow \bR$. By setting $f\equiv 0$, we see that \eqref{appendix:eq_P_limit} implies the convergence of the normalization constants $\cZTl$ as $T\rightarrow \infty$. Since \eqref{appendix:eq_uniqueness} and \eqref{appendix:eq_P_limit} only differ by a factor of $\cZTl$, we obtain that the limit in \eqref{appendix:eq_uniqueness} exists. \\
\emph{Step 2: Existence of the regularized limit \eqref{appendix:eq_regularized_limit}.}  Using the Boué-Dupuis formula (Theorem \ref{thm:bd_formula}) and arguing as in the derivation of \eqref{measure_var:eq_identity}, we have that 
\begin{equation}\label{appendix:eq_variational}
\begin{aligned}
 &- \log \Big( \E_{\bP}\Big[ \exp\Big(-f(W_\infty) - \lcol \cVT(\Winf{})\rcol - \epsilon \| W_\infty \|_{\cC_x^{-1/2-\kappa}}^k \Big)\Big] \Big) \\
 &= \inf_{ w\in \bH} \E_{\bP} \Big[ \mathcal{E}_1^{\scriptscriptstyle{T}}[w] +  \mathcal{E}_2^{\scriptscriptstyle{T}}[w] +  \mathcal{E}_3^{\scriptscriptstyle{T}}[w]  
 + \frac{\lambda}{4} \cV(\IT_\infty[w]) + \frac{1}{2} \| w\|_{L_{t,x}^2}^2 + \epsilon \| I_\infty[w] + W_\infty \|_{\cC_x^{-1/2-\kappa}}^k \Big].  
\end{aligned} 
\end{equation}
Here, $\cV$ is as in \eqref{measure_var:eq_cV} and  $ \mathcal{E}_1^{\scriptscriptstyle{T}}[w],  \mathcal{E}_2^{\scriptscriptstyle{T}}[w]$, and 
 $\mathcal{E}_3^{\scriptscriptstyle{T}}[w]$ are as in \eqref{measure_var:eq_E1}-\eqref{measure_var:eq_E3}, but the term $\varphi(W+I[u])$ in \eqref{measure_var:eq_E1} is replaced by 
 \begin{equation*}
 f(W_\infty + I_\infty[ \Jt \lcol ( V \ast (\Wt)^2  ) \Wt \rcol ] + I_\infty[w] ). 
 \end{equation*}
 In contrast to \eqref{measure_var:eq_E1}-\eqref{measure_var:eq_E3}, we also reflect the dependence on $T$ and $w$ in our notation. To avoid confusion, we also recall that the term $\E_{\bP} \big[\mathcal{E}_0 +\cT\big]$ in \eqref{measure_var:eq_identity} vanishes due to our choice of $\cT$. Our estimates in the proof of Proposition \ref{measure_var:proposition} show that the infimum in \eqref{appendix:eq_variational} can be take over $w\in \bH$ satisfying the additional bound
 \begin{equation*}
 \E_{\bP}\Big[  \frac{\lambda}{4} \cV(\IT_\infty[w]) + \frac{1}{2} \| w\|_{L_{t,x}^2}^2 + \epsilon \| I_\infty[w]  \|_{\cC_x^{-1/2-\kappa}}^k \Big] \lesssim_{\lambda} 1. 
 \end{equation*} 
 In order to conclude the existence of the limit \eqref{appendix:eq_regularized_limit}, it therefore suffices to prove for all $T,S\geq 1$ the estimate
 \begin{equation}\label{appendix:eq_variational_estimate}
 \begin{aligned}
 &\sum_{j=1}^3 \Big| \E_{\bP} \Big[ \mathcal{E}_j^{\scriptscriptstyle{T}}[w] - \mathcal{E}_j^{\scriptscriptstyle{S}}[w] \Big] \Big| 
 + \Big|  \E_{\bP} \Big[  \cV(\IT_\infty[w]) - \cV(\IS_\infty[w]) \Big] \Big| \\
 &\lesssim_{\lambda,\epsilon,k} \min(S,T)^{-\eta} \bigg( 1+  \E_{\bP} \Big[   \frac{1}{2} \| w\|_{L_{t,x}^2}^2 + \epsilon \| I_\infty[w]  \|_{\cC_x^{-1/2-\kappa}}^k \Big] \bigg),
 \end{aligned}
 \end{equation}
 where $\eta>0$ is sufficiently small.  We only present the estimate \eqref{appendix:eq_variational_estimate} for
  $\mathcal{E}_1^{\scriptscriptstyle{T}}[w] - \mathcal{E}_1^{\scriptscriptstyle{S}}[w]$ and $\cV(\IT_\infty[w]) - \cV(\IS_\infty[w])$, since the remaing estimates are similar. \\
  
  \emph{Step 2.a: Estimate of  $\mathcal{E}_1^{\scriptscriptstyle{T}}[w] - \mathcal{E}_1^{\scriptscriptstyle{S}}[w]$.} 
  For the convenience of the reader, we recall that 
  \begin{equation}\label{appendix:eq_E1}
  \begin{aligned}
  \mathcal{E}^{\scriptscriptstyle{T}}_1[w]  &\defe  f(W_\infty + I_\infty[ \Jt \lcol ( V \ast (\Wt)^2  ) \Wt \rcol ] + I_\infty[w] )   -\lambda^2 \int_{\T^3} (V* \lcol (\Winf)^2\rcol ) \bWinf{[3]} \Iinf[w] \dx \\ 
&\hspace{3ex}- 2 \lambda^2 \int_{\T^3} \Big( \big( V* (\Winf \bWinf{[3]})\big) \Winf - \Minf  \bWinf{[3]} \Big) \Iinf[w] \dx  . 
\end{aligned}
\end{equation}
We estimate the contributions of the three terms separately. For the first summand in \eqref{appendix:eq_E1}, we have that 
\begin{align*}
&\big| f(W_\infty + I_\infty[ \Jt \lcol ( V \ast (\Wt)^2  ) \Wt \rcol ] + I_\infty[w] ) - f(W_\infty + I_\infty[ \JS_t \lcol ( V \ast (\WS{}{t})^2  ) \WS{}{t} \rcol ] + I_\infty[w] ) \big| \\
&\leq \operatorname{Lip}(f) \big\| I_\infty[ \Jt \lcol ( V \ast (\Wt)^2  ) \Wt \rcol ]  - I_\infty[ \JS_t \lcol ( V \ast (\WS{}{t})^2  ) \WS{}{t} \rcol ]  \big\|_{\cC_x^{-1/2-\kappa}}. 
\end{align*} 
The desired estimate then follows from a minor modification of \eqref{measure_re:eq_so_III_1}. \\
We now turn to the second summand in \eqref{appendix:eq_E1}. First, we note for any $\gamma>0$ that 
\begin{equation}\label{appendix:eq_I_gain}
\| \IT_\infty[w] - \IS_\infty[w] \|_{H^{1-\gamma}_x} = \| (\rho_T(\nabla) - \rho_S(\nabla)) I_\infty[w] \|_{H^{1-\gamma}} \lesssim \min(S,T)^{-\gamma} \| I_\infty[w]\|_{H^1_x}. 
\end{equation}
Now, we let $0<\gamma<\gamma^\prime<\min(1/2,\beta)$. Using \eqref{appendix:eq_I_gain}, we obtain that 
\begin{align*}
&\Big|  \int_{\T^3} (V* \lcol (\Winf)^2\rcol ) \bWinf{[3]} \Iinf[w] \dx - \int_{\T^3} (V* \lcol (\WS{}{\infty})^2\rcol ) \mathbb{W}_\infty^{\scriptscriptstyle{S},[3]} \IS_\infty[w] \dx \Big| \allowdisplaybreaks[4]\\
&\lesssim \Big\|  (V* \lcol (\Winf)^2\rcol ) \bWinf{[3]} -(V* \lcol (\WS{}{\infty})^2\rcol ) \mathbb{W}_\infty^{\scriptscriptstyle{S},[3]} \Big\|_{\cC_x^{-1+\gamma^\prime}}
\| \Iinf[w] \|_{H_x^{1-\gamma}} \\
&+  \Big\| (V* \lcol (\WS{}{\infty})^2\rcol ) \mathbb{W}_\infty^{\scriptscriptstyle{S},[3]} \Big\|_{\cC_x^{-1+\gamma^\prime}} \| \IT_\infty[w] - \IS_\infty[w] \|_{H_x^{1-\gamma}} \allowdisplaybreaks[4]\\
&\lesssim \bigg( \Big\|  (V* \lcol (\Winf)^2\rcol ) \bWinf{[3]} -(V* \lcol (\WS{}{\infty})^2\rcol ) \mathbb{W}_\infty^{\scriptscriptstyle{S},[3]} \Big\|_{\cC_x^{-1+\gamma^\prime}} 
+ \min(S,T)^{-\gamma} \Big\| (V* \lcol (\WS{}{\infty})^2\rcol ) \mathbb{W}_\infty^{\scriptscriptstyle{S},[3]} \Big\|_{\cC_x^{-1+\gamma^\prime}} \bigg) \\
&\times  \| I_\infty[w] \|_{H_x^1}. 
\end{align*}
The desired estimate then follows from a minor modification of \eqref{measure_re:eq_so_III_1} and Lemma \ref{appendix:lemma_It}. The estimate of the third term in \eqref{appendix:eq_E1} is similar and we omit the details. \\

\emph{Step 2.b: Estimate of $\cV(\IT_\infty[w]) - \cV(\IS_\infty[w])$.} Using Hölder's inequality and interpolation, it holds that 
\begin{equation}\label{appendix:eq_L4}
\| \varphi \|_{L_x^4(\T^3)} \lesssim \| \varphi \|_{\cC_x^{-1/2-\kappa}(\T^3)}^{\frac{2(1-\kappa)}{3+4\kappa}}  \| \varphi \|_{H_x^{1-\kappa}(\T^3)}^{\frac{1+6\kappa}{3+4\kappa}}. 
\end{equation}
Using Hölder's inequality, \eqref{appendix:eq_L4}, and $\IT_t= \rho_T(\nabla) I_t$, we obtain that 
\begin{align*}
&\Big|\cV(\IT_\infty[w]) - \cV(\IS_\infty[w])\Big| \\
&\lesssim \big( \| \IT_\infty[w] \|_{L^4_x} +  \| \IS_\infty[w] \|_{L^4_x}\big)^3 \| \IT_\infty[w]- \IS_\infty[w] \|_{L^4_x} \\
&\lesssim \| I_\infty[w] \|_{\cC_x^{-1/2-\kappa}}^{4\cdot \frac{2(1-\kappa)}{3+4\kappa}} \| I_\infty[w] \|_{H_x^1}^{3 \cdot \frac{1+6\kappa}{3+4\kappa}}
 \| \IT_\infty[w] - \IS_\infty[w]  \|_{H_x^{1-\kappa}}^{\frac{1+6\kappa}{3+4\kappa}} \\
 &\lesssim \min(S,T)^{-\frac{(1+6\kappa)\kappa}{3+4\kappa}} \| I_\infty[w] \|_{\cC_x^{-1/2-\kappa}}^{4\cdot \frac{2(1-\kappa)}{3+4\kappa}} \| I_\infty[w] \|_{H_x^1}^{4 \cdot \frac{1+6\kappa}{3+4\kappa}}. 
\end{align*}
Since $4 (1+6\kappa)/(3+4\kappa)<2$, the desired estimate follows from Lemma \ref{appendix:lemma_It} and Young's inequality.  
\end{proof}
\end{appendix}

\begin{remark}
As seen in the proof of Proposition \ref{appendix:prop_uniqueness}, the regularizing factor $\exp\big(- \epsilon \| W_\infty \|_{\cC_x^{-1/2-\kappa}}^k\big)$ in \eqref{appendix:eq_regularized_limit} is needed to estimate
$\cV(\IT_\infty[w]) - \cV(\IS_\infty[w])$. 
\end{remark}

\bibliography{construction_library}
\bibliographystyle{myalpha}
\end{document}